\documentclass[a4,12pt]{amsart}
\usepackage{amssymb,amstext,amscd,amsthm,amsfonts,mathdots,mathrsfs}
\usepackage{hyperref}
\usepackage{pdflscape}
\usepackage{setspace}
\usepackage[all]{xy}
\usepackage{enumerate}
\usepackage{indentfirst}
\usepackage{ytableau}
\usepackage{color}
\usepackage{colortbl}
\usepackage{amsmath}
\usepackage{tikz}
\usetikzlibrary{shapes.geometric, arrows}
\usetikzlibrary{calc}
\numberwithin{equation}{section}
\def\Z{\mathbb{Z}}
\def\N{\mathbb{N}}
\def\R{\mathbb{R}}
\def\k{\Bbbk}
\def\P{\mathcal{P}}
\def\pcl{P^+_{cl,k}}
\def\ssum{\textstyle\sum\limits}
\def\Hom{\operatorname{Hom}}

\def\max{\operatorname{max}}
\def\ev{\operatorname{ev}}
\def\deg{\operatorname{deg}}
\def\res{\operatorname{res}}
\def\ch{\operatorname{char}}

\newtheorem{Theorem}{Theorem}[section] 
\newtheorem{theorem*}{Theorem}
\newtheorem{Lemma}[Theorem]{Lemma}
\newtheorem{Cor}[Theorem]{Corollary}
\newtheorem*{corollary}{COROLLARY}
\newtheorem{Prop}[Theorem]{Proposition}
\newtheorem{example}[Theorem]{Example}
\newtheorem{Defn}[Theorem]{Definition}

\newtheorem{rem}[Theorem]{Remark}
\theoremstyle{definition}

\newtheorem{claim}[Theorem]{Claim}

\newtheorem*{maintheorem}{MAIN THEOREM}

\begin{document}
\setlength{\baselineskip}{17pt}

\title[Representation type of cyclotomic quiver Hecke algebras of type $A_\ell^{(1)}$]
{Representation type of cyclotomic quiver Hecke algebras of type $A_\ell^{(1)}$}

\author{Susumu Ariki}
\address{Department of Pure and Applied Mathematics, Graduate School of Information Science and Technology, Osaka University, 1-5 Yamadaoka, Suita, Osaka, 565-0871, Japan.}
\email{ariki@ist.osaka-u.ac.jp}

\author{Linliang Song}
\address{School of Mathematical Science, Tongji University, Shanghai, 200092, China.}
\email{llsong@tongji.edu.cn}

\author{Qi Wang}
\address{Yau Mathematical Sciences Center, Tsinghua University, Beijing, 100084, China.}
\email{infinite-wang@outlook.com (cc:infinite-wang@tsinghua.edu.cn)}

\date{\today}

\thanks{2020 {\em Mathematics Subject Classification.} 
20C08, 16G60, 17B65, 16G20.
}

\keywords{
Dominant maximal weights, 
Ariki-Koike algebras, 
KLR algebras, 
Brauer graph algebras, 
representation type, 
derived equivalence.
}

\begin{abstract}
We first investigate a connected quiver consisting of all dominant maximal weights for an integrable highest weight module in affine type A. 
This quiver provides an efficient method to obtain all dominant maximal weights.
Then, we completely determine the representation type of cyclotomic Khovanov-Lauda-Rouquier algebras of arbitrary level in affine type A, by using the quiver we construct.
This result gives a complete classification for the representation type of blocks of cyclotomic Hecke algebras since cyclotomic KLR algebras of type $A^{(1)}_\ell$ form a one-parameter family and cyclotomic Hecke algebras occur at a special parameter, i.e., $t=-2$ if $\ell=1$ and $t=(-1)^{\ell+1}$ if $\ell\geq2$.
\end{abstract}

\maketitle

\tableofcontents

\section{Introduction}
Cyclotomic Hecke algebras (\cite{AK-algebra, BM-Hecke-alg}) are generalizations of Iwahori-Hecke algebras of type $A$ and $B$, since they can be thought as Hecke algebras of complex reflection groups of type $G(k,1,n)$. 
This class of algebras has been actively studied in the past several decades by various authors, such as Brundan-Kleshchev \cite{BK-graded-decom-number, BK-block}, Dipper-James-Mathas \cite{DJM-cyclotomic-q-schur-alg}, Fayers \cite{F-ariki-koike-alg}, Lyle-Mathas \cite{LM-cyclotomic-hecke}, to name a few.
Nowadays, cyclotomic Hecke algebras are considered as examples of cyclotomic quiver Hecke (or KLR) algebras (introduced by Khovanov-Lauda in the same paper \cite{KL-diagram1} as they introduced quiver Hecke algebras, we may also mention Rouquier \cite{Ro-2kac} as another paper which introduced the latter affine version) whose Lie type is affine type $A$.
Research on cyclotomic quiver Hecke algebras in other Lie types has just begun in recent years.

When we study a given algebra, basic representation theoretic information on the algebra are the classification of irreducible modules, representation types, etc. 
Given a cyclotomic Hecke algebra, it was already found in \cite{Ar-decom-number} that the irreducible modules are labeled by a Kashiwara crystal and the labeling by its Misra-Miwa realization (\cite{MM-crystal}) coincides with the labeling arising from the cellular algebra theory (\cite{Ar-simple-mod}). 
However, it took time to study the representation type.
The purpose here is to give a complete classification of representation types for the blocks of cyclotomic Hecke algebras in a slightly larger class, that is, the class of cyclotomic quiver Hecke algebras in affine type $A$.

Let $\Lambda\in \pcl$ (see \eqref{equ:levelkdominant}) be a level $k$ dominant integral weight of type $A^{(1)}_\ell$.
We denote by $R^\Lambda(\beta)$ the cyclotomic quiver Hecke algebra associated with $\beta\in Q_+$, where $Q_+$ is the positive cone of the root lattice.
Let $P(\Lambda)$ (see \eqref{equ::weight-set}) be the weight system of the integrable highest weight module $V(\Lambda)$ associated with $\Lambda$. 
We will show in \eqref{equ::max-set} that $R^\Lambda(\beta)\neq0$ if and only if the intersection of the Weyl group orbit through $\Lambda-\beta\in P(\Lambda)$ and $\Lambda-\P^{\Lambda}-\Z_{\geq 0}\delta\subseteq P(\Lambda)$ is not empty, where $\Lambda-\P^{\Lambda}=\max^+(\Lambda)$ is the set of dominant maximal weights of $V(\Lambda)$.
Since the representation type of $R^\Lambda(\beta)$ depends only on the Weyl group orbit (see Subsection 4.1), we may assume $\beta\in O(\Lambda):=\P^\Lambda+\Z_{\geq0}\delta$, without loss of generality.
Then, the main result of this article is as follows. Here, the definitions of $\mathscr F(\Lambda)$ and $\mathscr T(\Lambda)$ are given in \eqref{equ::def-finite-tame}, and our terminology \emph{tame} means tame and of infinite representation type following \cite{Er-tame-block}.

\begin{maintheorem}
Suppose $\Lambda\in \pcl$ with $k\geq 3$. 
For any $\beta\in O(\Lambda)$, $R^{\Lambda}(\beta)$ is
\begin{enumerate}
\item of finite representation type if $\beta\in \mathscr F(\Lambda)$,
\item of tame representation type if one of the following holds:
\begin{itemize}
    \item $\beta=\delta$, $\Lambda=k\Lambda_i$, $\ell=1$ with $t\neq \pm 2$,
    \item $\beta=\delta$, $\Lambda=k\Lambda_i$, $\ell\geq 2$ with $t\neq (-1)^{\ell+1}$,
    \item $\beta\in \mathscr T(\Lambda)$.
\end{itemize}
\end{enumerate}
Otherwise, it is of wild representation type.
\end{maintheorem}

We point out that the representation type of $R^\Lambda(\beta)$ with $k=1,2$ has been determined by Ariki-Iijima-Park in \cite{AIP-rep-type-A-level-1} and by Ariki in \cite{Ar-rep-type}, respectively. 
Combining these with our new achievement in this article, the representation type of $R^\Lambda(\beta)$ for arbitrary level $k$ is now completely determined. 

One of the reasons why it took time was that one did not have an explicit description of $\max^+(\Lambda)$, i.e., $\Lambda-\P^{\Lambda}$.
Little information was known about $\max^+(\Lambda)$ until recently, but Kim, Oh and Oh introduced in \cite{KOO-sieving-phenomenon} the sieving equivalence relation $\sim$ on $\pcl$ and then obtained a bijection $\phi_\Lambda:  \max^+(\Lambda)\rightarrow\pcl(\Lambda)$, where $\pcl(\Lambda)$ is the set of equivalence classes of $\Lambda$ under $\sim$. 
We construct the inverse of $\phi_\Lambda$ to describe $\max^+(\Lambda)$ through $\pcl(\Lambda)$.  
This is Theorem \ref{theo::bijection-phi-inversion}.

Since there are infinitely many $R^\Lambda(\beta)$ with $\beta\in \P^\Lambda+\Z_{\geq0}\delta$, we need reduction to a finite number of cases.
In a series of articles by the first author and his collaborators, they introduced a scheme for $k=1,2$ to reduce the general cases to $\beta=\delta$ or some $\beta\in \P^\Lambda$. 
To apply this strategy for $k\ge 3$, we need to explore more structures of $\max^+(\Lambda)$. 
For that purpose, we define a connected quiver $\vec C(\Lambda)$ (Definition \ref{def::arrow}) whose vertex set is $\pcl(\Lambda)$, such that an arrow $\Lambda'\rightarrow \Lambda''\in \vec C(\Lambda)$ encodes an arrow $\beta_{\Lambda'}\rightarrow \beta_{\Lambda''}\in \P^\Lambda$ via the inverse of $\phi_\Lambda$ mentioned above. 
Then, we show that $R^\Lambda(\beta_{\Lambda''})$ is of infinite (resp. wild) representation type if $R^\Lambda(\beta_{\Lambda'})$ is of infinite (resp. wild) representation type when there exists a directed path from $\Lambda'$ to $\Lambda''$ in $\vec C(\Lambda)$. 
In this way, the strategy above is successfully applied to $k\ge 3$. 
Finally, the proof is completed after a few direct calculations for some subquivers of $\vec C(\Lambda)$ with $k=3,4,5,6$. 
This is explained in Subsection 3.2 and Subsection 4.1.

A well-known feature of cyclotomic quiver Hecke algebras in affine type $A$ is that $\{R^\Lambda(\beta)\mid\beta\in O(\Lambda)\}$ provides complete representatives of derived equivalence classes of $R^\Lambda(\beta)$ with $\beta\in Q_+$.
After finding the representation type of $R^\Lambda(\beta)$ with $\beta\in O(\Lambda)$, one may want to specify the Morita equivalence classes of representation-finite and tame $R^\Lambda(\beta)$'s with $\beta\in Q_+$ as was done by the first author in \cite{Ar-rep-type} and \cite{Ar-tame-block} for $k=1,2$. 
Here, we give an answer to the cases for $k\ge 3$.

\begin{corollary}
Suppose $\Lambda\in \pcl$, $k\geq 3$ and $\beta\in Q_+$.
\begin{enumerate}
    \item If $R^\Lambda(\beta)$ is of finite representation type, then $R^\Lambda(\beta)$ is Morita equivalent to either $\k[X]/(X^m)$ for some $m\geq 1$ or a Brauer tree algebra without exceptional vertex.
    \item If $R^\Lambda(\beta)$ is of tame representation type, then $R^\Lambda(\beta)$ is Morita equivalent to one of the local algebras (1)-(4) listed in Subsection 8.2, or a Brauer graph algebra whose Brauer graph is a tree with 3 vertices and all multiplicities are 2, or a Brauer graph algebra whose Brauer graph is a tree and all multiplicities are $k$, or a Brauer graph algebra whose Brauer graph is a tree, one multiplicity is $1$ and all others are $m$ for $2\le m\le k-1$. 
\end{enumerate}
In all of the cases, the Brauer graph is a straight line if $R^\Lambda(\beta)$ is a cellular algebra.
\end{corollary}

In order to prove the above corollary, recent progress on Brauer graph algebras by Antipov-Zvonareva \cite{AZ-brauer-graph} and Opper-Zvonareva \cite{OZ-brauer-graph} play a crucial role. 
We refer to Subsection 2.7 for more details.

We mention here the relation between $R^\Lambda(\beta)$ and the blocks of cyclotomic Hecke algebra $H_n^\Lambda$.
Let $H^\Lambda(\beta)$ be the block algebra of $H_n^\Lambda$ indexed by $\beta\in Q_+$ with height $n$.
Then, it is proved by Brundan and Kleshchev in \cite{BK-block} that $H^\Lambda(\beta)\simeq R^\Lambda(\beta)$ when the parameter $t=-2$ if $\ell=1$ and $t=(-1)^{\ell+1}$ if $\ell\ge2$, in the definition of $R^\Lambda(\beta)$. Applying our main result above to this special parameter, one gets the representation type of $H^\Lambda(\beta)$. Since $H_n^\Lambda$ is known as a cellular algebra (see \cite{GL-cellular-alg}), its quiver has a symmetric shape (see, e.g., \cite[Proposition 2.8]{X-cellular-alg}), and hence, provides a nice behavior on the Morita equivalence classes of representation-finite and tame $H^\Lambda(\beta)$'s (e.g., the Brauer graph must be a straight line). Lastly, the decomposition matrices in some representation-finite and tame cases are obtained. 
This generalizes the results in \cite{AIP-rep-type-A-level-1}, \cite{Ar-rep-type} and \cite{Ar-tame-block}.  
We refer to Section 8 for details.

During the preparation of this article, we notice that Li and Qi posted a paper \cite{LQ-rep-type-cyclotomic-hecke-alg} on arXiv a short time ago, in which they characterize the representation-finite $H^\Lambda(\beta)$'s under the assumption $\ch \k\neq 2$, and classify the Morita equivalence classes of such $H^\Lambda(\beta)$'s. 
This is covered by our results above.

It is worth mentioning that the representation type of cyclotomic quiver Hecke algebras of level $1$ in types $A^{(2)}_{2\ell}$, $D^{(2)}_{\ell+1}$ and $C^{(1)}_\ell$ is determined in \cite{AP-rep-type-A_2-level-1}, \cite{AP-rep-type-D-level-1} and \cite{AP-rep-type-C-level-1}. 
We expect that a similar approach can be established in these types at higher levels.

The article is organized as follows. 
In Section 2, we recall the necessary backgrounds, such as Cartan datum, maximal dominant weights, cyclotomic (quiver) Hecke algebras, and Brauer graph algebras.
In Section 3, we first give an explicit description of $\max^+(\Lambda)$, and then construct the connected quiver $\vec C(\Lambda)$ as we mentioned above.
In section 4, we explain the main strategy to prove our main theorem, as well as give some useful reduction lemmas. 
Then, the proof is divided into 3 steps proved in Sections 5, 6, and 7, in which the contents are basically direct calculations.
In the last section, we will classify the Morita equivalence classes of representation-finite and tame blocks for cyclotomic (quiver) Hecke algebras. Also, some decomposition matrices are determined here.

\subsection*{Conventions}
Throughout, we set $\N:=\{1,2,\ldots\}$ and $\N_{0}:=\Z_{\geq0}=\{0,1,2,\ldots\}$.
Fix $\ell\in \N$.
We set $e:=\ell+1$ and call it the quantum characteristic.
For any $m,m'\in \Z$, we write $m \equiv_e m'$ if $e$ divides $m-m'$, and $m\not\equiv_e m'$ otherwise.

\section{Preliminaries}
In this section, we review some basic materials, including Cartan datum, maximal dominant weights, cyclotomic (quiver) Hecke algebras and Brauer graph algebras, etc.

\subsection{Cartan datum of type $A^{(1)}_{\ell}$}
Let $I=\{0,1,...,\ell\}$ be an index set.
Fix an \emph{affine Cartan datum} $(A,P,\Pi,P^{\vee},\Pi^{\vee})$ of type $A^{(1)}_{\ell}$, consisting of
\begin{enumerate}
\item a matrix $A=(a_{ij})_{i,j \in I}$ of corank $1$, which is called the \emph{affine Cartan matrix} of type $A^{(1)}_\ell$. 
If $\ell=1$, then
$$
A=\left(\begin{array}{cc}
2 & -2 \\
-2 & 2
\end{array}\right).
$$
If $\ell\geq 2$, then
$$
a_{ij}=\left\{\begin{array}{ll}
2  & \hbox{if } i=j, \\
-1 & \hbox{if } i-j \equiv_e \pm 1, \\
0  & \hbox{otherwise}.
\end{array}\right.
$$

\item a free abelian group $P=\bigoplus_{i=0}^{\ell} \Z \Lambda_i \oplus \Z \delta$, which is called the \emph{weight lattice},

\item a free abelian group $P^{\vee}=\Hom(P,\Z)$, which is called the \emph{coweight lattice},

\item a linearly independent set $\Pi=\{\alpha_i \mid i\in I \}\subset P$, which is called the set of \emph{simple roots}, and a linearly independent set $\Pi^{\vee}=\{ h_i \mid i\in I \} \subset P^{\vee}$, which is called the set of \emph{simple coroots}.
The elements in $\Pi$ and $\Pi^{\vee}$ satisfy
$$
\langle h_i, \alpha_j \rangle = a_{ij}
\quad \text{and} \quad
\langle h_i, \Lambda_j \rangle = \delta_{ij}
$$
for all $i,j\in I$.
\end{enumerate}
The null root $\delta$ and the canonical central element $c$ are
$$
\delta=\ssum_{i\in I}\alpha_i
\quad \text{and} \quad
c=\ssum_{i\in I} h_i.
$$
In particular, $\langle h_i, \delta \rangle=0$ for any $i\in I$.

The set of dominant integral weights is defined as $P^+:=\{\Lambda \in P\mid \langle h_i,\Lambda \rangle\in \Z_{\geq0}, i\in I\}$.
For any $k\in \N$, a weight $\Lambda \in P$ is said to be \emph{of level $k$} if $\langle c,\Lambda\rangle=k$.
Then, we define
$$
P^+_k:=\left \{\Lambda\in P^+ \mid \langle c,\Lambda \rangle =k \right \},
$$
the set of all level $k$ dominant integral weights.

The root lattice $Q$ and its positive cone $Q_+$ are defined as
$$
Q:=\ssum_{i\in I}\Z \alpha_i
\quad \text{and} \quad
Q_+:= \ssum_{i\in I} \Z_{\geq0} \alpha_i,
$$
respectively.
For any $\beta=\sum_{i\in I}k_i\alpha_i\in Q_+$, we define the \emph{height} of $\beta$ by $|\beta|=\sum_{i\in I}k_i$.

\subsection{Integrable highest weight modules}
Let $\mathfrak g$ be the affine Kac-Moody algebra associated with the affine Cartan datum $(A, P,\Pi, P^\vee, \Pi^\vee)$ and $U_v(\mathfrak g)$ its quantum group.
For any $\Lambda \in P^+$, $V(\Lambda)$ denotes the integrable highest weight $U_v(\mathfrak g)$-module with the highest weight $\Lambda$ and $P(\Lambda)$ denotes the set of weights of $V(\Lambda)$.

We call a weight $\mu \in P(\Lambda)$ \emph{maximal} if $\mu+\delta\notin P(\Lambda)$.
Let $\max(\Lambda)$ be the set of maximal weights in $P(\Lambda)$.
Then, we define
$$
\max^+(\Lambda):=\max(\Lambda)\cap P^+
$$
as the set of all \emph{dominant maximal weights} of $V(\Lambda)$.

Let $W$ be the Weyl group, which is generated by $\{r_i\}_{i\in I}$ with the defining relations
$$
r_i^2=1,\quad
r_ir_j=r_jr_i \ \text{if $i-j\not\equiv_e \pm 1$}, \quad
r_ir_jr_i=r_jr_ir_j \ \text{if $i-j\equiv_e \pm 1$}.
$$
Namely, $W$ is the affine symmetric group.
Since $V(\Lambda)$ is integrable, $W$ acts on $P(\Lambda)$ by
$r_i \mu =\mu -\langle h_i, \mu \rangle\alpha_i$, for any $r_i\in W$ and $\mu\in P(\Lambda)$.
It is known that $\max^+(\Lambda)$ is a finite set (e.g., \cite[Proposition 12.6]{K-Lie-alg})
and any maximal weight in $\max(\Lambda)$ is $W$-conjugate to a dominant maximal weight in $\max^+(\Lambda)$.
Moreover,
\begin{equation}\label{equ::weight-set}
P(\Lambda)= \bigsqcup_{\mu\in \max(\Lambda)}
\left \{\mu-m\delta\mid m\in \Z_{\geq 0}\right \}.
\end{equation}

\subsection{Sets in bijection with $\max^+(\Lambda)$}
Set $\mathfrak h:=\R\otimes_{\Z}P^\vee$.
There is a non-degenerate symmetric bilinear form $(\cdot,\cdot)$ on $\mathfrak h^*=\R\otimes_\Z P$ (see \cite[(6.2.2)]{K-Lie-alg}) such that for any $i,j\in I$,
\begin{equation}\label{equ::bilinear-form}
(\alpha_i,\alpha_j)=a_{ij}, \quad
(\alpha_i,\Lambda_0)=\delta_{i0}, \quad
(\delta, \Lambda_0)=1, \quad
(\delta, \delta)=(\Lambda_0,\Lambda_0)=0.
\end{equation}
Let $\mathfrak h_0$ be the subspace of $\mathfrak h$ spanned by $\{h_i\mid i\in I_0\}$, where $I_0:=I\setminus \{0\}$.
The orthogonal projection (see \cite[(6.2.7)]{K-Lie-alg}) $\bar{\ \ } : \mathfrak h^*=\mathfrak h^*_0\oplus \R\delta \oplus \R\Lambda_0\rightarrow \mathfrak h^*_0$ is given by
$$
\bar \mu= \mu -(\delta, \mu) \Lambda_0-(\Lambda_0,\mu)\delta
$$
for $\mu\in P$. 
Then, $\mathfrak h^*_0$ has basis $\{\bar \Lambda_i\mid i\in I_0\}$.
It is given in \cite[p.89]{K-Lie-alg} that $W$ can be realized as a group of affine transformations on $\mathfrak h^*_0$ and the fundamental alcove is defined as
$$
\mathcal{C}_{af}:=\left \{\mu\in \mathfrak h^*_0 \mid (\alpha_i,\mu)\geq 0 \text{ for } i\in I_0, (\mu,\delta-\alpha_0)\leq 1 \right \}.
$$

\begin{Prop}[{\cite[Propositions 12.5 and 12.6]{K-Lie-alg}}]\label{prop::bijection-kac}
For any $\Lambda\in P^+_k$, there is a bijection from $\max^+(\Lambda)$ to $k\mathcal{C}_{af}\cap (\bar\Lambda+\bar Q)$ given by $\mu\mapsto \bar\mu$, where
$$
k\mathcal{C}_{af}:=\left \{\mu\in \mathfrak h^*_0 \mid k^{-1}\mu\in \mathcal{C}_{af} \right \}
=\left \{\mu\in \mathfrak h^*_0\mid (\alpha_i, \mu)\geq 0 \text{ for } i\in I_0, (\mu, \delta-\alpha_0)\leq k \right \}.
$$
Moreover, $\max^+(\Lambda)=\{\mu\in P^+\mid \mu\leq \Lambda \text{ and } \Lambda-\mu-\delta\notin Q_+\}$.
\end{Prop}

Since $\bar \delta=0$, we have $k\mathcal{C}_{af}\cap (\overline{\Lambda+m\delta}+\bar Q)=k\mathcal C_{af}\cap (\bar\Lambda+\bar Q)$.
Then, the authors of \cite{KOO-sieving-phenomenon} introduced the set 
\begin{equation}\label{equ:levelkdominant}
\pcl:=P^+_k/\Z\delta
\end{equation}
which we shall identify with $P^+_k\cap \sum_{i\in I}\Z_{\geq0}\Lambda_i$. 
We explore more details as follows.
Since $\bar\Lambda_0=0$ and $\bar \alpha_0=-\sum_{i\in I_0}\bar \alpha_i$, we have
$$
\bar P=\ssum_{i\in I_0}\Z \bar \Lambda_i
\quad \text{and} \quad
\bar Q=\ssum_{i\in I_0}\Z \bar \alpha_i.
$$
For any $\Lambda\in \pcl$, the authors of \cite{KOO-sieving-phenomenon} defined an injective map
$$
\begin{aligned}
\iota_\Lambda: k\mathcal{C}_{af}\cap (\bar\Lambda+\bar Q)
&\rightarrow \pcl=P^+_k\cap \textstyle \sum_{i\in I}\Z_{\geq0}\Lambda_i\\
\ssum_{i\in I_0}m_i\bar \Lambda_i
&\mapsto m_0\Lambda_0+\ssum_{i\in I_0}m_i\Lambda_i,
\end{aligned}
$$
where $m_0=k-\sum_{i\in I_0}m_i$.
Furthermore, an equivalence relation $\sim$ on $\pcl$ is introduced in \cite[(2.4)]{KOO-sieving-phenomenon} as
$$
\Lambda \sim \Lambda' 
\quad \text{if and only if} \quad 
k\mathcal{C}_{af}\cap (\bar\Lambda+\bar Q)=k\mathcal{C}_{af}\cap (\overline{\Lambda'}+\bar Q),
$$
which they call the sieving equivalence relation.
It is proved in \cite[Lemma 2.4]{KOO-sieving-phenomenon} that a complete set of pairwise inequivalent representatives of $\pcl/\sim$ is given by
\begin{equation}\label{equ::representatives}
{\rm DR}(\pcl):= \{(k-1)\Lambda_0+\Lambda_s\mid s\in I\}.
\end{equation}

Let $\pcl(\Lambda)$ be the equivalence class of $\Lambda$, which is a subset of $\pcl$.
Since $\Lambda \sim \Lambda'$ if and only if $\Lambda'\in\text{Im}\ \iota_\Lambda$ (see \cite[Lemma 2.3]{KOO-sieving-phenomenon}), the above map $\iota_\Lambda$ induces a bijective map
\begin{equation}\label{equ::bijection-iota}
\iota_\Lambda: k\mathcal{C}_{af}\cap (\bar\Lambda+\bar Q)\rightarrow \pcl(\Lambda)
\end{equation}
with the inverse given by $\bar{\ \ }$.

\begin{Theorem}[{\cite[Theorem 2.14]{KOO-sieving-phenomenon}}]\label{theo::equivalence-class}
For any $\Lambda\in \pcl$, we have
$$
\pcl(\Lambda)=\left \{\Lambda'\in \pcl\mid \ev(\Lambda) \equiv_e \ev(\Lambda')\ \right \},
$$
where $\ev(\Lambda)=\sum_{i\in I_0} \langle h_i,\Lambda\rangle i$, as in \cite[Convention 2.13]{KOO-sieving-phenomenon}.
\end{Theorem}

\begin{example}\label{ex::P-level-2}
Fix $s\in I$. If $k=1$, then
$P^+_{cl,1}(\Lambda_s)=\{\Lambda_s\}$ and hence,
$\max^+(\Lambda_s)=\{\Lambda_s\}$.
If $k=2$, then
$P^+_{cl,2}(\Lambda_0+\Lambda_s)=\{\Lambda_i+\Lambda_j\mid i+j \equiv_e s\}$,
so that
$$
P^+_{cl,2}(2\Lambda_0)=\left \{2\Lambda_0,\Lambda_i+\Lambda_{e-i}\mid 1\leq i\leq \frac{e}{2}\right \}
$$
if $s=0$, and for any $s>0$,
$$
P^+_{cl,2}(\Lambda_0+\Lambda_s)=\left \{\Lambda_i+\Lambda_{s-i},
\Lambda_{s+j}+\Lambda_{e-j}\mid 0\leq i\leq \frac{s}{2}, 1\leq j\leq \frac{e-s}{2}\right \}.
$$
\end{example}

\subsection{Cyclotomic Hecke algebras}
Let $\k$ be an algebraically closed field and $q\in \k^\times$ an invertible element.
The quantum characteristic $e$ of the element $q$ is defined as the smallest positive integer such that $1+q+\ldots+ q^{e-1}=0$.
When $e=\ell+1$, the module categories over cyclotomic Hecke algebras of rank $n\in \Z_{\geq0}$ altogether categorify the integrable module $V(\Lambda)$ over the affine Kac-Moody algebra $\mathfrak g$ of type $A^{(1)}_\ell$, see \cite{Ar-decom-number}.

Fix $\Lambda\in \pcl$ and write $\Lambda=\Lambda_{i_1}+\Lambda_{i_2}+\ldots+\Lambda_{i_k}$ for some $i_1, i_2,\ldots,i_k\in I$.

\begin{Defn}
If $q\neq 1$, we denote by $H^\Lambda_n$ the cyclotomic Hecke algebra associated with a complex reflection group of type $G(k,1,n)$ (see \cite{AK-algebra, BM-Hecke-alg}).
It is the associative algebra generated by $T_0,T_1,\ldots, T_{n-1}$ subject to
\begin{enumerate}
\item $(T_i-q)(T_i+1)=0, T_iT_j=T_jT_i\ (j\ge i+2), T_iT_{i+1}T_i=T_{i+1}T_iT_{i+1}\ (1\le i\le n-2)$,

\item $(T_0T_1)^2=(T_1T_0)^2, T_0T_j=T_jT_0\ (j\geq 2)$,

\item $\prod_{j=1}^k(T_0-q^{i_j})=0$.
\end{enumerate}
In fact, $H_n^\Lambda$ is isomorphic to the quotient algebra of the (extended) affine Hecke algebra that is generated by $T_1,\ldots, T_{n-1}$, $X_1^{\pm1},\ldots, X^{\pm1}_n$ with relations (1) and
\begin{enumerate}
\item[(2')] $X^{\pm1}_r X_s^{\pm 1}=X_s^{\pm 1}X^{\pm1}_r, X_rX_r^{-1}=1$,

\item[(3')] $T_rX_rT_r=qX_{r+1}, T_iX_j=X_jT_i$ if $j\neq i,i+1$,
\end{enumerate}
modulo the two-sided ideal generated by $\prod_{j=1}^k(X_1-q^{i_j})$.

If $q=1$, we define $H^\Lambda_n$ as the quotient algebra of the degenerate affine Hecke algebra that is generated by $s_1,\ldots,s_{n-1}$, $x_1,\ldots, x_n$ with relations
\begin{enumerate}
\item $s_i^2=1, s_is_j=s_js_i\ (j\geq i+2), s_is_{i+1}s_i=s_{i+1}s_is_{i+1}\ (1\le i\le n-2)$,

\item $x_r x_s=x_sx_r$,

\item $s_rx_{r+1}=x_{r}s_r+1$, $s_ix_j=x_js_i$ if $j\neq i,i+1$,
\end{enumerate}
modulo the two-sided ideal generated by $\prod_{j=1}^k(x_1-i_j)$.
\end{Defn}

\subsection{Cyclotomic quiver Hecke algebras}
Fix $t\in\k$ if $\ell=1$ and $t\in\k^\times$ if $\ell \geq 2$.
For any $i,j\in I$, we take polynomials $Q_{i,j}(u,v)\in\k[u,v]$ such that $Q_{i,i}(u,v)=0$, $Q_{i,j}(u,v) =Q_{j,i}(v,u)$ and if $\ell\geq 2$,
$$
\begin{aligned}
Q_{i,i+1}(u,v)   &=u+v \text{ if } 0\leq i<\ell,\\
Q_{\ell, 0}(u,v) &= u+t v,\\
Q_{i,j}(u,v)     &=1  \text{ if } j\not\equiv_e  i, i\pm1.
\end{aligned}
$$
If $\ell=1$, we take $Q_{0,1}(u,v)= u^2+t uv+v^2$.
Let $\mathfrak S_n$ be the symmetric group generated by elementary transpositions $\{s_i\mid 1\leq i\leq n-1\}$.
Then, $\mathfrak S_n$ may act naturally on $I^n$ by place permutations.

\begin{Defn}\label{def::cyclotomic-quiver}
Let $\Lambda\in \pcl$.
The cyclotomic quiver Hecke algebra $R^{\Lambda}(n)$ associated with the dominant integral weight $\Lambda$ and the polynomials $(Q_{i,j}(u,v))_{i,j\in I}$ is the $\Z$-graded $\k$-algebra generated by
$$
\{ e(\nu)\mid \nu=(\nu_1, \nu_2, \ldots, \nu_n)\in I^n \}, \quad
\{x_i \mid 1\leq i \leq n \}, \quad
\{\psi_j \mid 1\leq j\leq n-1\},
$$
subject to the following relations:
\begin{enumerate}
\item $e(\nu)e(\nu')=\delta_{\nu, \nu'}e(\nu),\ \textstyle\sum_{\nu\in I^n}e(\nu)=1,\ x_ix_j=x_jx_i,\  x_ie(\nu)=e(\nu)x_i$.

\item $\psi_i e(\nu)=e(s_i(\nu))\psi_i,\ \psi_{i}\psi_j=\psi_j\psi_i$ if $|i-j|>1$.

\item $\psi_i^2 e(\nu)=Q_{\nu_i,\nu_{i+1}}(x_i,x_{i+1})e(\nu)$.

\item
$$
(\psi_ix_j-x_{s_i(j)}\psi_i)e(\nu)=\left\{
\begin{array}{ll}
-e(\nu)  & \text{ if } j=i \text{ and } \nu_i=\nu_{i+1}, \\
e(\nu)   & \text{ if } j=i+1 \text{ and } \nu_i=\nu_{i+1}, \\
 0       & \text{ otherwise.}
\end{array}\right.
$$

\item
$$
(\psi_{i+1}\psi_i\psi_{i+1}-\psi_i\psi_{i+1}\psi_{i})e(\nu)
=\left\{
\begin{array}{ll}
\frac{Q_{\nu_i,\nu_{i+1}}(x_i,x_{i+1})-Q_{\nu_i,\nu_{i+1}}(x_{i+2}, x_{i+1})}
{x_i-x_{i+2}}e(\nu) & \text{ if } \nu_i=\nu_{i+2}, \\
0 & \text{ otherwise.}
\end{array}\right.
$$

\item $x_1^{\langle h_{\nu_1}, \Lambda \rangle}e(\nu)=0$.
\end{enumerate}
The $\Z$-grading on $R^\Lambda(n)$ is defined by
$$
\deg(e(\nu))=0,\quad
\deg(x_ie(\nu))=2,\quad
\deg(\psi_ie(\nu))=-a_{\nu_i,\nu_{i+1}}.
$$
\end{Defn}

For each $\beta\in Q_+$ with $|\beta|=n$, we may define a central idempotent of $R^\Lambda(n)$ by
$$
e(\beta):=\ssum_{\nu\in I^\beta}e(\nu)\ \text{with}\  I^\beta=\left \{\nu=(\nu_1, \nu_2, \ldots, \nu_n)\in I^n\mid \ssum_{i=1}^n \alpha_{\nu_i}=\beta \right \}.
$$
Then, the cyclotomic quiver Hecke algebra associated with $\beta$ is defined as
$$
R^{\Lambda}(\beta):=R^\Lambda(n)e(\beta).
$$
It is proved in \cite{BK-block} that $R^\Lambda(n)\simeq H^\Lambda_n$, when $t=-2$ if $\ell=1$ and $t=(-1)^{\ell+1}$ if $\ell\geq2$.
According to \cite{LM-cyclotomic-hecke}, $R^{\Lambda}(\beta)$ is a block algebra of $R^\Lambda(n)$ if such an isomorphism exists.
Thus, we will identify the block algebras of $H^\Lambda_n$ with $R^{\Lambda}(\beta)$, for the aforementioned choice of $t$.

Let $\sigma:\Z/ e\Z \xrightarrow{\sim}\Z/ e\Z$ be an automorphism given by $\sigma(i)=i+1$.
We define
\begin{equation}\label{def::sigma}
\sigma\beta=\ssum_{i\in I}k_i\alpha_{\sigma(i)}
\quad \text{and} \quad
\sigma\Lambda=\ssum_{i\in I}m_i\Lambda_{\sigma(i)}
\end{equation}
for any $\beta=\sum_{i\in I}k_i\alpha_i\in Q_+$ and $\Lambda=\sum_{i\in I}m_i\Lambda_i\in \pcl$.
The following statement proved by the first author in \cite[Proposition 3.2]{Ar-rep-type} will be used frequently.
\begin{Prop}\label{prop::iso-sigma}
There is an algebra isomorphism between $R^{\Lambda}(\beta)$ and $R^{\sigma \Lambda}(\sigma \beta)$.
\end{Prop}

We mention that $R^{\Lambda}(\beta)$ and $R^{\sigma \Lambda}(\sigma \beta)$ share the same polynomials $Q_{i,j}(u,v)\in\k[u,v]$.

\subsection{Graded dimension of $R^\Lambda(\beta)$}
We recall some well-known notions.
A partition $\lambda=(\lambda_1, \lambda_2, \ldots )$ of $n$ is a non-increasing sequence of non-negative integers such that $|\lambda|:=\sum_{i=1} \lambda_i=n$.
A $k$-multipartition of $n$ is an ordered $k$-tuple of partitions $\lambda=(\lambda^{(1)},\lambda^{(2)},\ldots,\lambda^{(k)})$ such that $|\lambda|:= \sum_{s=1}^k|\lambda^{(s)}|=n$.
Let $\P_{k,n}$ be the set of all $k$-multipartitions of $n$.

Each partition may be identified with its Young diagram. 
We follow the English convention in this paper.
The Young diagram of a $k$-multipartition $\lambda$ is visualized as a column vector whose entries are Young diagrams of the components $\lambda^{(s)}$ of $\lambda$, i.e., the Young diagram of $\lambda^{(s)}$ is above $\lambda^{(t)}$ if $s<t$.
Then, a node of $\lambda$ is said to be removable (resp. addable) if we obtain a new $k$-multipartition after removing (resp. adding) the node from (resp. to) $\lambda$.

Suppose $\Lambda \in \pcl$ and fix an expression $\Lambda=\Lambda_{i_1}+\Lambda_{i_2}+\ldots +\Lambda_{i_k}$.
Let $p$ be a node of $\lambda=(\lambda^{(1)},\lambda^{(2)},\ldots,\lambda^{(k)})\in \P_{k,n}$.
If $p$ is in the $a$-th row and the $b$-th column of $\lambda^{(s)}$, the residue $\res p\in I$ is defined by
$$
\res p \equiv_e i_s+b-a.
$$
If $\res p=\omega$, then we call the node $p$ an $\omega$-node.
For any removable $\omega$-node $p$ of $\lambda$, we denote by $d_p(\lambda)$ the number of addable $\omega$-nodes of $\lambda$ below $p$ (in the Young diagram of $\lambda$) minus the number of removable $\omega$-nodes of $\lambda$ below $p$.

For any $\lambda\in \P_{k,n}$, a standard tableau $T=(T^{(1)}, T^{(2)},\ldots, T^{(k)})$ of $\lambda$ is obtained from the Young diagram of $\lambda$ by inserting the integers $1, 2, \ldots, n$ into the nodes without repeats, such that the entries of each $T^{(s)}$ are strictly increasing along the rows from left to right and down the columns from top to bottom.
Let $\text{Std}(\lambda)$ be the set of all standard tableaux of $\lambda$.
If the integer $i$ is inserted in the node $p$ of $T \in \text{Std}(\lambda)$, we take the residue of $p$ as the residue of $i$, i.e., $\omega_i:=\res p$, and the residue sequence of $T$ is defined by $\omega_T:=(\omega_1,\omega_2,\ldots,\omega_n)\in I^n$. 
Moreover, the degree of $T$ (see \cite[(4.33)]{BK-graded-decom-number}) is defined inductively by
\begin{equation}
\deg(T):=\left\{\begin{array}{ll}
\deg(T\downarrow_n)+d_p(\lambda) & \text{ if } n>0, \\
0 & \text{ if } n=0,
\end{array}\right.
\end{equation}
where $p$ is the node filled by $n$ if $n>0$ and $T\downarrow_n$ is the tableau obtained from $T$ by removing $p$.
In fact, the process of inserting $1,2,\ldots, n$ into $\lambda$ yields an increasing sequence $\chi$ of multipartitions:
$$
(\emptyset,\ldots,\emptyset)=\chi_0\rightarrow \chi_1\rightarrow\chi_2\rightarrow
\ldots \rightarrow \chi_{n-1}\rightarrow \chi_n=\lambda,
$$
where $\chi_s$ is obtained from $\chi_{s-1}$ by adding the node $\chi_s/\chi_{s-1}$ for $1\leq s\leq n$. 
Then, the degree of $T$ is given by
$$
\deg(T)=\ssum\limits_{s=1}^{n}d_{\chi_s/\chi_{s-1}}(\chi_s).
$$

\begin{Theorem}[{\cite[Theorem 4.20]{BK-graded-decom-number}}]\label{theo::graded-dim}
For any $\beta\in Q_+$ with $|\beta|=n$ and $\nu,\nu'\in I^\beta$, the graded dimension of $e(\nu)R^\Lambda(\beta)e(\nu')$ is
$$
\dim_q e(\nu)R^\Lambda(\beta)e(\nu')=
\sum_{
\substack{\lambda\in \P_{k,n}, S,T\in\emph{Std}(\lambda),\\
\omega_S=\nu,\omega_T=\nu'}
}
q^{\deg(S)+\deg(T)}.$$
\end{Theorem}

\subsection{Brauer graph algebras}
We call an undirected graph simply a graph and a directed graph a quiver. 
If we write $(V,E)$ for a graph, then $V$ is the set of vertices and $E$ is the set of undirected edges.
Similarly, if we write $(V,E)$ for a quiver, then $V$ is the set of vertices and $E$ is the set of directed edges.

We will see that certain block algebras of cyclotomic quiver Hecke algebras are Brauer graph algebras.
We recall the definition of Brauer graph algebras following \cite{Schroll-Brauer-graph}.

\begin{Defn}[{\cite[Definition 2.1]{Schroll-Brauer-graph}}]
A Brauer graph $\Gamma=(V,E,\mathfrak{m},\mathfrak{o})$ is a finite connected graph $(V,E)$ equipped with a multiplicity function $\mathfrak m:V \rightarrow \N$, and a cyclic ordering $\mathfrak{o}$ of the edges around each vertex.
If $v\in V$ is a vertex incident to multiple edges $E_{v,1},\ldots,E_{v,c_v}\in E$ which we read counterclockwise around $v$, then the cyclic ordering around $v$ is given by $E_{v,1}<\ldots<E_{v,c_v}<E_{v,1}$.

In particular, $\Gamma$ is called a Brauer tree if $(V,E)$ is a tree and $\mathfrak m(v)=1$ for all but at most one $v\in V$, and such a vertex is called the exceptional vertex of $\Gamma$.
\end{Defn}

\begin{Defn}[{\cite[Section 2.4]{Schroll-Brauer-graph}}]
Given a Brauer graph $\Gamma=(V,E,\mathfrak m,\mathfrak{o})$, we define a quiver $Q_\Gamma=(Q_0,Q_1)$ with
$Q_0=E$ and $Q_1=\bigsqcup_{v\in V}Q_{1,v}$ where
$$
Q_{1,v}=\{\alpha_{v,i}: E_{v,i}\rightarrow E_{v,i+1}\mid 1\le i \le c_v, c_v+1:=1 \}
$$
for the cyclic ordering $E_{v,1}<\ldots <E_{v,c_v}<E_{v,1}$ around $v$.
We also define an admissible ideal $\mathcal{I}_\Gamma$ of the path algebra $\k Q_\Gamma$, which is generated by the following elements:
\begin{itemize}
\item $(\alpha_{v,j}\alpha_{v,j+1}\cdots \alpha_{v,c_v}\alpha_{v,1}\cdots \alpha_{v,j-1})^{\mathfrak m(v)}\alpha_{v,j}$.

\item $(\alpha_{v,j}\cdots \alpha_{v,j-1})^{\mathfrak m(v)}-(\alpha_{u,i}\cdots \alpha_{u,i-1})^{\mathfrak m(u)}$, whenever $E_{u,i}=E_{v,j}\in E$ for $u\neq v\in V$.

\item $\alpha_{u,i}\alpha_{v,j}$ when $E_{u,i+1}=E_{v,j}\in E$ for $u\neq v\in V$.
\end{itemize}
Then, $\mathcal{A}_\Gamma:=\k Q_\Gamma/\mathcal{I}_\Gamma$ is called the Brauer graph algebra associated with $\Gamma$.
\end{Defn}

The representation theory of Brauer graph algebras has been studied well in the past decades. 
We recall two latest progress as follows.

\begin{Theorem}[{\cite[Corollary 1.3]{AZ-brauer-graph}}]\label{Theorem:brauer-graph-derived-closed}
Let $\mathcal{A}$ be a Brauer graph algebra. 
If $\mathcal{B}$ is derived equivalent to $\mathcal{A}$, then $\mathcal{B}$ is Morita equivalent to a Brauer graph algebra.  
\end{Theorem}

\begin{Theorem}[{\cite[Theorem A]{OZ-brauer-graph}}]\label{thm-brauer-graph-condition}
Let $\mathcal{A}$ and $\mathcal{B}$ be non-local Brauer graph algebras associated with Brauer graphs $\Gamma_{\mathcal{A}}$ and $\Gamma_{\mathcal{B}}$, respectively.
Then, $\mathcal{A}$ is derived equivalent to $\mathcal{B}$ if and only if the following conditions hold.
\begin{enumerate}
\item $\Gamma_\mathcal{A}$ and $\Gamma_{\mathcal{B}}$ share the same number of vertices, edges, and faces,
\item the multisets of multiplicities and the multisets of perimeters of faces of $\Gamma_\mathcal{A}$ and $\Gamma_{\mathcal{B}}$ coincide,
\item either both or none of $\Gamma_\mathcal{A}$ and $\Gamma_{\mathcal{B}}$ are bipartite.
\end{enumerate}
\end{Theorem}

In Theorem \ref{thm-brauer-graph-condition}, we may embed $\Gamma$ into the interior of the so-called ribbon surface $\Sigma$ such that every connected component of $\Sigma \setminus \Gamma$ is bounded by a unique boundary component in $\partial \Sigma$ and a (possibly self-glued) polygon consisting of edges in $\Gamma$. 
Then, the face of $\Gamma$ and its perimeter are encoded in the polygon. 
Also, $\Gamma$ is called bipartite if every cycle in $\Gamma$ has an even length. 
We refer to \cite{OZ-brauer-graph} for the details.

It is worth mentioning that if $\Gamma$ is a tree, then $\Gamma$ is always bipartite and admits the unique face whose perimeter is twice the number of edges in $\Gamma$.
We have the following corollary by Theorem \ref{thm-brauer-graph-condition}, which will be used in Section 8.

\begin{Cor}\label{cor::tree-to-tree}
Let $\mathcal{A}$ be a Brauer graph algebra with $\Gamma_\mathcal{A}$ being a tree. 
If $\mathcal{B}$ is derived equivalent to $\mathcal{A}$, then $\Gamma_{\mathcal{B}}$ is a tree with the same number of vertices as $\Gamma_{\mathcal{A}}$. 
If $\mathcal{B}$ is moreover a cellular algebra, then $\Gamma_{\mathcal{B}}$ is a straight line.
\end{Cor}
\begin{proof}
Since $\Gamma_\mathcal{A}$ is a tree, the number of vertices is greater than the number of edges by one. 
Theorem \ref{thm-brauer-graph-condition} implies that $\Gamma_{\mathcal{B}}$ is a tree.
Suppose that $\mathcal{B}$ is cellular. 
We show that the valency of each vertex is at most 2. 
If not so, then there exists a vertex $v$ such that there are at least $3$ edges $E_1<E_2<E_3$ connecting to $v$.
By the definition, the quiver of $\mathcal{B}$ has an arrow $\alpha: E_1\rightarrow E_2$ but without an arrow from $E_2$ to $E_1$. 
On the other hand, the quiver of a cellular algebra has a symmetric shape (see \cite[Proposition 2.8]{X-cellular-alg}), that is, the number of arrows from vertex $a$ to vertex $b$ must equal the number of arrows from $b$ to $a$. 
Therefore, we obtain a contradiction.
\end{proof}

\section{Dominant maximal weights}
In this section, we first give an explicit description of dominant maximal weights for any integrable highest weight module of type $A^{(1)}_{\ell}$, then construct a connected graph and a quiver whose vertex set is the set of all dominant maximal weights, and finally, capture a certain subquiver of the connected quiver which we need for our purpose, in the last subsection.

\subsection{An explicit description of $\max^+(\Lambda)$}
We give an explicit description of $\max^+(\Lambda)$ in terms of the elements in
$\pcl(\Lambda)=\{\Lambda'\in \pcl\mid \Lambda'\sim \Lambda\}
\subset
\pcl=P^+_k\cap\sum_{i\in I}\Z_{\geq0}\Lambda_i$.

\begin{Defn}
For any $\Lambda\in \pcl$, let $\phi_\Lambda:= \iota_\Lambda \circ \bar {\ }: \max^+(\Lambda)\rightarrow \pcl(\Lambda)$ be the composition of the bijection in Proposition \ref{prop::bijection-kac} and the bijection in \eqref{equ::bijection-iota},
that is,
$$
\begin{aligned}
\mu \mapsto \bar\mu 
&=\mu-(\delta,\mu)\Lambda_0-( \Lambda_0,\mu)\delta=\textstyle\sum_{i\in I_0}m_i\bar\Lambda_i\\
&\mapsto \phi_\Lambda(\mu)= m_0\Lambda_0+\textstyle\sum_{i\in I_0}m_i\Lambda_i .
\end{aligned}
$$
\end{Defn}

The following lemma shows that $\langle h_i,\mu\rangle=\langle h_i,\phi_\Lambda(\mu)\rangle$ for any $\mu\in \max^+(\Lambda)$ and $i\in I$.

\begin{Lemma}\label{lem::bijection-phi}
Suppose $\Lambda \in \pcl$.
For any $\mu\in \max^+(\Lambda)$, we have
$$
\phi_\Lambda(\mu)=\ssum_{i\in I} \langle h_i,\mu\rangle \Lambda_i.
$$
In particular, $\phi_\Lambda(\Lambda)=\Lambda$.
\end{Lemma}
\begin{proof}
Write $\mu=\Lambda-\sum_{i\in I}k_i\alpha_i=\sum_{i\in I}m_i\Lambda_i+a\delta$ for some $k_i,m_i\in \Z_{\geq0}$ and $a\in \Z$.
Since $\langle c,\alpha_i\rangle=0$ for any $i\in I$, we have $\langle c,\mu \rangle=\langle c, \Lambda\rangle=k$, so that $\mu\in P^+_k$ and $\sum_{i\in I}m_i=k$.
On the other hand, since $\bar \delta=0$, we have $\bar\mu= \sum_{i\in I_0}m_i\bar \Lambda_i$ and
$$
\phi_\Lambda(\mu)=
\phi_\Lambda\left (\ssum_{i\in I}m_i\Lambda_i+a\delta \right )=
\iota_\Lambda\left (\ssum_{i\in I_0}m_i\bar \Lambda_i \right )=
\ssum_{i\in I}m_i\Lambda_i.
$$
Therefore, $\phi_\Lambda(\mu)=\sum_{i\in I}m_i\Lambda_i=\sum_{i\in I}\langle h_i,\mu\rangle\Lambda_i$.
\end{proof}

For any $\Lambda\in \pcl$ and $\Lambda'\in \pcl(\Lambda)$, let
$$
Y_{\Lambda,\Lambda'}:=(y_0,y_1,\ldots, y_\ell)\in \Z^{e}
$$
be the vector whose entries are given by $y_i:=\langle h_i, \Lambda-\Lambda'\rangle$, for $i\in I$.
If there is no confusion about $\Lambda$, we also write $Y_{\Lambda,\Lambda'}$ as $Y_{\Lambda'}$ to simplify the notation.

Recall that $A=(a_{ij})_{i,j\in I}$ is the affine Cartan matrix of  type $A^{(1)}_\ell$.
For any $X=(x_0,x_1,\ldots,x_\ell)\in\N_0^{e}$, we define $\min X:=\min\{x_i\mid i\in I\}$.

\begin{Lemma}\label{lem::unique-solution}
For any $\Lambda'\in \pcl(\Lambda)$, the system of linear equations
\begin{equation}\label{equ::the-equation}
AX^t=Y_{\Lambda'}^t
\end{equation}
has a unique solution if we require $X=(x_0,x_1,\ldots,x_\ell)\in \N_0^{e}$ and $\min X=0$.
\end{Lemma}
\begin{proof}
Since $\phi_\Lambda^{-1}(\Lambda')\in \max^+(\Lambda)$, we can write $\phi_\Lambda^{-1}(\Lambda')=\Lambda-\sum_{i\in I}k_i\alpha_i$ for some $k_i\in \Z_{\geq0}$.
Moreover, there is some $i$ such that $k_i=0$.
Otherwise, $\phi_\Lambda^{-1}(\Lambda')=\Lambda-\sum_{i\in I}(k_i-1)\alpha_i-\delta$, and $\Lambda-\phi_\Lambda^{-1}(\Lambda')-\delta\in Q_+$, which is a contradiction by Proposition \ref{prop::bijection-kac}.

By Lemma \ref{lem::bijection-phi}, we have $\langle h_j, \phi_\Lambda^{-1}(\Lambda')\rangle =\langle h_j, \Lambda'\rangle$.
Thus,
$\langle h_j, \sum_{i\in I}k_i\alpha_i\rangle=y_j$ for any $j\in I$, i.e., $A(k_0,k_1,\ldots,k_\ell)^t=Y_{\Lambda'}^t$ and $K=(k_0,k_1,\ldots,k_\ell)$ is a solution of \eqref{equ::the-equation} satisfying the required condition. 
We show the uniqueness of such a solution.
Since $A$ has rank $\ell$ and $(1^{e})$ is a solution of $AX^t=0$, any solution of \eqref{equ::the-equation} is of form $K+a(1^{e})$ for some scalar $a$.
If \eqref{equ::the-equation} has another solution $X=(x_0, x_1, \dots, x_\ell)\in \N_0^{e}$, then $X=K+a(1^{e})$ with some positive integer $a$, and this would force $x_i>0$ for all $i\in I$, violating the requirement.
\end{proof}

The proof of Lemma \ref{lem::unique-solution} shows the following.

\begin{Theorem}\label{theo::bijection-phi-inversion}
The inverse map $\phi^{-1}_{\Lambda}:\pcl(\Lambda) \rightarrow \max^+(\Lambda)$ of $\phi_\Lambda$ is given by
$$
\phi^{-1}_{\Lambda}(\Lambda')= \Lambda-\ssum_{i\in I}x_i \alpha_i,
$$
where $(x_0,x_1,\ldots,x_\ell)$ is the unique solution of \eqref{equ::the-equation} given in Lemma \ref{lem::unique-solution}.
\end{Theorem}

\begin{Defn}
For any $\Lambda\in \pcl$ and $\Lambda'\in \pcl(\Lambda)$, we set $\beta_{\Lambda,\Lambda'}:=\sum_{i\in I}x_i\alpha_i$ with
$$
X_{\Lambda,\Lambda'}:=(x_0,x_1,\ldots,x_\ell)
$$
being the unique solution of \eqref{equ::the-equation} in Lemma \ref{lem::unique-solution}.
If there is no confusion about $\Lambda$, we also write the above as $\beta_{\Lambda'}$ and $X_{\Lambda'}$ to simplify the notation.
In particular, if $\Lambda'=\Lambda$, then $Y_\Lambda=(0^e)$, $X_{\Lambda}=(0^{e})$ and $\beta_{\Lambda}=0$.
\end{Defn}

Define $\P^\Lambda:=\{\beta_{\Lambda'} \mid \Lambda'\in \pcl(\Lambda) \}$ for any $\Lambda\in \pcl$. 
Then Theorem \ref{theo::bijection-phi-inversion} implies
\begin{equation}\label{equ::max-set}
\max^+(\Lambda)=\left \{\Lambda-\beta\mid \beta\in \P^\Lambda\right \}.
\end{equation}

\begin{rem}
Theorem \ref{theo::bijection-phi-inversion} gives us an efficient way to obtain all elements of $\max^+(\Lambda)$ for $\Lambda\in P^+$ of arbitrary level $k\in \N$.
\end{rem}

\begin{rem}
Recall the map $\sigma$ defined in \eqref{def::sigma}. 
By Lemma \ref{lem::unique-solution}, we have 
\begin{equation}\label{equ:sigmaX}
 \beta_{\sigma\Lambda, \sigma\Lambda'}=\sigma \beta_{\Lambda,\Lambda'}   
\end{equation} 
for any $\Lambda'\in \pcl (\Lambda)$.
\end{rem}

Recall from Theorem \ref{theo::equivalence-class} that $\ev(\Lambda)=\sum_{i\in I_0}\langle h_i, \Lambda\rangle i$.

\begin{Lemma}\label{lem::embedding-Lambda}
Suppose $\Lambda=\bar\Lambda+\tilde \Lambda$ with $\Lambda\in \pcl$, $\bar\Lambda\in P^+_{cl,k'}$ and $\tilde \Lambda\in P^+_{cl,k-k'}$.
Define $\P^\Lambda$ and $\P^{\bar\Lambda}$ as before.
Then,
$$
P^+_{cl,k'}(\bar\Lambda)+\tilde \Lambda\subset \pcl(\Lambda)
\quad\text{and}\quad
\beta_{\Lambda,\Lambda'+\tilde \Lambda}=\beta_{\bar\Lambda,\Lambda'}
$$ 
for any $\Lambda'\in P^+_{cl,k'}(\bar\Lambda)$.
In particular, $\P^{\bar\Lambda} \subset \P^\Lambda$.
\end{Lemma}
\begin{proof}
We see that $\ev(\Lambda'+\tilde\Lambda)\equiv_e \ev(\Lambda)$ since $\ev(\Lambda') \equiv_e\ev(\bar\Lambda)$ by Theorem \ref{theo::equivalence-class}.
We have $\Lambda'+\tilde\Lambda\in \pcl(\Lambda)$ for any $\Lambda'\in P^+_{cl,k'}(\bar\Lambda)$ by Theorem \ref{theo::equivalence-class} again.
Then, $\beta_{\Lambda,\Lambda'+\tilde \Lambda}$ is definable and $\beta_{\Lambda,\Lambda'+\tilde\Lambda}\in \P^\Lambda$.
Moreover, the equality
$$
\Lambda-(\Lambda'+\tilde \Lambda)=\bar\Lambda-\Lambda'
$$
implies that $Y_{\Lambda,\Lambda'+\tilde \Lambda }=Y_{\bar\Lambda,\Lambda'}$
and hence, $X_{\Lambda,\Lambda'+\tilde \Lambda }=X_{\bar\Lambda,\Lambda'}$ by Lemma \ref{lem::unique-solution}.
It follows that $\beta_{\bar\Lambda, \Lambda'}=\beta_{\Lambda,\Lambda'+\tilde \Lambda}\in \P^{\Lambda}$ for any $\beta_{\bar\Lambda, \Lambda'}\in \P^{\bar\Lambda}$.
\end{proof}

\subsection{A connected graph of $\max^+(\Lambda)$}
In this subsection, we give an inductive method to find all elements of $\max^{+}(\Lambda)$, without solving equation \eqref{equ::the-equation}.
This method will make $\max^+(\Lambda)$ into a connected quiver.

Before proceeding further, we introduce a convention in order to simplify statements: when we will write $\Lambda_i$ or $\alpha_i$ or $x_i$, it always means 
$\Lambda_{i\!\!\mod e}$ or $\alpha_{i\!\!\mod e}$ or $x_{i\!\!\mod e}$.

Fix $\Lambda\in \pcl$.
Since we know $\max^+(\Lambda)=\{\Lambda\}$ when $k=1$, we will assume $k\geq 2$ in the following.
For any $\Lambda'\in \pcl(\Lambda)$ with $k\geq 2$, we can write $\Lambda'=\Lambda_i+\Lambda_j+\tilde\Lambda$ for some $i,j\in I$ and $\tilde\Lambda\in P^+_{cl,k-2}$.
Then, we define
$$
\Lambda'_{i,j}:=\Lambda_{i-1}+\Lambda_{j+1}+\tilde\Lambda.
$$
By the description of $P^+_{cl,2}(\Lambda_i+\Lambda_j)$ similar to Example \ref{ex::P-level-2} and Lemma \ref{lem::embedding-Lambda}, we see that $\Lambda'_{i,j} \in  \pcl(\Lambda)$.
Moreover,
$$
(\Lambda'_{i,j})_{j+1,i-1}=(\Lambda_{i-1}+\Lambda_{j+1}+\tilde\Lambda)_{j+1,i-1}= \Lambda_{i}+\Lambda_{j}+\tilde\Lambda=\Lambda'
$$
and
\begin{equation}\label{equ::edge-loop}
\Lambda'_{i,j}=\Lambda' \quad \text{if and only if}\quad  j\equiv_e i-1.
\end{equation}

\begin{Defn}
Fix $\Lambda\in \pcl$ with $k\geq 2$.
Let $C(\Lambda)$ be an undirected graph whose vertex set is $\pcl(\Lambda)$, and we draw an edge between $\Lambda'$ and $\Lambda''$ if $\Lambda''=\Lambda'_{i,j}$ for some $i,j\in I$ with $j\not\equiv_e i-1$.
The last condition guarantees $\Lambda'\neq \Lambda''$ by \eqref{equ::edge-loop}, i.e., no loop is allowed in $C(\Lambda)$.
\end{Defn}

For example, set $\Lambda=\Lambda_0+\Lambda_3+\Lambda_6 \in P^+_{cl,3}$ with $\ell=6$, $C(\Lambda)$ is displayed as follows.
$$
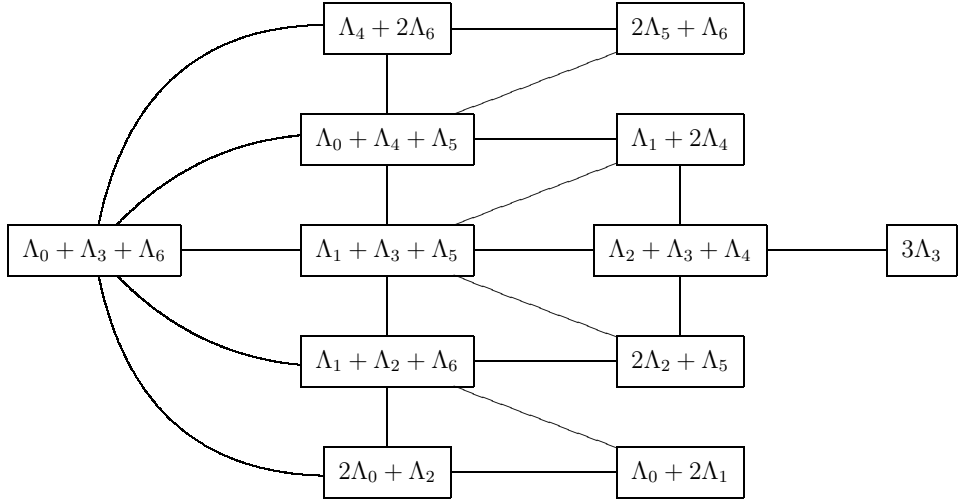
\scalebox{0.8}{
\xymatrix@C=2cm@R=1cm{
&*++[F]{\Lambda_4+2\Lambda_6}\ar@{-}[d]\ar@{-}[r]&*++[F]{2\Lambda_5+\Lambda_6}&\\
&*++[F]{\Lambda_0+\Lambda_4+\Lambda_5}\ar@{-}[ru]\ar@{-}[r]\ar@{-}[d]&*++[F]{\Lambda_1+2\Lambda_4}\ar@{-}[d]&\\
*++[F]{\Lambda_0+\Lambda_3+\Lambda_6}
\ar@{-}[r]\ar@{-}@/^0.75cm/[ru]\ar@{-}@/^1.5cm/[ruu]
\ar@{-}@/_0.75cm/[rd]\ar@{-}@/_1.5cm/[rdd]
&*++[F]{\Lambda_1+\Lambda_3+\Lambda_5}\ar@{-}[r] \ar@{-}[ru] \ar@{-}[rd]\ar@{-}[d]
&*++[F]{\Lambda_2+\Lambda_3+\Lambda_4}\ar@{-}[r]  \ar@{-}[d]
&*++[F]{3\Lambda_3}\\
&*++[F]{\Lambda_1+\Lambda_2+\Lambda_6}\ar@{-}[r] \ar@{-}[rd] \ar@{-}[d] &*++[F]{2\Lambda_2+\Lambda_5}&\\
&*++[F]{2\Lambda_0+\Lambda_2}\ar@{-}[r]  &*++[F]{\Lambda_0+2\Lambda_1}&
}}
$$

We show that $C(\Lambda)$ is connected.

\begin{Lemma}\label{lem::path-edge}
For any $\Lambda'\in \pcl(\Lambda)$, there exists a path from $\Lambda$ to $\Lambda'$ in $C(\Lambda)$. 
In particular, $C(\Lambda)$ is a finite connected graph.
\end{Lemma}
\begin{proof}
According to \eqref{equ::representatives}, we may assume $\Lambda=(k-1)\Lambda_0+\Lambda_s\in \text{DR}(\pcl)$ with $s\in I$.
Obviously, the statement is true when $k=2$ by Example \ref{ex::P-level-2}.

Write $\Lambda'=\sum_{i\in I}m_i\Lambda_i\in \pcl(\Lambda)$. 
Then, $m_0\leq k$.
Set $|\Lambda'|_+:=k-m_0$, i.e., the sum of coefficients of $\Lambda_i$ with $i\neq 0$.
We prove the statement by induction on $|\Lambda'|_+$.
If $|\Lambda'|_+=0$ or 1, then $m_0=k$ or $k-1$, and $\Lambda'=\Lambda$ since $\Lambda'\in \text{DR}(\pcl)$, hence the statement is trivial.

We now assume that $|\Lambda'|_+\geq2$.
Then $\Lambda'=\Lambda_i+\Lambda_j+\tilde\Lambda$ for some $i,j\in I_0$ and $\tilde\Lambda\in P^+_{cl,k-2}$.
Suppose $i+j\equiv_e h$ for some $h\in I$.
Then $\Lambda'':=\Lambda_0+\Lambda_h+\tilde \Lambda\in \pcl(\Lambda)$ and $|\Lambda''|_+<|\Lambda'|_+$.
Using Example \ref{ex::P-level-2}, we see that there is a path between $\Lambda_0+\Lambda_h$ and $\Lambda_i+\Lambda_j$ in $C(\Lambda_0+\Lambda_h)$.
So, there is a path between $\Lambda''$ and $\Lambda'$ in $C(\Lambda)$.
By the induction assumption on $|\Lambda''|_+$, there is a path between $\Lambda$ and $\Lambda''$, which yields a path between $\Lambda$ and $\Lambda'$.

Finally, $C(\Lambda)$ is finite since $\max^+(\Lambda)$ is finite by \cite[Proposition 12.6]{K-Lie-alg}.
\end{proof}

Next, we look at the relation between $\beta_{\Lambda'}$ and $\beta_{\Lambda'_{i,j}}$, which we deduce from the relation between $X_{\Lambda'}$ and $X_{\Lambda'_{i,j}}$.
For any $i,j\in I$, let
$$
[i,j]:= \left\{\begin{array}{ll}
\{i,i+1, \ldots, j\}                 & \text{ if } i\leq j, \\
\{0,1,\ldots, j,i,i+1,\ldots,\ell\} & \text{ if } i>j.
\end{array}\right.
$$
Also, we define $\Delta_{i,j}:=(\delta_0,\delta_1, \ldots,\delta_\ell)\in \N_0^{e}$, where $\delta_h=1$ if $h\in [i,j]$ and $\delta_h=0$ if $h\notin [i,j]$.
More explicitly,
$$
\Delta_{i,j}=\left\{\begin{array}{ll}
 (0^i,1^{j-i+1},0^{\ell-j})       & \text{ if } i\leq j, \\
 (1^{j+1},0^{i-j-1},1^{\ell-i+1}) & \text{ if } i> j.
\end{array}\right.
$$
Note that $\Delta_{i,j}=(1^{e})$ if and only if $j\equiv_e i-1$.

Fix $\Lambda\in \pcl$ with $k\geq 2$.
Let $\Lambda'=\Lambda_i+\Lambda_j+\tilde \Lambda\in \pcl(\Lambda)$ for some $i,j\in I$ such that $j\not\equiv_e i-1 $, i.e., $\Lambda'\neq\Lambda'_{i,j}$, see \eqref{equ::edge-loop}.

\begin{Lemma}\label{lem::recurrence}
Let $\Lambda, \Lambda'$ be as above.
Then, $\min(X_{\Lambda'}+\Delta_{i,j})$ is either $0$ or $1$.
Moreover, denoting $\Lambda'':= \Lambda'_{i,j}$, we have $\Lambda_{j+1,i-1}''=\Lambda'$ and one of the following holds.
\begin{enumerate}
\item If $\min(X_{\Lambda'}+\Delta_{i,j})=0$, then $X_{\Lambda''}= X_{\Lambda'}+\Delta_{i,j}$ and $\min (X_{\Lambda''}+\Delta_{j+1,i-1})=1$.

\item If $\min(X_{\Lambda'}+\Delta_{i,j})=1$, then $X_{\Lambda''}=X_{\Lambda'}-\Delta_{j+1,i-1}$ and $\min (X_{\Lambda''}+\Delta_{j+1,i-1})=0$.
\end{enumerate}
\end{Lemma}
\begin{proof}
Set $X_{\Lambda'}=(x_0,x_1,\ldots,x_\ell)$.
Then, the first statement follows from
Lemma \ref{lem::unique-solution} (i.e., $\min\{x_i\mid i\in I\}=0$) since the entries of $\Delta_{i,j}$ are 0 or 1. Next, we prove (1) and (2).

Since $\Lambda''=\Lambda_{i-1}+\Lambda_{j+1}+\tilde \Lambda$, we have
$$
(\Lambda-\Lambda'')-(\Lambda-\Lambda')=\Lambda'-\Lambda''=
\Lambda_i-\Lambda_{i-1}-\Lambda_{j+1}+\Lambda_j.
$$
Set $Y=Y_{\Lambda''}-Y_{\Lambda'}$, i.e., $Y=(t_m)_{m\in I}$ with $t_m=\langle h_m, \Lambda'-\Lambda''\rangle$.
Using $\alpha_m=2\Lambda_m-\Lambda_{m-1}-\Lambda_{m+1}+\delta_{0m}\delta$,
it is easy to check that, for any $s\in I$,
$$
\left< h_s, \ssum_{m\in [i,j]}\alpha_m \right>=\left< h_s,\Lambda_i-\Lambda_{i-1}-\Lambda_{j+1}+\Lambda_j\right >=t_s,
$$
which implies $A\Delta_{i,j}^t=Y^t$.
Hence,
$$
A(X^t_{\Lambda'}+\Delta_{i,j}^t)=A X^t_{\Lambda'}+Y^t= Y^t_{\Lambda''}.
$$
If $\min(X_{\Lambda'}+\Delta_{i,j})=0$, then $X_{\Lambda''}=X_{\Lambda'}+\Delta_{i,j}$ by the uniqueness in Lemma \ref{lem::unique-solution}.
Otherwise, we have $A(X^t_{\Lambda'}+\Delta^t_{i,j}-(1^{e})^t)=Y^t_{\Lambda''}$ and $\min(X_{\Lambda'}+\Delta_{i,j}-(1^{e}))=0$.
Thus $X _{\Lambda''}= X_{\Lambda'}+\Delta_{i,j}-(1^{e})$ by Lemma \ref{lem::unique-solution} again.
Since $\Delta_{j+1,i-1}=(1^{e})-\Delta_{i,j}$, we have $X_{\Lambda''}= X_{\Lambda'}-\Delta_{j+1,i-1}$.
\end{proof}

\begin{rem}
Fix $\Lambda\in \pcl$ with $k\geq2$.
The argument in the proof of Lemma \ref{lem::path-edge} gives us an algorithm to find a path between $\Lambda$ and $\Lambda'$ in $C(\Lambda)$.
Note that $X_\Lambda=(0^{e})$. Then,
using Lemma \ref{lem::recurrence} repeatedly, we obtain $X_{\Lambda'}$ from $X_\Lambda$ for any $\Lambda'\in \pcl(\Lambda)$.
This provides a useful way to obtain all elements of the set $\max^+(\Lambda)$ for the dominant integral weight $\Lambda$ of arbitrary level, without solving the equation \eqref{equ::the-equation}.
\end{rem}

\begin{Defn}\label{def::arrow}
Let $\vec{C}(\Lambda)$ be the quiver whose underlying undirected graph is $C(\Lambda)$, and we choose an orientation of the edge between $\Lambda'$ and $\Lambda''$ as $\Lambda' \rightarrow \Lambda''$ if $X_{\Lambda''}=X_{\Lambda'}+\Delta_{i,j}$.  
We label this arrow by $(i,j)$.
\end{Defn}

\begin{rem}
Suppose that $\Lambda'$ and $\Lambda''$ are two endpoints of an edge in $C(\Lambda)$.
By Lemma \ref{lem::recurrence}, we have either
$$
\xymatrix@C=1cm@R=1cm{\Lambda'\ar[r]^{(i,j)} &\Lambda''} \text{ in case (1)} 
\quad \text{or} \quad
\xymatrix@C=1.5cm@R=1cm{\Lambda''\ar[r]^{(j+1,i-1)} &\Lambda'} \text{ in case (2)},
$$
and we cannot have both. Hence the choice of the orientation is always possible and unique.
\end{rem}

For example, set $\Lambda=\Lambda_0+\Lambda_3+\Lambda_6$ with $\ell=6$ as before, $\vec{C}(\Lambda)$ is displayed as
$$
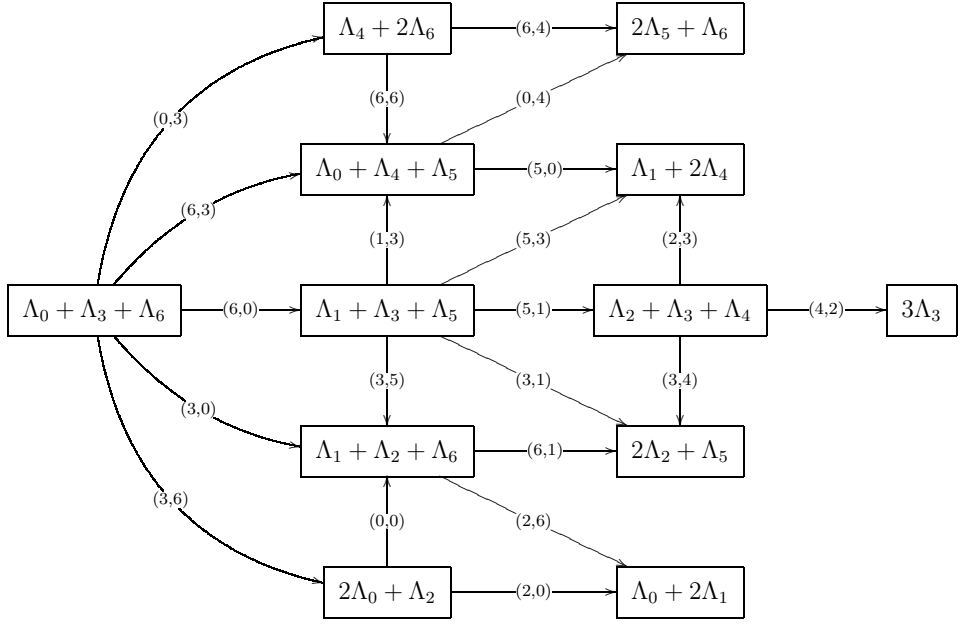
\scalebox{0.8}{
\xymatrix@C=2cm@R=1.5cm{
&*++[F]{\Lambda_4+2\Lambda_6}
\ar[d]|-{(6,6)}\ar[r]|-{(6,4)}
&*++[F]{2\Lambda_5+\Lambda_6}
&\\
&*++[F]{\Lambda_0+\Lambda_4+\Lambda_5}
\ar[ru]|-{(0,4)}\ar[r]|-{(5,0)}
&*++[F]{\Lambda_1+2\Lambda_4}
&\\
*++[F]{\Lambda_0+\Lambda_3+\Lambda_6}
\ar[r]|-{(6,0)}\ar@/^0.75cm/[ru]|-{(6,3)}\ar@/^1.4cm/[ruu]|-{(0,3)}
\ar@/_0.75cm/[rd]|-{(3,0)}\ar@/_1.4cm/[rdd]|-{(3,6)}
&*++[F]{\Lambda_1+\Lambda_3+\Lambda_5}
\ar[r]|-{(5,1)}\ar[ru]|-{(5,3)}
\ar[rd]|-{(3,1)}\ar[d]|-{(3,5)}\ar[u]|-{(1,3)}
&*++[F]{\Lambda_2+\Lambda_3+\Lambda_4}
\ar[r]|-{(4,2)}\ar[u]|-{(2,3)}\ar[d]|-{(3,4)}
&*++[F]{3\Lambda_3}\\
&*++[F]{\Lambda_1+\Lambda_2+\Lambda_6}
\ar[r]|-{(6,1)}\ar[rd]|-{(2,6)}
&*++[F]{2\Lambda_2+\Lambda_5}
&\\
&*++[F]{2\Lambda_0+\Lambda_2}
\ar[r]|-{(2,0)}\ar[u]|-{(0,0)}
&*++[F]{\Lambda_0+2\Lambda_1}
&
}}
$$
where the corresponding $X_{\Lambda'}$'s are given below.
$$
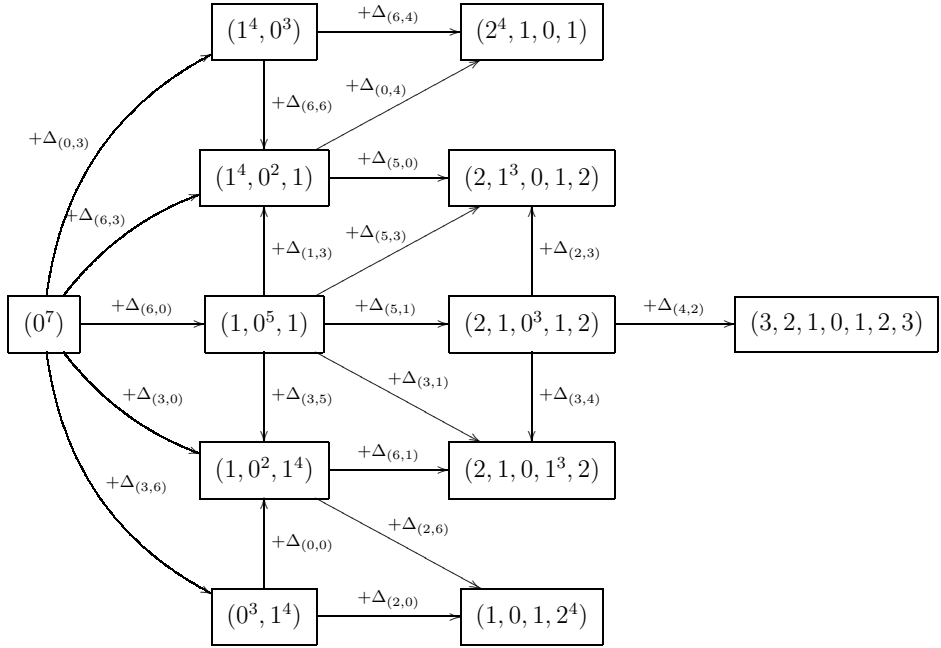
\scalebox{0.8}{
\xymatrix@C=2cm@R=1.5cm{
&*++[F]{(1^4, 0^3)}
\ar[d]^-{+\Delta_{(6,6)}}\ar[r]^-{+\Delta_{(6,4)}}
&*++[F]{(2^4, 1,0,1)}
&\\
&*++[F]{(1^4,0^2, 1)}
\ar[ru]^-{+\Delta_{(0,4)}}\ar[r]^-{+\Delta_{(5,0)}}
&*++[F]{(2,1^3,0,1,2)}
&\\
*++[F]{(0^7)}
\ar[r]^-{+\Delta_{(6,0)}}\ar@/^0.5cm/[ru]^-{+\Delta_{(6,3)}}
\ar@/^1cm/[ruu]^-{+\Delta_{(0,3)}}
\ar@/_0.5cm/[rd]^-{+\Delta_{(3,0)}}
\ar@/_1cm/[rdd]^-{+\Delta_{(3,6)}}
&*++[F]{(1, 0^5, 1)}\ar[r]^-{+\Delta_{(5,1)}}
\ar[ru]^-{+\Delta_{(5,3)}}\ar[rd]^-{+\Delta_{(3,1)}}
\ar[d]^-{+\Delta_{(3,5)}}\ar[u]_-{+\Delta_{(1,3)}}
&*++[F]{(2,1,0^3,1,2)}
\ar[r]^-{+\Delta_{(4,2)}}\ar[u]_-{+\Delta_{(2,3)}}
\ar[d]^-{+\Delta_{(3,4)}}
&*++[F]{(3,2,1,0,1,2,3)}
\\
&*++[F]{(1, 0^2, 1^4)}
\ar[r]^-{+\Delta_{(6,1)}}\ar[rd]^-{+\Delta_{(2,6)}}
&*++[F]{(2,1,0,1^3,2)}
&\\
&*++[F]{(0^3,1^4)}
\ar[r]^-{+\Delta_{(2,0)}}\ar[u]_-{+\Delta_{(0,0)}}
&*++[F]{(1,0,1,2^4)}
&
}}
$$

In general, we have the following useful observation.

\begin{Cor}\label{cor::find-arrow}
Suppose $\Lambda'=\Lambda_i+\Lambda_j+\tilde \Lambda\in \pcl(\Lambda)$ with $k\geq 2$ such that $j\not\equiv_e i-1$ and $\Lambda''=\Lambda_{i-1}+\Lambda_{j+1}+\tilde \Lambda$. 
Set $X_{\Lambda'}=(x_0,x_1,\ldots,x_\ell)$. 
Then, there is an arrow $\xymatrix@C=1cm@R=1cm{\Lambda'\ar[r]|-{(i,j)} &\Lambda''}$ in $\vec{C}(\Lambda)$ if and only if there exists $h\in [j+1,i-1]$ satisfying $x_h=0$.
\end{Cor}
\begin{proof}
By definition, $\xymatrix@C=1cm@R=1cm{\Lambda'\ar[r]|-{(i,j)} &\Lambda''}$ if and only if $\min(X_{\Lambda'}+\Delta_{i,j})=0$, which is equivalent to the condition $x_h=0$ for some
$$
h\in \left\{\begin{array}{ll}
\{0, 1, \ldots, i-1, j+1, j+2, \ldots, \ell\} & \hbox{if } i\leq j,\\
\{j+1, j+2, \ldots, i-1\}                     & \hbox{if } i> j. 
\end{array}\right.
$$
Namely, $x_h=0$ for some $h\in [j+1,i-1]$.
\end{proof}

\begin{example}
Suppose $0\leq s\leq \ell$.
The explicit description of $\emph{max}^{+}(\Lambda_0+\Lambda_s)$ was given by Tsuchioka in \cite[Theorem 1.4]{T-level-2}.
Now, we give a different approach to obtain the same description of $\emph{max}^{+}(\Lambda_0+\Lambda_s)$ as follows.
We remind readers that the set $P^+_{cl,2}(\Lambda_0+\Lambda_s)$ is already mentioned in Example \ref{ex::P-level-2}.
Recall that $X_{\Lambda_0+\Lambda_s}=(0^{e})$.
Using Corollary \ref{cor::find-arrow} repeatedly, we obtain $\vec{C}(\Lambda_0+\Lambda_s)$ as follows. 
If $s=0$, then $\vec{C}(2\Lambda_0)$ is displayed as
$$
\xymatrix@C=1.5cm@R=1cm{
2\Lambda_0\ar[d]^{(0,0)} \\
\Lambda_1+\Lambda_\ell \ar[d]^-{(\ell,1)}\\
\Lambda_2+\Lambda_{\ell-1} \ar[d]^-{(\ell-1,2)}\\
\Lambda_3+\Lambda_{\ell-2} \ar[d]^-{(\ell-2,3)}\\
\vdots \ar[d]^-{( \lceil \frac{\ell+4}{2} \rceil, \lceil \frac{\ell-1}{2}\rceil)}\\
\Lambda_{\lceil \frac{e}{2} \rceil}
+\Lambda_{\lceil \frac{e+1}{2}\rceil}}
\ \rightsquigarrow \
\xymatrix@C=1.5cm@R=0.9cm{
X_{2\Lambda_0}=(0^{e})\\
X_{\Lambda_1+\Lambda_\ell}=(1,0^\ell) \\
X_{\Lambda_2+\Lambda_{\ell-1}}=(2, 1,0^{\ell-2},1) \\
X_{\Lambda_3+\Lambda_{\ell-2}}=(3,2,1,0^{\ell-4},1,2) \\
\vdots\\}
$$
We obtain
$$
X_{\Lambda_i+\Lambda_{e-i}}=\left(i,i-1,\ldots, 2,1,0^{\ell-2i+2}, 1,2,\ldots, i-2, i-1 \right),
$$
$$
\beta_{\Lambda_i+\Lambda_{e-i}}=\ssum_{0\leq j\leq i }(i-j)\alpha_j+\ssum_{\ell-i+2\leq j\leq \ell} (\ell-j+i-1)\alpha_j,
$$
for $1\leq i\leq \frac{e}{2}$.

If $s>0$, then $\vec{C}(\Lambda_0+\Lambda_s)$ is displayed as
$$
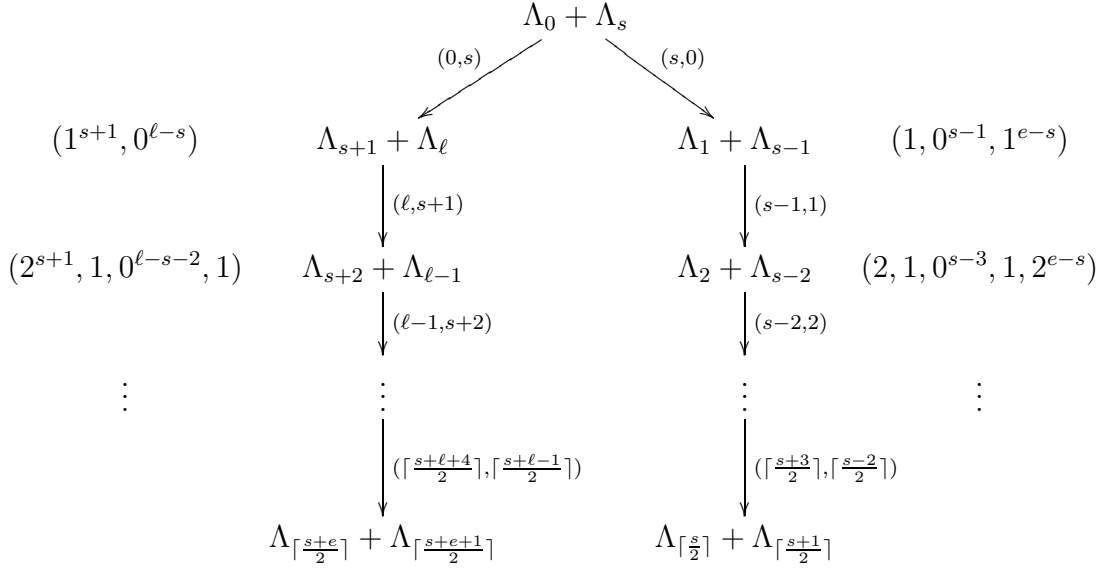
\xymatrix@C=0.1cm@R=1cm{
&&\Lambda_0+\Lambda_s\ar[dr]^-{(s,0)}\ar[dl]_-{(0,s)}
&&\\
(1^{s+1}, 0^{\ell-s})
&\Lambda_{s+1}+\Lambda_\ell \ar[d]^{(\ell,s+1)} 
&&\Lambda_1+\Lambda_{s-1} \ar[d]^-{(s-1,1)}
& (1, 0^{s-1}, 1^{e-s})
\\
(2^{s+1}, 1, 0^{\ell-s-2},1)
&\Lambda_{s+2}+\Lambda_{\ell-1} \ar[d]^-{(\ell-1,s+2)} 
&&\Lambda_2+\Lambda_{s-2} \ar[d]^-{(s-2,2)}
&(2,1, 0^{s-3}, 1, 2^{e-s})
\\
\vdots &\vdots \ar[d]^-{( \lceil \frac{s+\ell+4}{2}\rceil,\lceil \frac{s+\ell-1}{2}\rceil)} 
&& \vdots \ar[d]^-{(\lceil \frac{s+3}{2}\rceil,\lceil \frac{s-2}{2}\rceil )}
&\vdots \\
&\Lambda_{\lceil \frac{s+e}{2} \rceil}+\Lambda_{\lceil \frac{s+e+1}{2}\rceil}
&&\Lambda_{\lceil\frac{s}{2}\rceil}+\Lambda_{\lceil \frac{s+1}{2}\rceil}
&
}
$$
We obtain
$$
X_{\Lambda_j+\Lambda_{s-j}}=(j,j-1,\ldots,2,1,0^{s-2j+1},1,2,\ldots,j-1, j^{\ell-s+1}),
$$
$$
\beta_{\Lambda_j+\Lambda_{s-j}}=\ssum_{0\leq h\leq j  }(j-h)\alpha_h+\ssum_{s-j< h\leq s-1}(h-s+j)\alpha_h+\ssum_{s\leq h\leq \ell}j\alpha_h,
$$
for any $0\leq j\leq \frac{s}{2}$, and
$$
X_{\Lambda_{s+i}+\Lambda_{e-i}}=(i^{s+1},i-1,\ldots,2,1,0^{\ell-2i+2},1,2,\ldots,i-1),
$$
$$
\beta_{\Lambda_{s+i}+\Lambda_{e-i}}=\ssum_{0\leq h\leq s  }i\alpha_h+\ssum_{s<h\leq s+i-1} (s+i-h)\alpha_h +\ssum_{\ell-i+1< h\leq \ell}(h-\ell+i-1)\alpha_h,
$$
for any $1\leq i\leq \frac{e-s}{2}$.
One easily finds that $\beta_{\Lambda_j+\Lambda_{s-j}}$ and $\beta_{\Lambda_{s+i}+\Lambda_{e-i}}$ are exactly $\lambda_l^s$ and $\mu_l^s$ in \cite[Theorem 1.4]{T-level-2}, respectively.
\end{example}

We define $|X|:=\sum_{i\in I}x_i$ for any $X=(x_0,x_1\ldots,x_\ell)\in \N_0^{e}$.

\begin{Lemma}\label{lem::path-arrow}
For any $\Lambda'\in \pcl(\Lambda)$ with $\Lambda'\neq \Lambda$, there is a directed path from $\Lambda$ to $\Lambda'$ in $\vec{C}(\Lambda)$.
\end{Lemma}
\begin{proof}
Suppose $\Lambda'=\sum_{j\in I}m_j\Lambda_j$ and $X_{\Lambda'}=(x_0,x_1,\ldots,x_\ell)$ is the unique solution of \eqref{equ::the-equation} in Lemma \ref{lem::unique-solution}. 
We first observe that $X_{\Lambda'}$ is not a zero vector since $\Lambda\neq \Lambda'$. 
It implies $|X_{\Lambda'}|>0$. 
On the other hand, there exists at least one $i$ satisfying $x_i=0$. 
Thus, we may assume that there exist $i\neq j \in I$ such that $x_{h}\geq 1$ (resp. $x_h=0$)
 if $h\in [i+1,j]$ (resp. $h=i, j+1$).

Suppose $i\equiv_e j+1$, i.e., $x_h\geq 1$ for any $h\neq i$.
Since
$$
y_i=\langle h_{i}, \Lambda-\Lambda' \rangle=2x_i-x_{i-1}-x_{i+1}=-x_{i-1}-x_{i+1}\leq -2,
$$
we have
$$
m_{i}=\langle h_{i}, \Lambda' \rangle\geq\langle h_{i}, \Lambda \rangle+2\geq 2.
$$
So, $\Lambda'_{i,i}$ is well-defined and $\Lambda'_{i,i}\neq \Lambda'$.
In this case, we have $\min(X_{\Lambda'}+\Delta_{i,i})=1$ and
$$
X_{\Lambda'} =X_{\Lambda'_{i,i}}+\Delta_{i+1,i-1}
$$
by Lemma \ref{lem::recurrence} (2). 
Hence, we have $\xymatrix@C=1.7cm{\Lambda'_{i,i}\ar[r]|-{(i+1,i-1)} &\Lambda'}$ by Definition \ref{def::arrow}.
Note that $|X_{\Lambda'_{i,i}}|<|X_{\Lambda'}|$. 
If $|X_{\Lambda'_{i,i}}|=0$, then $\Lambda'_{i,i}=\Lambda$ and there is an arrow $\Lambda\rightarrow\Lambda' $. 
If $|X_{\Lambda'_{i,i}}|>0$, then there is a path $\Lambda\rightarrow\ldots \rightarrow\Lambda'_{i,i}$ by the induction hypothesis and the result follows.

Suppose $i\not\equiv_e j+1$. 
Then,
$$
\langle h_i, \Lambda-\Lambda' \rangle\leq-x_{i+1}\leq -1
\quad \text{and}\quad
\langle h_{j+1}, \Lambda-\Lambda' \rangle\leq-x_{ j}\leq -1.
$$
So, $m_i\geq 1$ and $m_{j+1}\geq 1$. 
Moreover, $i\not\equiv _e (j+1)-1$ since $i\neq j $. 
Hence, $\Lambda'_{j+1,i}$ is well-defined and $\Lambda'_{j+1,i}\neq \Lambda'$ by \eqref{equ::edge-loop}. 
In this case, $\min(X_{\Lambda'}+\Delta_{j+1,i})=1$ since $x_h\geq 1$ for $h\in [i+1,j]$, and we have
$$
X_{\Lambda'}=X_{\Lambda'_{j+1,i}}+\Delta_{i+1,j}
$$
by Lemma \ref{lem::recurrence} (2). 
Hence we have $\xymatrix@C=1.5cm{\Lambda'_{j+1,i}\ar[r]|-{(i+1,j)} &\Lambda'}$ by Definition \ref{def::arrow}.
Now the result follows from the induction hypothesis as above.
\end{proof}

Lemma \ref{lem::embedding-Lambda} has the following corollary.

\begin{Cor}\label{cor::embedding-path}
Suppose $\Lambda=\bar\Lambda+\tilde \Lambda$ with $\Lambda\in \pcl, \bar\Lambda\in P^+_{cl,k'}$ and $\tilde \Lambda\in P^+_{cl,k-k'}$. 
Then, there is a directed path
$$
\xymatrix@C=1.8cm@R=1cm{
\Lambda^{(1)}\ar[r]^-{(i_1,j_1)} &\Lambda^{(2)}\ar[r]^-{(i_2,j_2)} & \ldots \ar[r]^-{(i_{m-1},j_{m-1})} &\Lambda^{(m)}} \in \vec C(\bar\Lambda)
$$
if and only if there is a directed path
$$
\xymatrix@C=1.8cm@R=1cm{
\Lambda^{(1)}+\tilde\Lambda\ar[r]^-{(i_1,j_1)} &\Lambda^{(2)}+\tilde\Lambda\ar[r]^-{(i_2,j_2)} & \ldots \ar[r]^-{(i_{m-1},j_{m-1})} &\Lambda^{(m)}+\tilde\Lambda} \in \vec C(\Lambda).
$$
\end{Cor}
\begin{proof}
It follows from $X_{\bar \Lambda, \Lambda^{(i)}}=X_{\Lambda,\Lambda^{(i)}+\tilde\Lambda}$ for $1\leq i\leq m$.
\end{proof}

The following result will be useful in determining the representation type of cyclotomic quiver Hecke algebras.
\begin{Lemma}\label{lem::recurrence-beta}
Suppose $\Lambda'\overset{(j,i)}\longrightarrow\Lambda''$ in $\vec{C}(\Lambda)$.
\begin{enumerate}
\item If $i<j-1$, then there is a sequence $\beta_0=\beta_{\Lambda'},\beta_1,\ldots,\beta_{\ell-j+i+2}=\beta_{\Lambda''}$ in $Q_+$ defined by
$$
\beta_s=\left\{\begin{array}{ll}
    \beta_{s-1}+\alpha_{i-s+1}      & \hbox{if } 1\leq s\leq i+1, \\
    \beta_{s-1}+\alpha_{\ell-s+i+2} & \hbox{if } i+2\leq s\leq \ell-j+i+2,
\end{array}\right.
$$
such that $\langle h_{c_s}, \Lambda-\beta_{s-1} \rangle\geq 1$ for $1\leq s\leq \ell-j+i+1$, where $c_s=i-s+1$ (resp. $ \ell-s+i+2$) if $1\leq s\leq i+1$ (resp. if $i+2\leq s\leq \ell-j+i+2$).

\item If $i\geq j$, there is a sequence $\beta_0=\beta_{\Lambda'},\beta_1,\ldots, \beta_{i-j+1}=\beta_{\Lambda''}$ in $Q_+$ defined by $\beta_s=\beta_{s-1}+ \alpha_{i-s+1} $ for $ 1\leq s\leq i-j+1$, such that $\langle h_{i-s+1}, \Lambda-\beta_{s-1} \rangle\geq 1$ for $1\leq s\leq i-j+1$.
\end{enumerate}
\end{Lemma}
\begin{proof}
We only prove (1) since (2) can be checked similarly.
Note that we must have $\ell\geq 2$ in (1) and $(j,i)=(0, \ell)$ does not occur in (2) by $j-1\not\equiv_ei$.
As $i<j-1$, we have $X_{\Lambda''}=X_{\Lambda'}+(1^{i+1},0^{j-i-1},1^{e-j})$ by Definition \ref{def::arrow} and $\beta_{\ell-j+i+2}=\beta_{\Lambda''}$ by
$$
\beta_{\Lambda'}+(\alpha_i+\alpha_{i-1}+\ldots+\alpha_0+\alpha_\ell+\ldots+\alpha_j)=\beta_{\Lambda''}.
$$
Write $\Lambda'=\sum_{t\in I}m_t\Lambda_t$.
Since there is an arrow $\xymatrix@C=1cm@R=1cm{\Lambda'\ar[r]|-{(j,i)} &\Lambda''}$, $\Lambda'_{j,i}$ is well-defined and hence $m_i\geq 1$.
Then $\langle h_{c_1}, \Lambda-\beta_{0} \rangle= \langle h_{i}, \Lambda- \beta_{\Lambda'}\rangle=\langle h_i,\Lambda'\rangle=m_i\geq1$.
For $2\leq s\leq i+1$,
$$
\begin{aligned}
\left< h_{c_s}, \Lambda-\beta_{s-1} \right>
&=\left< h_{i-s+1},\Lambda- \beta_{\Lambda'}-\ssum_{1\leq f\leq s-1}\alpha_{i-f+1} \right>  \\
&=\left< h_{i-s+1}, \Lambda'-\ssum_{1\leq f\leq s-1}\alpha_{i-f+1} \right>
\geq \left< h_{i-s+1}, -\ssum_{1\leq f\leq s-1}\alpha_{i-f+1}\right>=1. 
\end{aligned}
$$
Similarly, for $i+2\leq s\leq \ell-j+i+2$, we have
$$
\left< h_{c_s}, \Lambda-\beta_{s-1} \right>
\geq \left< h_{\ell-s+i+2 }, -\ssum_{0\leq f\leq i}\alpha_{f}-\ssum_{\ell-s+i+3\leq g\leq \ell}\alpha_{g}\right>=1,
$$
where we understand $\sum_{\ell-s+i+3\leq g\leq \ell}\alpha_{g}=0$ if $s=i+2$.
\end{proof}

\subsection{A connected subquiver of $\vec C(\Lambda)$}\label{section::subgraph}
We have understood that each $\Lambda'\in \vec C(\Lambda)$ can be obtained from $\Lambda$ along a directed path of a finite length, say, $d$. 
For our purpose, we only need certain (not all) $\Lambda'$'s with $d\leq 2$. 
In particular, $\beta_{\Lambda'}$ can be easily obtained by using Lemma \ref{lem::recurrence} for any such $\Lambda'$'s.
With this idea in mind, we construct a connected subquiver $T(\Lambda)$ of $\vec C(\Lambda)$ in this subsection. 
One may find in the next section the role that $T(\Lambda)$ will play in proving the main result.

Fix $\Lambda=\sum_{i\in I}m_i\Lambda_i\in \pcl$ with $k\geq 2$. 
For any $0\leq s \leq k-1$, we define
$$
I(\Lambda)_s=\{i\in I\mid m_i\geq s+1\}.
$$
Then $\Lambda=\sum_{i\in I(\Lambda)_0}m_i\Lambda_i$. 
In the following, Corollary \ref{cor::find-arrow} will be frequently used without notice to find arrows in $\vec C(\Lambda)$.

We first consider the following two cases.
\begin{enumerate}
    \item Suppose $|I(\Lambda)_0|\geq2$. 
    In this case, we may write $\Lambda=\Lambda_i+\Lambda_j+\tilde \Lambda$ with $i<j\in I(\Lambda)_0$ and $\tilde \Lambda\in P^+_{cl,k-2}$. 
    If $i-1\not\equiv_e j$, then $\min (X_\Lambda+\Delta_{i,j})=0$ since $X_\Lambda=(0^e)$. 
    If $j-1\not\equiv_e i$, then $\min (X_\Lambda+\Delta_{j,i})=0$. 
    Hence, for each choice of $i<j$, there are arrows
    $$
    \xymatrix@C=1cm@R=1cm{\Lambda\ar[r]^-{(i,j)} &\Lambda_{i, j}} \in \vec C(\Lambda)
    \quad \text{and/or} \quad
    \xymatrix@C=1cm@R=1cm{\Lambda\ar[r]^-{(j,i)} &\Lambda_{j,i}} \in \vec C(\Lambda).
    $$
    We set
    $$
    \begin{aligned}
    T(\Lambda)_0:
    &=\left \{
    \Lambda_{i, j} \text{ if } i-1\not\equiv_e j,\  \Lambda_{j, i} \text{ if } j-1\not\equiv_e i \mid i<j\in I(\Lambda)_0 \right \}\\
    &=\left \{ \Lambda_{i,j}\mid i\neq j \in I(\Lambda)_0, [i,j]\neq I \right \}.
    \end{aligned}
    $$

    \item Suppose $|I(\Lambda)_1|\geq 1$. 
    In this case, we may write $\Lambda=2\Lambda_i+\tilde \Lambda$ with $i\in I(\Lambda)_1$ and $\tilde \Lambda\in P^+_{cl,k-2}$. 
    Since $\min (X_\Lambda+\Delta_{i,i})=0$, there is an arrow
    $$
    \xymatrix@C=1cm@R=1cm{\Lambda\ar[r]^-{(i,i)} &\Lambda_{i,i}} \in \vec C(\Lambda).
    $$
    We set $T(\Lambda)_1:=\{\Lambda_{i,i} \mid i\in I(\Lambda)_1\}$.
\end{enumerate}
Thus, $T(\Lambda)_0\sqcup T(\Lambda)_1$ consists of the vertices of $\vec C(\Lambda)$ which are reached from $\Lambda$ by directed paths of length one.

In case (2), we also consider certain $\Lambda''$'s which are reached from $\Lambda'$ by directed paths of length one, i.e.,
$$
\xymatrix@C=1cm@R=1cm{\Lambda\ar[r]^-{(i,i)} &\Lambda'\ar[r]^-{ ? }&\Lambda''} \in \vec C(\Lambda),
$$
where $\Lambda'=\Lambda_{i,i}$ and $X_{\Lambda'}=\Delta_{i,i}=(0^i,1,0^{\ell-i})$, with $i\in I(\Lambda)_1$.
\begin{enumerate}
    \item[(2-1)] Suppose $\ell \geq 3$. Then, we have
    $$
    \Delta_{i-1,i+1}=\left\{\begin{array}{ll}
      (1^2, 0^{\ell-2}, 1)         & \hbox{if } i=0,\\
      (0^{i-1},1^3,0^{\ell-i-1})   & \hbox{if } 0<i<\ell, \\
      (1, 0^{\ell-2}, 1^2)         & \hbox{if } i=\ell.
    \end{array}\right.
    $$
    There exists a directed path
    $$
     \xymatrix@C=1.5cm@R=1cm{\Lambda\ar[r]^-{(i,i)} &\Lambda'\ar[r]^-{(i-1, i+1)}&\Lambda'_{i-1,i+1}} \in \vec C(\Lambda).
    $$
    We set $T(\Lambda)_2:=\{\Lambda'_{i-1,i+1}\mid \Lambda'=\Lambda_{i,i},i\in I(\Lambda)_1\}$ if $\ell\geq 3$.

    \item[(2-2)] Suppose $|I(\Lambda)_2|\geq 1$. 
    In this case, $\Lambda=3\Lambda_i+\tilde \Lambda$ with $i\in I(\Lambda)_2$ and $\tilde \Lambda\in P^+_{cl,k-3}$. Then,
    $$
    \Delta_{i,i+1}=\left\{\begin{array}{ll}
      (0^{i},1^2,0^{\ell-i-1})   & \hbox{if } 0\leq i<\ell, \\
      (1, 0^{\ell-1}, 1)         & \hbox{if } i=\ell.
    \end{array}\right.
    $$
    $$
    \Delta_{i-1,i}=\left\{\begin{array}{ll}
    (1, 0^{\ell-1}, 1)           & \hbox{if } i=0,\\
    (0^{i-1},1^2,0^{\ell-i})     & \hbox{if } 0< i\leq \ell.
    \end{array}\right.
    $$
    If $\ell\geq 2$, then there exist two directed paths
    \begin{center}
    $\xymatrix@C=1.2cm@R=1cm{\Lambda\ar[r]^-{(i,i)} &\Lambda'\ar[r]^-{(i, i+1)}&\Lambda'_{i, i+1}} \in \vec C(\Lambda)
     \quad \text{and} \quad
     \xymatrix@C=1.2cm@R=1cm{\Lambda\ar[r]^-{(i,i)} &\Lambda'\ar[r]^-{(i-1, i)}&\Lambda'_{i-1, i}} \in \vec C(\Lambda)$.
    \end{center}
    We set $T(\Lambda)_3:=\{\Lambda'_{i, i+1}, \Lambda'_{i-1, i}\mid \Lambda'=\Lambda_{i,i}, i\in I(\Lambda)_2\}$ if $\ell\geq 2$.
    
    \item[(2-3)] Suppose $|I(\Lambda)_3|\geq 1$. 
    In this case, $\Lambda=4\Lambda_i+\tilde \Lambda$ with $i\in I(\Lambda)_3$ and $\tilde \Lambda\in P^+_{cl,k-4}$. 
    One can easily find a directed path
    $$
     \xymatrix@C=1cm@R=1cm{\Lambda\ar[r]^{(i,i)} &\Lambda'\ar[r]^-{(i, i)}&\Lambda'_{i, i}} \in \vec C(\Lambda).
    $$
    We set $T(\Lambda)_4:=\{\Lambda'_{i, i}\mid \Lambda'=\Lambda_{i,i},i\in I(\Lambda)_3\}$.
    
    \item[(2-4)] Suppose $|I(\Lambda)_1|\geq 2$ and $\ell\geq 2$. 
    In this case, $\Lambda=2\Lambda_i+2\Lambda_j+\tilde \Lambda$ with $i\neq j\in I(\Lambda)_1$ and $\tilde \Lambda\in P^+_{cl,k-4}$. 
    There is a directed path
    $$
     \xymatrix@C=1cm@R=1cm{\Lambda\ar[r]^{(i,i)} &\Lambda'\ar[r]^-{(j, j)}&\Lambda'_{j, j}} \in \vec C(\Lambda).
    $$
    We set $T(\Lambda)_5:=\{\Lambda'_{j, j}\mid \Lambda'=\Lambda_{i,i}, i\neq j\in I(\Lambda)_1\}$ if $\ell\geq 2$.
\end{enumerate}

\begin{Defn}\label{def::T_Lambda}
For any $\Lambda\in \pcl$ with $k\geq 2$. 
We define $T(\Lambda)$ as the subquiver of $\vec C(\Lambda)$ consisting of vertices in
$$
\{\Lambda\}\cup \bigcup_{0\leq s\leq 5}T(\Lambda)_s,
$$
and arrows given in the construction of $T(\Lambda)_s$.
We sometimes use $T(\Lambda)$ to indicate the vertex set of this subquiver by abuse of notation. Furthermore, we define 
$$
\mathcal T(\Lambda):=\{\beta_{\Lambda'}\mid \Lambda'\in T(\Lambda)\}.
$$
\end{Defn}

By the construction of $T(\Lambda)_s$, we have the following remark.

\begin{rem}\label{rem::T_Lambda-beta}
Define $\mathcal T(\Lambda)_s:=\{\beta_{\Lambda'}\mid \Lambda'\in T(\Lambda)_s\}$ for $0\leq s\leq 5$.
Then
\begin{itemize}
\item $\mathcal T(\Lambda)_0= \{\sum_{m\in [i,j]}\alpha_m \mid
 i\neq j\in I(\Lambda)_0, [i,j]\neq I \}$ when $|I(\Lambda)_0|\geq 2$.

\item $\mathcal T(\Lambda)_1=\{\alpha_i\mid i\in I(\Lambda)_1\}$ when $|I(\Lambda)_1|\geq 1$.

\item $\mathcal T(\Lambda)_2=\{2\alpha_i+\alpha_{i-1}+\alpha_{i+1}
\mid i\in I(\Lambda)_1\}$ when $|I(\Lambda)_1|\geq 1$ and $\ell\geq 3$.

\item $\mathcal T(\Lambda)_3=\{2\alpha_i+\alpha_{i-1}, 2\alpha_i+\alpha_{i+1} \mid i\in I(\Lambda)_2\}$ when $|I(\Lambda)_2|\geq 1$ and $\ell\geq 2$.

\item $\mathcal T(\Lambda)_4=\{2\alpha_i\mid i\in I(\Lambda)_3\}$ when $|I(\Lambda)_3|\geq 1$.

\item $\mathcal T(\Lambda)_5=\{\alpha_i+\alpha_j \mid i\neq j\in I(\Lambda)_1\}$ when $|I(\Lambda)_1|\geq 2$ and $\ell\geq 2$.
\end{itemize}
\end{rem}

\begin{rem}
Recall that $\beta_{\Lambda,\Lambda'}=\sum_{i\in I}x_i\alpha_i$ for $\Lambda'\in \pcl (\Lambda)$, where $X=(x_0,\ldots,x_\ell)$ is the solution of $AX^t=Y^t_{\Lambda,\Lambda'}$ from Lemma \ref{lem::unique-solution}.
Recall $\sigma$ from \eqref{def::sigma}.
By \eqref{equ:sigmaX}, we have $\xymatrix@C=1.2cm{\Lambda' \ar[r]|-{(i,j)}&\Lambda''}$ in $\vec C(\Lambda)$ if and only if $\xymatrix@C=2.2cm{\sigma \Lambda' \ar[r]|-{(\sigma(i), \sigma(j))}& \sigma \Lambda''}$ in $\vec C(\sigma \Lambda)$. 
In other words, we have
\begin{equation}\label{equ::iso-sigma-graph}
\vec{C}(\sigma\Lambda)=\sigma \vec C(\Lambda) 
\quad \text{and} \quad 
T(\sigma\Lambda)=\sigma T(\Lambda).
\end{equation}
\end{rem}

We have the next lemma.

\begin{Lemma}\label{lem::embdding-subgraph}
Suppose $k\geq 2$ and $\Lambda=\bar\Lambda+\tilde \Lambda$ with $\Lambda\in \pcl, \bar\Lambda\in P^+_{cl,k'}$ and $\tilde \Lambda\in P^+_{cl,k-k'}$. 
Then, $T(\bar\Lambda)$ is a subquiver of $T(\Lambda)$
and $\mathcal T(\bar\Lambda)\subset \mathcal T(\Lambda)$.
\end{Lemma}
\begin{proof}
This follows from Lemma \ref{lem::embedding-Lambda}, since if $\Lambda'\in T(\bar\Lambda)$ then $\Lambda'+\tilde \Lambda\in T(\Lambda)$.
\end{proof}

\begin{Defn}
For any $\Lambda'\in \vec{C}(\Lambda)$, we define
$$
S(\Lambda')_\rightarrow:=
\{\Lambda''\in \vec{C}(\Lambda)\mid \exists\  \Lambda'\overset{(i,j)}{\longrightarrow} \Lambda'' \in \vec{C}(\Lambda) \},
$$
the set of all successors of $\Lambda'$.
\end{Defn}

\begin{rem}\label{rem::-P-lambda}
As we mentioned above, $S(\Lambda)_\rightarrow=T(\Lambda)_0\cup T(\Lambda)_1$.
\end{rem}

The following lemma plays a critical role in the proof of Claim \ref{claim::step-3}.

\begin{Lemma}\label{lem::reduction-step-3}
For any $\Lambda\in \pcl$ with $k\geq 2$ and $\Lambda'\in T(\Lambda)_s$ with $0\leq s\leq 5$, we suppose that $\Lambda''\in S(\Lambda')_\rightarrow$ and one of the following conditions holds:
\begin{enumerate}
    \item  $k\geq 5$ when $s=0,1,2$,
    \item  $k\geq 6$ when $s=3$,
    \item  $k\geq 7$ when $s=4,5$.
\end{enumerate}
Then, there exist some $r\in I(\Lambda)_0$ and $\bar\Lambda, \bar\Lambda', \bar\Lambda''\in P^+_{cl,k-1}$ such that
$$ 
\Lambda=\bar\Lambda+\Lambda_r,
\Lambda'=\bar\Lambda'+\Lambda_r,
\Lambda''=\bar\Lambda''+\Lambda_r,
\bar\Lambda' \in  T(\bar \Lambda)_s \text{ and } \bar\Lambda''\in S(\bar\Lambda')_\rightarrow.
$$
\end{Lemma}
\begin{proof}
(1) Suppose $s=0$. By the definition of $T(\Lambda)_0$, we have
$$
\xymatrix@C=1.2cm@R=0.1cm{
\Lambda \ar[r]^{(i,j)}&\Lambda'\ar[r]^{(m,n)}&\Lambda''}
\in \vec C(\Lambda),
$$
where $\Lambda'=\Lambda_{i,j}$, $i\neq j\in I(\Lambda)_0, [i,j]\neq I$ and $m, n \in I$ are such that $n\not\equiv_e m-1$. 
Since each arrow changes exactly two summands of the source, the composition of $(i,j)$ and $(m,n)$ changes at most four summands of $\Lambda$.
There is at least one summand, say $\Lambda_r$, for some $r\in I(\Lambda)_0$, which is preserved under the two moves, if $k\ge 5$.
Thus, we may write $\Lambda=\bar\Lambda+\Lambda_r$, $\Lambda'=\bar\Lambda'+\Lambda_r$, $\Lambda''=\bar\Lambda''+\Lambda_r$, and the above path yields a directed path
$$
\xymatrix@C=1.2cm@R=0.1cm{
\bar\Lambda \ar[r]^{(i,j)}&\bar\Lambda'\ar[r]^{(m,n)}&\bar\Lambda''}
\in \vec C(\bar\Lambda),
$$
by Corollary \ref{cor::embedding-path}. 
It is then obvious that $\bar\Lambda'\in T(\bar\Lambda)_0$ and $\bar \Lambda''\in S(\bar \Lambda')_\rightarrow$. 
Similarly, the statement is true for $s=1, k\ge 5$ if one takes $i=j\in I(\Lambda)_1$ in the above.

(2) $s=2$. By the definition of $T(\Lambda)_2$, we have
$$
\xymatrix@C=1.5cm@R=0.1cm{
\Lambda \ar[r]^-{(i,i)}&\Lambda_{i,i}
\ar[r]^-{(i-1,i+1)}&\Lambda'\ar[r]^{(m,n)}&\Lambda''}
\in \vec C(\Lambda),
$$
where $\Lambda'=(\Lambda_{i,i})_{i-1,i+1}$, $i\in I(\Lambda)_1$ and $m,n \in I$ are such that $n\not\equiv_e m-1$. 
Since the composition of arrows $(i,i)$ and $(i-1,i+1)$ changes only two summands $2\Lambda_i$ of $\Lambda$ to $\Lambda_{i-2}$ and $\Lambda_{i+2}$, the next composition by $(m,n)$ changes at most four summands of $\Lambda$. 
Hence, the statement is true if $k\ge 5$.

(3) $s=3$. By the definition of $T(\Lambda)_3$, we obtain
$$
\vcenter{\xymatrix@C=1.5cm@R=0.1cm{ &&(\Lambda_{i,i})_{i,i+1}\ar[r]^-{(m,n)}&\Lambda''\\
\Lambda \ar[r]^-{(i,i)}&\Lambda_{i,i}
\ar[dr]_-{(i-1,i)}\ar[ur]^-{(i,i+1)}&& \\
&&(\Lambda_{i,i})_{i-1,i}\ar[r]^-{(m,n)}&\Lambda''
}}
\in \vec C(\Lambda),
$$
where $i\in I(\Lambda)_2$ and $m,n \in I$ are such that $n\not\equiv_e m-1$. 
In this case, the composition of arrows $(i,i)$ and $(i,i+1)$ (or $(i-1,i)$) changes three summands of $\Lambda$.
Thus, the statement is true if $k\ge 6$.

(4) We settle $s=5$ first. By the definition of $T(\Lambda)_5$, we have
$$
\xymatrix@C=1.2cm@R=0.1cm{
\Lambda \ar[r]^-{(i,i)}&\Lambda_{i,i}
\ar[r]^-{(j,j)}&\Lambda'\ar[r]^{(m,n)}&\Lambda''}
\in \vec C(\Lambda),
$$
where $\Lambda'=(\Lambda_{i,i})_{j,j}$, $i\neq j\in I(\Lambda)_1$ and $m,n \in I$ are such that $n\not\equiv_e m-1$. 
Since the composition of arrows $(i,i)$ and $(j,j)$ changes four summands of $\Lambda$, the statement is true if $k\ge 7$. 
Similarly, the statement is true for $s=4, k\ge 7$ if one takes $i=j\in I(\Lambda)_3$.
\end{proof}

We end this section with the following example.

\begin{example}
Let $\Lambda=4\Lambda_0+2\Lambda_3+\Lambda_6$ and $\ell=6$. 
Then, $T(\Lambda)$ is displayed as follows.
$$
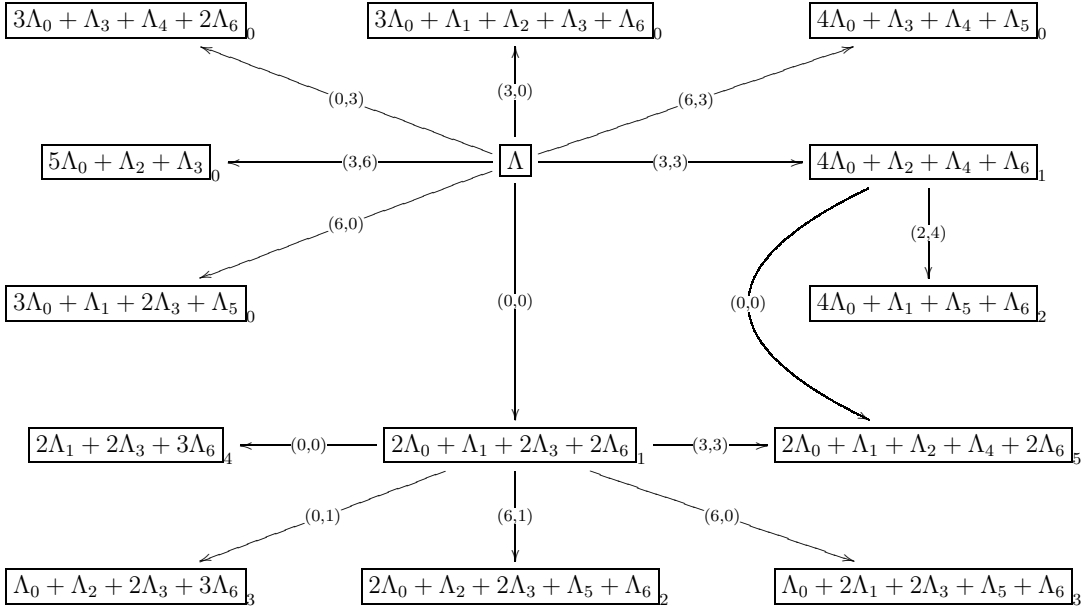
\scalebox{0.8}{
\xymatrix@C=1.5cm@R=1.5cm{
\boxed{3\Lambda_0+\Lambda_3+\Lambda_4+2\Lambda_6}_0
&\boxed{3\Lambda_0+\Lambda_1+\Lambda_2+\Lambda_3+\Lambda_6}_0
&
\boxed{4\Lambda_0+\Lambda_3+\Lambda_4+\Lambda_5}_0
\\
\boxed{5\Lambda_0+\Lambda_2+\Lambda_3}_0
&\boxed{\Lambda}
\ar[dd]|-{(0,0)}\ar[r]|-{(3,3)}\ar[ul]|-{(0,3)}\ar[ur]|-{(6,3)}
\ar[u]|-{(3,0)}\ar[l]|-{(3,6)}\ar[dl]|-{(6,0)}
&\boxed{4\Lambda_0+\Lambda_2+\Lambda_4+\Lambda_6}_1
\ar@/_3cm/[dd]|-{(0,0)}\ar[d]|-{(2,4)}
\\
\boxed{3\Lambda_0+\Lambda_1+2\Lambda_3+\Lambda_5}_0
&&\boxed{4\Lambda_0+\Lambda_1+\Lambda_5+\Lambda_6}_2
\\
\boxed{2\Lambda_1+2\Lambda_3+3\Lambda_6}_4
&\boxed{2\Lambda_0+\Lambda_1+2\Lambda_3+2\Lambda_6}_1
\ar[r]|-{(3,3)}\ar[d]|-{(6,1)}\ar[dr]|-{(6,0)}\ar[l]|-{(0,0)}\ar[ld]|-{(0,1)}
&\boxed{2\Lambda_0+\Lambda_1+\Lambda_2+\Lambda_4+2\Lambda_6}_5
\\
\boxed{\Lambda_0+\Lambda_2+2\Lambda_3+3\Lambda_6}_3
&\boxed{2\Lambda_0+\Lambda_2+2\Lambda_3+\Lambda_5+\Lambda_6}_2
&\boxed{\Lambda_0+2\Lambda_1+2\Lambda_3+\Lambda_5+\Lambda_6}_3
}}
$$
Here, the subscript $s$ on the lower right corner of each box indicates $T(\Lambda)_s$ that the element belongs.
\end{example}

\section{Representation type of cyclotomic quiver Hecke algebras}

In this section, we first introduce some technical lemmas in order to reduce the general cases to some small rank cases.
Secondly, we explain the main strategy to determine the representation type of cyclotomic quiver Hecke algebras in affine type $A$.
Finally, we give the main result of this paper, whose proof will be provided in the next three sections.

\subsection{Reduction to certain maximal weights}
We reduce the problem to the problem on certain maximal weights following the strategy in \cite{AP-rep-type-A_2-level-1, Ar-rep-type}.

A nice feature of $R^{\Lambda}(\beta)$ is that it is a symmetric algebra.
Rickard's classical result \cite{Rickard-derived-equi} combined with Krause's result \cite{Krause-rep-type-stable-equi} tells us that two self-injective (especially, symmetric) algebras have the same representation type if they are derived equivalent.
On the other hand, it is known from \cite{CR-categorification} that $R^{\Lambda}(\beta)$ and $R^{\Lambda}(\beta')$ are derived equivalent if $\Lambda-\beta$ and $\Lambda-\beta'$ lie in the same $W$-orbit.
Thus, it suffices to consider the representatives of $W$-orbits of $P(\Lambda)$ if one studies the representation type of $R^{\Lambda}(\beta)$ for $\beta\in Q_+$.

By \eqref{equ::weight-set} and \eqref{equ::max-set}, the representatives of $W$-orbits in $P(\Lambda)$ are given by $\{\Lambda-\beta -m\delta\mid \beta\in \P^\Lambda, m\in \Z_{\geq0} \}$, where $\P^\Lambda$ is defined in \eqref{equ::max-set}.
Hence, in order to prove the main result, it suffices to consider $R^{\Lambda}(\gamma)$  with $\gamma$ being an element in
$$
O(\Lambda):= \{ \beta+m\delta\mid \beta\in \P^\Lambda, m\in \Z_{\ge 0} \}.
$$

\subsection{Reduction Lemmas}
In this subsection, we give two useful reduction lemmas to reduce the problem to a certain subset of $O(\Lambda)$.
\begin{Lemma}\label{lem::reduction-level}
Suppose $\Lambda=\bar\Lambda+\tilde \Lambda$ with $\Lambda\in \pcl$, $\bar\Lambda\in P^+_{cl,k'}$ and $\tilde \Lambda\in P^+_{cl,k-k'}$.
If $R^{\bar\Lambda}(\beta+m\delta)$ is of infinite (resp. wild) representation type for $\beta\in \P^{\bar\Lambda}$ and $m\in \Z_{\ge 0}$, then $R^{\Lambda}(\beta+m\delta)$ is of infinite (resp. wild) representation type.
\end{Lemma}
\begin{proof}
Since $\langle h_{\nu_1}, \bar\Lambda\rangle\leq \langle h_{\nu_1}, \Lambda\rangle$ for any $\nu\in I^{\beta+m\delta}$, we have a surjective algebra homomorphism from $R^{\Lambda}(\beta+m\delta)$ to $R^{\bar\Lambda}(\beta+m\delta)$, which sends $e(\nu), x_i, \psi_i$ to $e(\nu), x_i, \psi_i$ elementwise. 
The assertion follows from the definition of infinite (resp. wild) representation type.
\end{proof}

\begin{Lemma}\label{lem::reduction-arrow}
Suppose that there is an arrow $\Lambda'\overset{(i,j)}\longrightarrow\Lambda''$ in $\vec{C}(\Lambda)$.
If $R^{\Lambda}(\beta_{\Lambda'}+m\delta)$ is of infinite (resp. wild) representation type for $m\in \Z_{\ge 0}$, then $R^{\Lambda}(\beta_{\Lambda''}+m\delta)$ is of infinite (resp. wild) representation type.
\end{Lemma}
\begin{proof}
There are the induction functor $F_i:R^{\Lambda}(\beta)\rightarrow R^{\Lambda}(\beta+\alpha_i)$ and the restriction functor $E_i:R^{\Lambda}(\beta)\rightarrow
R^{\Lambda}(\beta-\alpha_i)$ for any $\beta\in Q_+$ and $i\in I$. 
It is proved in \cite{Ka-klr-alg, KK-categorification} that if $q_i=\langle h_i, \Lambda-\beta \rangle\geq 0$, then
$$
F_i E_i\oplus q_i\text{id}\simeq E_i F_i.
$$
Moreover, if $q_i>0$, then any finite-dimensional $R^{\Lambda}(\beta)$-module $M$ is a direct summand of $E_iF_i(M)$, and $\dim F_i(M)\leq C \dim M$ for a constant $C$.
Thus, by \cite[Proposition 2.3]{EN-rep-type-Hecke}, the wildness of $R^{\Lambda}(\beta)$ implies the wildness of $R^{\Lambda}(\beta+\alpha_i)$ if $\langle h_i, \Lambda-\beta \rangle\geq 1$.
Using this fact repeatedly to the sequence of $\beta_i$'s in Lemma \ref{lem::recurrence-beta}, we conclude that the wildness of $R^{\Lambda}(\beta_{\Lambda'}+m\delta)$ implies the wildness of $R^{\Lambda}(\beta_{\Lambda''}+m\delta)$.

It remains to show that the infiniteness of $R^{\Lambda}(\beta)$ implies the infiniteness of $R^{\Lambda}(\beta+\alpha_i)$ if $\langle h_i, \Lambda-\beta \rangle\geq 1$.
Equivalently, we prove that if $R^{\Lambda}(\beta+\alpha_i)$ is representation-finite, then so is $R^{\Lambda}(\beta)$.
We assume that $R^{\Lambda}(\beta+\alpha_i)$ is representation-finite
and $\{N_1,N_2,\ldots,N_t\}$ is a complete set of the isomorphism classes of indecomposable $R^{\Lambda}(\beta+\alpha_i)$-modules.
Let $N:=\oplus_{i=1}^tN_i$. 
Then, we have $X\in \text{add}\ N$ for any $R^{\Lambda}(\beta+\alpha_i)$-module $X$, where $\text{add}\ N$ is the full subcategory of all direct summands of finite direct sums of $N$.

For any indecomposable $R^{\Lambda}(\beta)$-module $M$, we have $F_i(M)\in\text{add}\ N$ and hence,
$M\in \text{add}\ E_i(N)$.
Since $E_i(N)$ has finitely many indecomposable direct summands, $R^\Lambda(\beta)$ is representation-finite.
\end{proof}

We have the following immediate corollary of Lemma \ref{lem::reduction-arrow}.

\begin{Cor}\label{cor::reduction-path}
If $R^{\Lambda}(\beta_{\Lambda'}+m\delta)$ for $\Lambda'\in\vec C(\Lambda)$ and $m\in\Z_{\geq0}$ is of infinite (resp. wild) representation type and there is a directed path from $\Lambda'$ to $\Lambda''$ in $\vec{C}(\Lambda)$, then $R^{\Lambda}(\beta_{\Lambda''}+m\delta)$ is again of infinite (resp. wild) representation type.
\end{Cor}

\begin{Defn}\label{def::non-wild-component}
Let $NW(\Lambda)$ be the set of vertices $\Lambda'\in \vec C(\Lambda)$ such that $R^{\Lambda}(\beta_{\Lambda'})$ is not of wild representation type. 
We define $\mathcal {NW}(\Lambda):=\{ \beta_{\Lambda'}\mid \Lambda'\in NW(\Lambda)\}$.
\end{Defn}

By Lemma \ref{lem::path-arrow} and Lemma \ref{lem::reduction-arrow},
there is a directed path from $\Lambda$ to any $\Lambda'\in NW(\Lambda)$ such that the vertices of the path are contained in $NW(\Lambda)$.
We will use the next corollary of Lemma \ref{lem::reduction-arrow} when we complete Step 3 to prove the main result.

\begin{Cor}\label{cor::cover-non-wild}
Let $S\subset \pcl(\Lambda)$ be a vertex set containing $\Lambda$ and define
$$
S':=\cup_{\Lambda'\in S}S(\Lambda')_\rightarrow\setminus  S.
$$
If $R^\Lambda(\beta_{\Lambda''})$ is wild for any $\Lambda''\in S'$,
then $NW(\Lambda)\subseteq S$.
\end{Cor}
\begin{proof}
Suppose there exists some $\Lambda'\in NW(\Lambda)$ such that $\Lambda'\notin S$. 
By Lemma \ref{lem::path-arrow}, there is a directed path
$$
\Lambda=\Lambda^{(0)}\rightarrow \Lambda^{(1)}\rightarrow \ldots\rightarrow \Lambda^{(m)}=\Lambda' \in \vec C(\Lambda).
$$
Since $\Lambda\in S$ and $\Lambda'\notin S$, there exists some $0\leq i\leq m-1$ such that $\Lambda^{(i+1)}\notin S$ and $\Lambda^{(j)}\in S$ for all $j\leq i$.
It implies that $\Lambda^{(i+1)}\in S'$. 
By our assumption, $R^\Lambda(\beta_{\Lambda^{(i+1)}})$ is wild, and so is $R^{\Lambda}(\beta_{\Lambda'})$ by Corollary \ref{cor::reduction-path}. 
This is a contradiction since $\Lambda'\in NW(\Lambda)$.
\end{proof}

\subsection{Main strategy}
As the case of $k=2$ is already settled by the first author \cite{Ar-rep-type}, we
assume that $\Lambda\in \pcl$ with $k\geq 3$.
We determine the representation type of $R^{\Lambda}(\beta+m\delta)$ for any $\beta\in \P^\Lambda, m\in \Z_{\ge 0}$ by the following three steps.

\textbf{Step 1:} 
We show that $R^{\Lambda}(\beta+m\delta)$ is wild for all $m\geq 1$ if $\beta \neq 0$ and $R^{\Lambda}(m\delta)$ is wild for all $m\geq 2$, by using Lemma \ref{lem::reduction-level} and Lemma \ref{lem::reduction-arrow}.
It follows that if $R^{\Lambda}(\gamma)$ with $\gamma \in O(\Lambda)$ is not of wild representation type, then $\gamma\in \mathcal {NW}(\Lambda)\cup\{\delta\}$.

\textbf{Step 2:} 
We determine the representation type of $R^{\Lambda}(\gamma)$ for $\gamma\in \mathcal T(\Lambda)\cup\{\delta\}$. 
This is done basically via case-by-case consideration. 
In particular, the systematic approach introduced by the first author in \cite{Ar-rep-type} is well applied to find the quiver presentation of the basic algebra of $R^{\Lambda}(\gamma)$.

\textbf{Step 3:} 
We show that $\mathcal{NW}(\Lambda)\subset \mathcal T(\Lambda)$. 
This is proved via case-by-case consideration on small $k$ (i.e., $k=3,4,5,6$) and via induction on $k\geq 7$. 
Note that, by Corollary \ref{cor::cover-non-wild}, it suffices to consider $S(\Lambda')_\rightarrow$ for $\Lambda'\in T(\Lambda)$, so that the case-by-case analysis for small $k$ is feasible.

\subsection{Main result}\label{sec::main-result}
We introduce several more notations to state our main result.
For any $\Lambda=\sum_{i\in I(\Lambda)_0}m_i\Lambda_i\in \pcl$ with $k\geq 3$, we have defined  $I(\Lambda)_i$ and $\mathcal T(\Lambda)_i$ in Subsection \ref{section::subgraph}. We enumerate indices in $I(\Lambda)_0$ as
$$
I(\Lambda)_0=\{i_j\mid 1\leq j\leq h, i_1<i_2<\ldots<i_h\}
$$
and $i_{j-1}:=i_h$ (resp. $i_{j+1}:=i_1$) if $j=1$ (resp. $j=h$).
Then, we define the following subsets of $\mathcal T(\Lambda)$.
$$
\begin{aligned}
  \mathscr T(\Lambda)_0:=
  &\left\{\textstyle \sum_{m\in [i_j,i_{j+1}]}\alpha_{m} \mid  m_{i_j}=1, m_{i_{j+1}}=1, 1\leq j\leq h \right\} \subseteq \mathcal T(\Lambda)_0
  \\
  \mathscr T(\Lambda)_1:=
  &\left\{\textstyle \sum_{i_j\leq m\leq i_{j+1}}\alpha_m \mid m_{i_j}=1, m_{i_{j+1}}>1 \text{ or } m_{i_j}>1, m_{i_{j+1}}=1\right\} \subseteq \mathcal T(\Lambda)_0
  \\
  \mathscr T(\Lambda)_2:=
  &\left\{2\alpha_{i_j}+\alpha_{i_j-1}+\alpha_{i_j+1}\mid m_{i_j}=2, i_{j-1}\not\equiv_e i_j-1,i_{j+1}\not\equiv_e i_{j}+1\right\} \subseteq \mathcal T(\Lambda)_2
  \\
  &\text{if } \ch \k\neq 2. \text{ Otherwise, } \mathscr T(\Lambda)_2:=\emptyset.\\
  \mathscr T(\Lambda)_3:=
  &\left\{ 2\alpha_{i_j}+\alpha_{i_j+1}\mid m_{i_j}=3, i_{j+1}\not\equiv_e i_{j}+1\right\} 
  \\
  &\qquad\qquad\cup \left\{2\alpha_{i_j}+\alpha_{i_j-1}\mid m_{i_j}=3, i_{j-1}\not\equiv_e i_j-1\right\}  \subseteq \mathcal T(\Lambda)_3
  \\
  &\text{if } \ch \k\neq 3. \text{ Otherwise, } \mathscr T(\Lambda)_3:=\emptyset.
  \\
  \mathscr T(\Lambda)_4:=
  &\left\{2\alpha_{i_j}\mid m_{i_j}=4\right\} \subseteq \mathcal T(\Lambda)_4 
  \text{ if } \ch \k\neq 2. \text{ Otherwise, } \mathscr{T}(\Lambda)_4:=\emptyset.
  \\
  \mathscr T(\Lambda)_5:=
  &\left\{\alpha_{i_j}+\alpha_{i_p} \mid i_p\not\equiv_e i_j\pm1, m_{i_j}=m_{i_p}=2, j\neq p \right\} \subseteq \mathcal T(\Lambda)_5
\end{aligned}
$$

We will show in Step 2 that $R^\Lambda(\beta)$, for $\beta\in \mathscr F(\Lambda)$ (resp. $\beta\in \mathscr T(\Lambda)$) has finite (resp. tame) representation type, where 
\begin{equation}\label{equ::def-finite-tame}
\mathscr F(\Lambda)=\{0\}\cup\mathscr T(\Lambda)_0\cup \mathcal T(\Lambda)_1 
\quad \text{and}\quad  
\mathscr T(\Lambda)= \cup_{1\leq j\leq 5} \mathscr T(\Lambda)_j.
\end{equation}

By analyzing the representation type of $R^\Lambda(\delta)$ separately, we have the following.

\begin{Theorem}\label{theo::main-result}
Suppose $\Lambda\in \pcl$ with $k\geq 3$ and let $R^{\Lambda}(\beta+m\delta)$, for $\beta\in \P^\Lambda, m\in \Z_{\geq 0}$, be the cyclotomic quiver Hecke algebra of type $A_\ell^{(1)}$.
\begin{enumerate}
    \item If $\beta=0$, then $R^{\Lambda}(m\delta)$ is
    \begin{itemize}
        \item of finite representation type if $m=0$,
        \item of tame representation type if $m=1$, $\Lambda=k\Lambda_i$, $\ell=1$ with $t\neq \pm 2$,
        \item of tame representation type if $m=1$, $\Lambda=k\Lambda_i$, $\ell\geq 2$ with $t\neq (-1)^{\ell+1}$.
    \end{itemize}
Otherwise, it is of wild representation type.
    \item If $\beta\neq 0$, then $R^{\Lambda}(\beta+m\delta)$ is
    \begin{itemize}
        \item of finite representation type if $m=0$, $\beta=\alpha_i$ and $\Lambda\in 2\Lambda_i+P^+_{cl,k-2}$ with $i\in I$,
        \item of finite representation type if $m=0$ and $\beta\in \mathscr T(\Lambda)_0$,
        \item of tame representation type if $m=0$ and $\beta\in \cup_{1\leq j\leq 5}\mathscr T(\Lambda)_j$.
    \end{itemize}
    Otherwise, it is of wild representation type.
\end{enumerate}
\end{Theorem}

\begin{proof}
It follows from Step 1 and Step 3 that $R^\Lambda(\gamma)$ is wild for any $\gamma\notin \mathcal T(\Lambda)\cup\{\delta\}$.
Then, Step 2 provides us with a complete classification of representation types for $R^{\Lambda}(\gamma)$ with $\gamma\in \mathcal T(\Lambda)\cup\{\delta\}$: see Propositions \ref{prop::step-2-result-delta}, \ref{prop::step-2-result-T_1}, \ref{prop::step-2-result-T_0}, \ref{prop::step-2-result-T_2}, \ref{prop::step-2-result-T_3}, \ref{prop::step-2-result-T_4}, \ref{prop::step-2-result-T_5}.
\end{proof}

\section{Proof of Step 1}
In this section, we prove Step 1 by showing that for any $\gamma\in O(\Lambda)=\P^\Lambda+\Z_{\geq0}\delta$, $R^{\Lambda}(\gamma)$ is wild if $\gamma\notin \P^\Lambda\cup\{\delta\}$.
We consider $\Lambda=\sum_{i\in I(\Lambda)_0}m_i\Lambda_i\in \pcl$ with $k\geq 3$ and recall $I(\Lambda)_0=\{i\in I\mid m_i>0\}$.
We divide the proof into the following three lemmas.

\begin{Lemma}\label{lem::step-1-11}
If $|I(\Lambda)_0|=1$, then $R^\Lambda(\beta_{\Lambda'}+m\delta)$ is of wild representation type for any $\beta_{\Lambda'}\in\P^\Lambda\setminus \{0\}$ and $m\geq1$.
\end{Lemma}
\begin{proof}
Since $|I(\Lambda)_0|=1$, $\Lambda=k\Lambda_i$ for some $i\in I$.
By Proposition \ref{prop::iso-sigma} and \eqref{equ:sigmaX}, it is equivalent to considering $\Lambda=k\Lambda_0$ with $k\geq 3$.
By Remark \ref{rem::-P-lambda} and $T(\Lambda)_0=\emptyset$,
$$
S(k\Lambda_0)_\rightarrow=\{\Lambda_1+\Lambda_\ell+(k-2)\Lambda_0\}.
$$
Noting that $y_i=\left <h_i,2\Lambda_0-\Lambda_1-\Lambda_\ell\right >$ implies $x_i=\delta_{i0}$, we have $\beta_{\Lambda_1+\Lambda_\ell+(k-2)\Lambda_0}=\alpha_0$.
Since $R^{2\Lambda_0}(\alpha_0+m\delta)$ is wild by \cite[Lemma 7.3, Lemma 7.4, Proposition 11.9]{Ar-rep-type} for all $m\geq 1$, $R^{k\Lambda_0}(\alpha_0+m\delta)$ is also wild by Lemma \ref{lem::reduction-level}.
For any $\Lambda'\neq k\Lambda_0, (k-2)\Lambda_0+\Lambda_1+\Lambda_\ell$, there is a directed path from $(k-2)\Lambda_0+\Lambda_1+\Lambda_\ell$ to $\Lambda'$ in $\vec{C}(k\Lambda_0)$.
It follows that $R^{k\Lambda_0}(\beta_{\Lambda'}+m\delta)$ is wild for any $\beta_{\Lambda'}\in\P^{k\Lambda_0}\setminus \{0, \alpha_0\}$, by Corollary \ref{cor::reduction-path}.
\end{proof}

\begin{Lemma}
If $|I(\Lambda)_0|=1$, then $R^\Lambda(m\delta)$ is wild for any $m\geq 2$.
\end{Lemma}
\begin{proof}
Similar to the proof of Lemma \ref{lem::step-1-11}, it suffices to show that $R^{2\Lambda_0}(m\delta)$ is wild for all $m\geq 2$. 
This is also covered by \cite[Lemma 7.3, Proposition 11.9]{Ar-rep-type}.
\end{proof}

\begin{Lemma}\label{lem::step-1-21}
If $|I(\Lambda)_0|\geq 2$, then $R^\Lambda(\beta_{\Lambda'}+m\delta) $ is wild for any $\beta_{\Lambda'}\in\P^\Lambda$ and $m\geq 1$.
\end{Lemma}
\begin{proof}
It is proved in \cite[Theorem 7.5]{Ar-rep-type} that $R^{\Lambda_0+\Lambda_s}(m\delta)$ for $s\neq 0$ is wild unless $\ell=1$ and $m=1$. 
If $\ell>1$ or $\ell=1$ and $m>1$, then $R^\Lambda(\beta_\Lambda+m\delta)=R^{\Lambda}(m\delta)$ is wild by Proposition \ref{prop::iso-sigma} and Lemma \ref{lem::reduction-level}, and so is $R^{\Lambda}(\beta_{\Lambda'}+m\delta)$ by Corollary \ref{cor::reduction-path}.

Suppose $\ell=1$ and $m=1$.
It suffices to consider $k=3$ by Lemma \ref{lem::reduction-level}, and we may assume  $\Lambda=2\Lambda_0+\Lambda_1$ by Proposition \ref{prop::iso-sigma}.
In this case, $R^{\Lambda}(\delta)$ has two non-zero idempotent $e_1=e(01)$ and $e_2=e(10)$.
By Theorem \ref{theo::graded-dim} and the following patterns for adding nodes,
$$
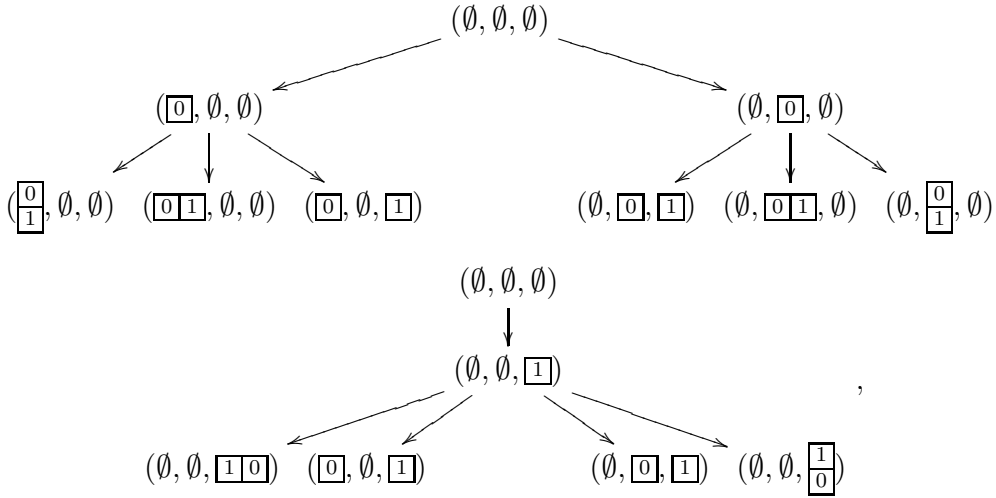
\ytableausetup{centertableaux, smalltableaux}
\vcenter{\xymatrix@C=0.1cm@R=0.5cm{
&&&
(\emptyset,\emptyset, \emptyset) \ar[drr] \ar[dll]
&&&\\
&(\ytableaushort{0},\emptyset, \emptyset ) \ar[d]\ar[dl]\ar[dr]
&&&&
(\emptyset,\ytableaushort{0},\emptyset)\ar[dr]\ar[d]\ar[dl]
&\\
(\ytableaushort{0,1}, \emptyset, \emptyset)
&
(\ytableaushort{01}, \emptyset, \emptyset)
&
(\ytableaushort{0}, \emptyset, \ytableaushort{1})
&&
(\emptyset, \ytableaushort{0}, \ytableaushort{1})
&
(\emptyset, \ytableaushort{01}, \emptyset)
&
(\emptyset, \ytableaushort{0,1}, \emptyset)
}}
$$
$$
\vcenter{\xymatrix@C=0.1cm@R=0.5cm{
&&(\emptyset, \emptyset,\emptyset) \ar[d] &&\\
& & (\emptyset, \emptyset,\ytableaushort{1}) \ar[dr] \ar[dl]\ar[drr] \ar[dll]&& \\
(\emptyset, \emptyset,\ytableaushort{10})
&
(\ytableaushort{0}, \emptyset, \ytableaushort{1})
&&
(\emptyset, \ytableaushort{0}, \ytableaushort{1})
&
(\emptyset, \emptyset,
\ytableausetup{smalltableaux}
\ytableaushort{1,0})
}},
$$
we have
$$
\begin{aligned}
\dim_q e_1 R^{\Lambda}(\delta)e_1&=1+2q^2+2q^4+q^6,\\
\dim_q e_2 R^{\Lambda}(\delta)e_2&=1+q^2+q^4+q^6,\\
\dim_q e_i R^{\Lambda}(\delta)e_j&=q^2+q^4, i\neq j.
\end{aligned}
$$
By \cite[Lemma 1.3]{Ar-rep-type}, the quiver of $R^{\Lambda}(\delta)$ contains two loops on vertex $1$, one loop on vertex $2$, one arrow $1\rightarrow 2$ and one arrow $2\rightarrow 1$. 
Then, $R^{\Lambda}(\delta)$ is wild by \cite[I.10.8 (i)]{Er-tame-block}, and so is $R^{\Lambda}(\beta_{\Lambda'}+\delta)$ by Corollary \ref{cor::reduction-path}.
\end{proof}

\section{Proof of Step 2}\label{sec::step-2}
We determine the representation type of $R^{\Lambda}(\gamma)$ for $\gamma\in \mathcal T(\Lambda)\cup\{\delta\}$ case by case.

\subsection{Representation type of $R^{\Lambda}(\delta)$}

It remains to consider the case $|I(\Lambda)_0|=1$ by Lemma \ref{lem::step-1-21}.

\begin{Lemma}\label{lem::step-2-delta-1}
Suppose $\Lambda=k\Lambda_i$ with $k\geq 3$ and $\ell=1$.
Then, $R^{\Lambda}(\delta)$ is a symmetric local algebra of tame representation type if $t\neq \pm2$, and is of wild representation type otherwise.
\end{Lemma}
\begin{proof}
By Proposition \ref{prop::iso-sigma}, it suffices to assume $i=0$.
Since $\delta=\alpha_0+\alpha_1$ and $\langle h_1, \Lambda_0\rangle=0$, there is the unique non-zero idempotent $e(01)$ of $R^{\Lambda}(\delta)$, i.e., $1=e(01)$. 
Considering the following patterns
$$
\ytableausetup{centertableaux, smalltableaux}
\vcenter{\xymatrix@C=-1.5cm@R=0.5cm{
&(\emptyset,\ldots,\emptyset) \ar[d] &\\
&(\emptyset, \ldots, \emptyset,\ytableaushort{0},
\emptyset, \ldots, \emptyset) \ar[dr] \ar[dl]& \\
(\emptyset, \ldots, \emptyset, \ytableaushort{01},
\emptyset, \ldots, \emptyset) & &(\emptyset, \ldots, \emptyset,
\ytableaushort{0,1},
\emptyset, \ldots, \emptyset)
}},
$$
we find that the graded dimension of $R^{\Lambda}(\delta)$ is given by
$$
\dim_q R^{\Lambda}(\delta) =1+2\ssum_{1\leq i\leq k-1}q^{2i}+q^{2k}.
$$
Hence, $R^{\Lambda}(\delta)$ is local. 
By Definition \ref{def::cyclotomic-quiver}, we have
$$
\psi_1=0,\quad  x_1^k=0, \quad x_1^2+t x_1x_2+ x_2^2=0.
$$
We conclude that $R^{\Lambda}(\delta)$ has $\{x_1^ax_2^b\mid 0\leq a\leq k-1, b=0,1 \}$ as basis. In particular, $R^{\Lambda}(\delta)$ is isomorphic to the bound quiver algebra $\k Q/\mathcal{I}$ with
$$
Q:\ \xymatrix@C=0.8cm{\bullet \ar@(dl,ul)^{\alpha} \ar@(ur,dr)^{\beta}}
\quad \text{and} \quad
\mathcal{I}:\  \left<\alpha^k, \alpha\beta-\beta\alpha, \alpha^2+\beta^2+t\alpha\beta \right>.
$$

Suppose $t\neq \pm2$.
Let $c^{\pm1}\in\k$ be the roots of the equation $z^2-t z+1=0$.
Then, $X:=\beta+c\alpha$ and $Y:=\beta+c^{-1}\alpha$ satisfy $XY=0=YX$, and $(X-Y)^k=0$ implies $X^k=(-1)^{k-1}Y^k$.  
We redefine $X$ by multiplying a suitable root of unity with $X$. Then, we may assume $XY=YX=0$ and $X^k=Y^k$.
By comparing the dimensions, we see that $R^{\Lambda}(\delta)$ is isomorphic to the bound quiver algebra defined by
$$
\xymatrix@C=0.8cm{\bullet \ar@(dl,ul)^{X} \ar@(ur,dr)^{Y}}
\quad \text{and}\quad  
\left<X^k-Y^k, XY, YX  \right>,
$$
which is a tame local symmetric algebra with $m=n=k\geq 3$ in \cite[Theorem III.1 (a)]{Er-tame-block}.

Suppose $t=\pm 2$.
Set $X:=\alpha$ and $Y:=\alpha \pm \beta$.
Then $R^{\Lambda}(\delta)$ is isomorphic to the bound quiver algebra defined by
$$
\xymatrix@C=0.8cm{\bullet \ar@(dl,ul)^{X} \ar@(ur,dr)^{Y}}
\quad \text{and}\quad \left<X^k, Y^2, XY-YX \right>,
$$
whose quotient algebra modulo $X^3$ and $X^2Y$ is exactly the wild local algebra labeled by $(c)$ in \cite[(1.1)]{Ringel-local-alg}. Thus, $R^{\Lambda}(\delta)$ is wild in this case.
\end{proof}

\begin{Lemma}\label{lem::step-2-delta-2}
Suppose $\Lambda=k\Lambda_i$ with $k\geq 3$ and $\ell\geq 2$. Then, $R^{\Lambda}(\delta)$ is of tame representation type if $t\neq (-1)^{e}$ and of wild representation type if $t=(-1)^{e}$.
\end{Lemma}
\begin{proof}
It suffices to assume $i=0$ by Proposition \ref{prop::iso-sigma}.
If $t=(-1)^{e}$, then $R^{\Lambda}(\delta)$ is wild for any $k\geq 3$ by \cite[Proposition 11.8]{Ar-rep-type}, which proved the wildness for $k=2$. Next, we assume that $t\neq (-1)^{e}$.

(1) Suppose $\ell=2$.
There are two non-zero idempotents $e_1=e(012)$ and $e_2=e(021)$ of $R^{\Lambda}(\delta)$, since $\{012,021\}$ is the set of residue sequences of standard tableaux.
Considering the following patterns
$$
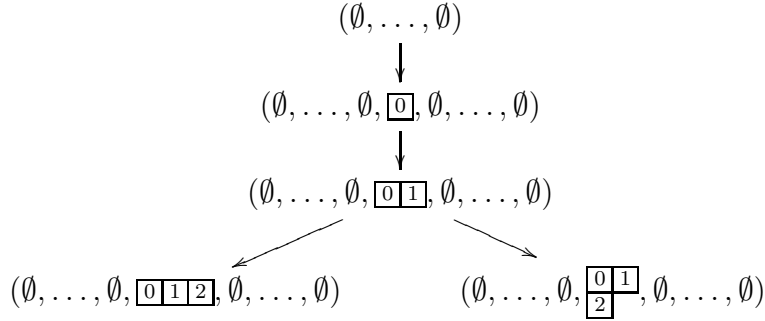
\vcenter{\xymatrix@C=-1.5cm@R=0.5cm{
&(\emptyset,\ldots,\emptyset) \ar[d] &\\
&(\emptyset, \ldots, \emptyset,
\ytableausetup{smalltableaux}
\ytableaushort{0},
\emptyset, \ldots, \emptyset) \ar[d] & \\
&(\emptyset, \ldots, \emptyset,
\ytableausetup{smalltableaux}
\ytableaushort{01},
\emptyset, \ldots, \emptyset) \ar[dr] \ar[dl]&\\
(\emptyset, \ldots, \emptyset,
\ytableausetup{smalltableaux}
\ytableaushort{012},
\emptyset, \ldots, \emptyset) & &(\emptyset, \ldots, \emptyset,
\ytableausetup{centertableaux, smalltableaux}
\ytableaushort{01,2},
\emptyset, \ldots, \emptyset)
}},
$$
$$
\vcenter
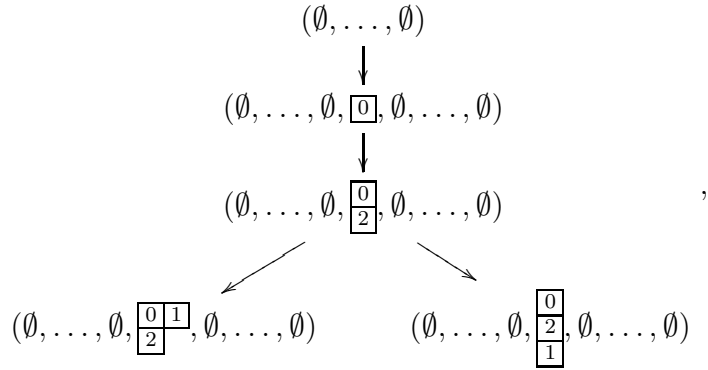
{\xymatrix@C=-1.5cm@R=0.5cm{
&(\emptyset,\ldots,\emptyset) \ar[d] &\\
&(\emptyset, \ldots, \emptyset,
\ytableausetup{smalltableaux}
\ytableaushort{0},
\emptyset, \ldots, \emptyset) \ar[d] & \\
&(\emptyset, \ldots, \emptyset,
\ytableausetup{smalltableaux}
\ytableaushort{0,2},
\emptyset, \ldots, \emptyset) \ar[dr] \ar[dl]&\\
(\emptyset, \ldots, \emptyset,
\ytableausetup{smalltableaux}
\ytableaushort{01,2},
\emptyset, \ldots, \emptyset) & &(\emptyset, \ldots, \emptyset,
\ytableausetup{centertableaux, smalltableaux}
\ytableaushort{0,2,1},
\emptyset, \ldots, \emptyset)
}},
$$
we find that the graded dimension of $R^{\Lambda}(\delta)$ is given by
$$
\begin{aligned}
\dim_q  e_1R^{\Lambda}(\delta)e_2& =\dim_q  e_2R^{\Lambda}(\delta)e_1 =\ssum_{1\leq i\leq k}q^{2i-1},\\
\dim_q e_jR^{\Lambda}(\delta)e_j  & =1+2\ssum_{1\leq i\leq k-1}q^{2i}+q^{2k}
\ \ \text{for}\  j=1,2.
\end{aligned}
$$
We give a basis of $R^{\Lambda}(\delta)$ as follows. Note that $x_1^ke_j=0$ for $j=1,2$.
\begin{itemize}
    \item Since $\psi_1e_1=0$ and $\psi_1^2e_1=(x_1+x_2)e_1$, we have $x_2e_1=-x_1e_1$. Similarly, we obtain $x_2e_2=-tx_1e_2$. Since $\psi_2^2e_1=(x_2+x_3)e_1$ and $x_3\psi_2^2e_1=\psi_2 x_2 e_2\psi_2= -tx_1\psi_2^2e_1$, we have $x_3^2e_1=tx_1^2e_1+(1-t)x_1x_3e_1$. Similarly, $x_3^2e_2=t x_1^2e_2+(t-1)x_1x_3e_2$. If $e_j\psi_w e_j\neq 0$, then $\psi_w=1$ because we may choose $\{ \psi_w \mid 1\neq w\in \mathfrak S_3\}$ to be $\{\psi_1,\psi_2,\psi_1\psi_2,\psi_2\psi_1, \psi_1\psi_2\psi_1\}$ and $e_1\psi_2e_1=e_1e_2\psi_2=0$, $e_2\psi_2e_2=e_2e_1\psi_2=0$. Thus, $e_jR^{\Lambda}(\delta)e_j$ has basis $\{ x_1^ax_3^be_j\mid 0\leq a\leq k-1, 0\leq b\leq 1\}$ for $j=1,2$.

    \item Since $x_2\psi_2e_2=-x_1\psi_2e_2$,
    $x_3\psi_2e_2=\psi_2x_2e_2=-tx_1\psi_2e_2$ and $x_2\psi_2e_1=-tx_1\psi_2e_1$, $x_3\psi_2e_1=\psi_2x_2e_1=-x_1\psi_2e_1$, $e_2 R^{\Lambda}(\delta)e_1$ has basis $\{x_1^a\psi_2 e_1\mid 0\leq a\leq k-1\}$ and $e_1 R^{\Lambda}(\delta)e_2$ has basis $\{x_1^a\psi_2 e_2\mid 0\leq a\leq k-1\}$.
\end{itemize}

Let $\alpha=x_1e_1$, $\beta=x_1e_2$, $\mu=\psi_2e_2$ and $\nu=\psi_2e_1$.
By the above computation, we obtain
$$
\alpha^k=\beta^k=0, \alpha\mu=\mu\beta, \beta\nu=\nu\alpha,
\mu\nu\mu=-(1+t)\alpha\mu, \nu\mu\nu=-(1+t)\beta\nu.
$$
Since $t\neq -1$, we may use $\alpha'=\mu\nu+(1+t)\alpha$ and $\beta'=\nu\mu+(1+t)\beta$ instead of $\alpha$ and $\beta$.
We have $\alpha'\mu=0=\mu\beta'$ and $\nu\alpha'=0=\beta'\nu$. Also,
$$
(\alpha'-\mu\nu)^k=(\alpha')^k+(-1)^k(\mu\nu)^k=0 \text{ and } (\beta'-\nu\mu)^k=(\beta')^k+(-1)^k(\nu\mu)^k=0
$$
imply $(\alpha')^k=(-1)^{k-1}(\mu\nu)^k$ and $(\beta')^k=(-1)^{k-1}(\nu\mu)^k$. By comparing the dimensions, $R^{\Lambda}(\delta)$ is isomorphic to the Brauer graph algebra whose Brauer graph is as follows.
$$
\scalebox{0.8}{\xymatrix@C=0.9cm{*++[o][F]{k}\ar@{-}[r] &*++[o][F]{k}\ar@{-}[r]&*++[o][F]{k}}}
$$
Thus, $R^{\Lambda}(\delta)$ is of tame representation type.

(2) Suppose $\ell\geq 3$.
Let $e_i=e(\nu_i)$ with $\nu_{i}=(0~1~2~\ldots~i-1~\ell~\ell-1~\ldots~i)$, for any $1\leq i\leq \ell$.
Set $\mathcal{A}=e'R^{\Lambda}(\delta) e'$ with $e'=\sum_{i=1}^{\ell}e_i$.
By simple observation of adding nodes in the sequence of multipartitions, one finds that the graded dimension of $\mathcal{A}$ is given by
$$
\begin{aligned}
\dim_q e_i \mathcal{A}e_j&=0, \quad |i-j|>1,\\
\dim_q e_i \mathcal{A}e_i &=1+2\ssum_{1\leq j\leq k-1}q^{2j}+q^{2k}, \quad 1\leq i\leq \ell,\\
\dim_{q} e_i\mathcal{A}e_{i+1}&=\dim_q e_{i+1}\mathcal Ae_i=\ssum_{1\leq j\leq k}q^{2j-1}, \quad 1\leq i\le \ell-1.
\end{aligned}
$$
A basis of $\mathcal{A}$ is given as follows.
For any $1\leq i\leq \ell$, we have $x_1^ke_i=0$.
\begin{itemize}
    \item For any $\psi_j$ with $1\le j\le \ell$, we observe that $\psi_1e(\nu)=0$ and $(tx_1+x_2)e(\nu)=0$, whenever $e(\nu)\neq 0$ such that the first entry of $\nu$ is $0$ and the second entry of $\nu$ is $\ell$; $\psi_je_i=0, (x_j+x_{j+1})e_i=\psi_j^2e_i=0$ if $j\neq i, \ell$; $\psi_\ell^2 e_i=(x_\ell+x_{\ell+1})e_i$; $\psi_i^2 e_i=e_i$ if $2\le i\le \ell-1$. Then,
    \begin{equation}\label{equ::step-2-basis-computation}
    \begin{aligned}
    x_{i+1}\psi_i^2e_i
    &=\psi_i x_i e(0~1~\ldots~i-2~\ell~i-1~\ell-1~\ldots~i)\psi_i\\
    &=\psi_i x_i \psi_{i-1}^2 e(0~1~\ldots~i-2~\ell~i-1~\ell-1~\ldots~i)\psi_i\\
    &=\psi_i \psi_{i-1} x_{i-1} e(0~1~\ldots~\ell~i-2~i-1~\ell-1~\ldots~i) \psi_{i-1} \psi_i\\
    &=\ldots \ldots\\
    &=\psi_i \psi_{i-1}\cdots \psi_3 \psi_2 x_2 e(0~\ell~1~2~\ldots~i-1~\ell-1~\ldots~i) \psi_2 \cdots \psi_{i-1} \psi_i\\
    &=\psi_i \psi_{i-1}\cdots \psi_2 (-tx_1) e(0~\ell~1~2~\ldots~i-1~\ell-1~\ldots~i) \psi_2 \cdots \psi_{i-1} \psi_i\\
    &=-tx_1\psi_i \psi_{i-1}\cdots \psi_3\psi_2^2 e(0~1~\ell~2~\ldots~i-1~\ell-1~\ldots~i) \psi_3 \cdots \psi_{i-1} \psi_i\\
    &=-tx_1\psi_i^2e_i,
    \end{aligned}
    \end{equation}
    and this implies $x_{i+1}e_i=-tx_1e_i$ for $2\le i\le \ell-1$. By simple substitution, we find
    \begin{equation}\label{equ::step-2-basis-computation-2}
    x_je_i=\left\{\begin{array}{ll}
    (-1)^{j-1}x_1e_i  & \text{ if } 2\leq j\leq i, \\
    (-1)^{j-i}tx_1e_i & \text{ if } i+1\leq j\leq \ell.
    \end{array}\right.
    \end{equation}
    for any $1\le i\le \ell$. Besides, $x_{\ell+1}\psi_\ell^2 e_\ell=\psi_\ell x_\ell e_{\ell-1} \psi_\ell =-tx_1\psi_\ell^2 e_\ell$, and the similar computation as \eqref{equ::step-2-basis-computation} shows
    $$
    x_{\ell+1}\psi_{\ell}^2e_i
    =\psi_\ell\cdots \psi_{i+1}x_{i+1}e_{i+1}\psi_{i+1}\cdots\psi_\ell
    =(-1)^i x_1\psi_\ell^2e_i
    $$
    for any $1\leq i\leq \ell-1$. This implies that
    \begin{equation}\label{equ::step-2-basis-computation-3}
    (x_{\ell+1}-c_ix_1)(x_{\ell+1}-d_ix_1)e_i=0, 1\leq i\leq \ell,
    \end{equation}
    with $c_i=(-1)^{i}$ and $d_i=(-1)^{\ell-i+1}t$, i.e., $x_{\ell+1}^2e_i$ can be replaced by a linear combination of $x_1^2e_i, x_1x_{\ell+1}e_i$ for any $1\leq i\leq \ell$. On the other hand, since $\nu_i$ has no entry repeated, $w\nu_i=\nu_i$ yields $\psi_w=1$ if $e_i\psi_w e_i\neq 0$ for $w\in \mathfrak S_{\ell+1}$ and $1\leq i\leq \ell$. Thus, we see that $e_i\mathcal{A}e_i$ has basis $\{ x_1^ax_{\ell+1}^b e_i\mid 0\leq a\leq k-1, 0\leq b\leq 1\}$ for any $1\le i\le \ell$, by $\dim_q e_i \mathcal{A} e_i$ given above.

    \item Note that $e_{i+1}\psi_w e_i\neq 0$ and $e_i\psi_we_{i+1}\neq0$ yield $w=s_{i+1}s_{i+2}\cdots s_\ell$ and $w=s_{\ell}s_{\ell-1}\cdots s_{i+1}$, respectively. Using \eqref{equ::step-2-basis-computation-2}, we have
    \begin{center}
    $
    x_{\ell+1}\psi_{i+1}\psi_{i+2}\cdots\psi_\ell e_i=\psi_{i+1}\psi_{i+2}\cdots\psi_\ell x_{\ell}e_i=(-1)^{\ell-i}tx_1\psi_{i+1}\psi_{i+2}\cdots\psi_\ell e_i,
    $
    \end{center}
    and
    \begin{center}
    $
    x_{\ell+1}\psi_\ell\cdots \psi_{i+2}\psi_{i+1} e_{i+1}=\psi_\ell\cdots \psi_{i+2}\psi_{i+1}x_{i+1} e_{i+1} =(-1)^ix_1\psi_\ell\cdots \psi_{i+2}\psi_{i+1} e_{i+1},
    $
    \end{center}
    for $1\le i\le \ell-1$.
    By the graded dimensions, we conclude that $e_{i+1} \mathcal{A} e_i$ has basis $\{x_1^a\psi_{i+1}\psi_{i+2}\cdots \psi_\ell e_i\mid 0\leq a\leq k-1\}$ and $e_{i} \mathcal{A} e_{i+1}$ has basis $\{x_1^a\psi_\ell\cdots \psi_{i+2}\psi_{i+1} e_{i+1}\mid 0\leq a\leq k-1\}$.
\end{itemize}

Set $\alpha=x_1e_1$, $\beta=x_1e_{\ell}$, $\mu_i=\psi_\ell\cdots \psi_{i+2}\psi_{i+1}e_{i+1}$ and $\eta_i=\psi_{i+1}\psi_{i+2}\cdots\psi_\ell e_i$, for $1\leq i\leq \ell-1$.
Then, $\alpha^k=\beta^k=0$.
Since $\psi_{\ell-1}e_{i}=e_{i}\psi_{\ell-1}=0$ for $1\leq i \le \ell-2$, we have
$$
\begin{aligned}
\eta_{i+1}\eta_i&=\psi_{i+2}\psi_{i+3}\cdots\psi_{\ell-1}\psi_{i+1}\cdots \psi_{\ell-1}\psi_{\ell}\psi_{\ell-1}e_i=0,\\
\mu_i\mu_{i+1}&=e_i\psi_{\ell-1}\psi_{\ell}\psi_{\ell-1}\psi_{\ell-2}\cdots \psi_{i+1} \psi_{\ell-1}\cdots \psi_{i+2}=0.
\end{aligned}
$$
We can show $\eta_i\mu_i=(x_{\ell+1}-c_{i+1}x_1)e_{i+1}$ and $\mu_{i+1}\eta_{i+1}=(x_{\ell+1}-d_{i+1}x_1)e_{i+1}$ for $1\leq i\leq \ell-2$.
Then, $(\mu_{i+1}\eta_{i+1}-\eta_i\mu_i)^k=(c_{i+1}-d_{i+1})^kx_1^ke_{i+1}=0$ together with \eqref{equ::step-2-basis-computation-3} imply
$$
(\eta_i\mu_i)^k=(-1)^{k-1} (\mu_{i+1}\eta_{i+1})^k.
$$
Similarly, we obtain $\mu_1\eta_1 =(x_{\ell+1}-d_1x_1)e_1$ and $\eta_{\ell-1}\mu_{\ell-1}=(x_{\ell+1}-c_\ell x_1)e_{\ell}$.
Then,
$$
\eta_1\mu_1\eta_1=(c_1-d_1)\eta_1\alpha, \quad \mu_1\eta_1\mu_1=(c_1-d_1)\alpha\mu_1,
$$
$$
\mu_{\ell-1}\eta_{\ell-1}\mu_{\ell-1}=(d_{\ell}-c_\ell)\mu_{\ell-1}\beta,
\quad
\eta_{\ell-1}\mu_{\ell-1}\eta_{\ell-1}=(d_{\ell}-c_\ell)\beta\eta_{\ell-1}.
$$
Since $t\neq (-1)^{\ell+1}$, both $c_1-d_1$ and $d_{\ell}-c_\ell$ are not zero. We then use $\alpha'=\mu_1\eta_1-(c_1-d_1)\alpha$ and $\beta'=\eta_{\ell-1}\mu_{\ell-1}-(d_{\ell}-c_\ell)\beta$ to replace $\alpha$ and $\beta$, respectively.
It gives
$$
\alpha'\mu_1=\eta_1\alpha'=\beta'\eta_{\ell-1}=\mu_{\ell-1}\beta'=0,
(\alpha')^k=(-1)^{k-1}(\mu_1\eta_1)^k, (\beta')^k=(-1)^{k-1}(\eta_{\ell-1}\mu_{\ell-1})^k.
$$
By comparing the dimensions, we find that $\mathcal{A}$ is isomorphic to the Brauer graph algebra whose Brauer graph has $\ell+1$ vertices and is as follows.
$$
\begin{xy}
(0,0) *[o]+[Fo]{k}="A",
(15,0)*[o]+[Fo]{k}="B",
(30,0)*[o]+[Fo]{k}="C",
(45,0)*[o]+[Fo]{k}="D",
(60,0)*[o]+[Fo]{k}="E",
\ar@{-} "A";"B"
\ar@{-} "B";"C"
\ar@{.} "C";"D"
\ar@{-} "D";"E"
\end{xy}
$$

By the categorification theorem, the number of non-isomorphic simple modules of $R^\Lambda(\delta)$ is equal to the size of the $\Lambda-\delta$ part in the highest weight crystal of $V(\Lambda)$.
In particular, the size is given by the number of Kleshchev multipartitions which admits a standard tableau whose residue sequence belongs to $I^{\delta}$.
Since the set of Kleshchev multipartitions (e.g.,\cite[Section 1.1.1]{F-ariki-koike-alg}) corresponding to $\delta$ is $\{(\emptyset,\ldots,\emptyset, (i,1^{e-i}))\mid 1\leq i\leq \ell\}$, we conclude that the number of pairwise non-isomorphic simple modules of $R^\Lambda(\delta)$ is $\ell$, the same as that of $\mathcal{A}$. Thus, $\mathcal{A}$ is the basic algebra of $R^{\Lambda}(\delta)$ and $R^{\Lambda}(\delta)$ is tame.
\end{proof}

We summarize the result for $R^{\Lambda}(\delta)$ as follows.

\begin{Prop}\label{prop::step-2-result-delta} 
Suppose $\Lambda\in \pcl$ with $k\geq 3$.
Then, $R^{\Lambda}(\delta)$ is tame if $\Lambda=k\Lambda_i$, $\ell=1$ with $t\neq \pm 2$ or $\ell\geq 2$ with $t\neq (-1)^{e}$.
Otherwise, $R^{\Lambda}(\delta)$ is wild.
\end{Prop}
\begin{proof}
If $|I(\Lambda)_0|\geq 2$, then $R^{\Lambda}(\delta)$ is wild by Lemma \ref{lem::step-1-21}. If $|I(\Lambda)_0|=1$, this is exactly the case of Lemma \ref{lem::step-2-delta-1} or Lemma \ref{lem::step-2-delta-2}.
\end{proof}

\subsection{Representation type on $\mathcal T(\Lambda)_0\cup \mathcal T(\Lambda)_1$}
Suppose $\Lambda\in \pcl$ with $k\geq 3$.
In the next proposition, we see that $R^{\Lambda}(\beta)$ is representation-finite if $\beta\in \mathcal T(\Lambda)_1=\{\alpha_i \mid i\in I(\Lambda)_1\}$.

\begin{Prop}\label{prop::step-2-result-T_1}
For any $\beta\in \{0\}\cup\mathcal T(\Lambda)_1$, $R^{\Lambda}(\beta)$ is of finite representation type.
Among them, $R^{\Lambda}(0)$ is a simple algebra.
\end{Prop}
\begin{proof}
If $\beta=0$, then $R^{\Lambda}(\beta)=\k$ and it is simple.
If $\beta=\alpha_i$ for some $i\in I(\Lambda)_1$, then $R^{\Lambda}(\beta)\simeq \k[x]/(x^m)$ with $m=\langle h_i, \Lambda \rangle \ge 2$, and it is of finite representation type which is not a simple algebra.
\end{proof}

We look at the case of $\mathcal T(\Lambda)_0$.
Recall that $T(\Lambda)_0=\{\Lambda_{i,j}\mid i\neq j\in I(\Lambda)_0, [i,j]\neq I \}$ and $\mathcal T(\Lambda)_0= \{\sum_{m\in [i,j]}\alpha_m \mid
 i\neq j\in I(\Lambda)_0, [i,j]\neq I \}$ with $|I(\Lambda)_0|\geq 2$.
We enumerate the indices in $I(\Lambda)_0$ by
$$
I(\Lambda)_0=\{i_j\mid 1\leq j\leq h, i_1<i_2<\ldots<i_h\}
$$
as in Subsection 4.4 and  we write $\Lambda=m_{i_1}\Lambda_{i_1}+m_{i_2}\Lambda_{i_2}+\ldots +m_{i_h}\Lambda_{i_h}$.

\begin{Lemma}\label{lem::step-2-T_0-2i2j}
Let $\Lambda=2\Lambda_i+2\Lambda_j$ with $i\neq j \in I$.
Then, $R^\Lambda(\beta )$ is wild for any $\beta\in \mathcal T(\Lambda)_0$.
\end{Lemma}
\begin{proof}
By Proposition \ref{prop::iso-sigma}, we can assume $\Lambda=2\Lambda_0+2\Lambda_s$ with $1\le s\leq \ell$.
In this case, $T(\Lambda)_0=\{\Lambda_{0,s}, \Lambda_{s,0}\}$, where $\Lambda_{0,s}$ (resp. $\Lambda_{s,0}$) appears only if $s<\ell$ (resp. $s>1$), and $\mathcal T(\Lambda)_0=\{\sum_{m=0}^s\alpha_m, \alpha_0+\sum_{m=s}^\ell\alpha_m\}$.
We only prove the case $\beta=\sum_{m=0}^s\alpha_m$ since the other case can be checked similarly.

Let $\mathcal{A}=e_0 R^\Lambda(\beta )e_0$ with $e_0=e(0~1~2~\ldots~s)$.
By Theorem \ref{theo::graded-dim}, $\dim_q \mathcal{A}= 1+2q^2+2q^4+q^6$.
On the other hand, $x_1^2e_0=0$ and $\psi_i e_0=0, 1\leq i\leq s-1$ imply $x_je_0=(-1)^{j-1}x_1e_0$ for any $2\leq j\leq s$. Moreover, $e_0\psi_w e_0\neq 0$ yields that the permutation $w\nu$ of $\nu=(0~1~2~\ldots~s)$ must coincide with $\nu$, and then, $w=1$. Hence, $\mathcal{A}$ is generated by $e_0, x_1e_0$ and $x_{s+1}e_0$. By the graded dimension of $\mathcal{A}$ and $x_1^2e_0=0$, the degree $2$ homogeneous component has basis $\{x_1e_0, x_{s+1}e_0\}$ and the degree $4$ homogeneous component has basis $\{x_1x_{s+1}e_0, x_{s+1}^2e_0\}$.
Then, $\mathcal{A}/\text{rad}^3 \mathcal{A}\simeq \k[X,Y]/(X^2,Y^3,XY^2)$ by sending $x_1e_0$ and $x_{s+1}e_0$ to $X$ and $Y$, respectively.
Since $\mathcal{A}/\text{rad}^3 \mathcal{A}$ is wild, $\mathcal{A}$ is wild and so is $R^\Lambda(\beta)$.
\end{proof}

Recall that $T(\Lambda)_0$ is defined only when $|I(\Lambda)_0|\ge 2$.
\begin{Defn}
Let $T'(\Lambda)_0:=\{ \Lambda_{i_j,i_{j+1}}\in T(\Lambda)_0\mid i_j,i_{j+1} \in I(\Lambda_0), 1\le j \le h\}$ and
$$
\mathcal T'(\Lambda)_0:=\left \{\textstyle \sum_{m\in[i_j,i_{j+1}]}\alpha_{m}\in \mathcal T(\Lambda)_0\mid i_j,i_{j+1} \in I(\Lambda_0), 1\le j \le h\right \}.
$$
Note that $T'(\Lambda)_0=T(\Lambda)_0$ if $|I(\Lambda)_0|=2$.
\end{Defn}

\begin{Lemma}\label{lem::step-2-T_0-exception}
Suppose $|I(\Lambda)_0|\geq3$ and $\beta\in \mathcal T(\Lambda)_0$.
Then, $R^{\Lambda}(\beta)$ is wild if $\beta\notin {\mathcal T'}(\Lambda)_0$.
\end{Lemma}
\begin{proof}
Suppose $k=3$.
By Proposition \ref{prop::iso-sigma}, we assume $\Lambda=\Lambda_0+\Lambda_i+\Lambda_j$ with $0<i<j \leq \ell$.
In this case, $T(\Lambda)_0\setminus T'(\Lambda)_0=\{\Lambda_{0,j}, \Lambda_{j,i},\Lambda_{i,0}\}$, where $\Lambda_{j,i}$ (resp. $\Lambda_{0,j}$, resp. $\Lambda_{i,0}$) appears only if $i\leq j-2$ (resp. $j<\ell$, resp. $i>1$).
If $\Lambda_{j,i}$ appears, then $X_{\Lambda_{j,i}}=(1^{i+1},0^{j-i-1}, 1^{\ell-j+1})$ by Lemma \ref{lem::recurrence}, and then $\beta_{\Lambda_{j,i}}=\alpha_0+\ldots+\alpha_i+\alpha_j+\ldots+\alpha_\ell$.

Set $e_1=e(0~1~2~\ldots~i~j~j+1~\ldots~\ell)$ and $e_2=e(j~j+1~\ldots~\ell~i~0~1~\ldots~i-1)$.
We define $\mathcal{A}=(e_1+e_2)R^\Lambda(\beta_{\Lambda_{j,i}})(e_1+e_2)$.
By Theorem \ref{theo::graded-dim}, we have
$$
\dim_q e_1\mathcal Ae_1=1+2q^2+q^4, \quad \dim_q e_1\mathcal Ae_2=\dim_q e_2\mathcal Ae_1=q^2,
$$
$$
\dim_q e_2\mathcal Ae_2=1+2q^2+q^4 \text{ if } i>1, \quad \dim_q e_2\mathcal Ae_2= 1+q^2+q^4 \text{ if } i=1.
$$
By \cite[Lemma 1.3]{Ar-rep-type}, the quiver of $\mathcal{A}$ has a subquiver containing two loops on vertex 1 and one arrow $1\rightarrow2$. Then, $\mathcal{A}$ is wild deduced from \cite[I.10.8.(i)]{Er-tame-block}, and so is $R^\Lambda(\beta_{\Lambda_{j,i}})$.
Since the choice of $i,j$ is arbitrary, one can use Proposition \ref{prop::iso-sigma} to deduce that $R^\Lambda(\beta_{\Lambda'})$ is wild for $\Lambda'\in\{\Lambda_{0,j},\Lambda_{i,0} \}$. This implies that the assertion holds when $k=3$.

Suppose $k\geq 4$.
Let $\Lambda_{i_j,i_l}\in T(\Lambda)_0\setminus T'(\Lambda)_0$ with $i_l\neq i_{j+1}$, and $\bar\Lambda:= \Lambda_{i_j}+\Lambda_{i_{j+1}}+ \Lambda_{i_l}$. Then, $R^{\bar\Lambda}(\beta_{\bar\Lambda, {\bar\Lambda}_{i_j,i_l}})$ is wild by the result for $k=3$.
Since $\beta_{\Lambda, \Lambda_{i_j,i_l}}=\beta_{\bar\Lambda,{\bar\Lambda}_{i_j,i_l}}$ by Lemma \ref{lem::embedding-Lambda}, we conclude that $R^{\Lambda}(\beta_{\Lambda_{i_j,i_l}})$ is wild by Lemma \ref{lem::reduction-level}.
\end{proof}

Recall the subsets $\mathscr T(\Lambda)_0$ and $\mathscr T(\Lambda)_1$ of $\mathcal T(\Lambda)_0$ in Subsection \ref{sec::main-result}.

\begin{Prop}\label{prop::step-2-result-T_0}
Suppose $\Lambda\in \pcl$ with $k\geq 2$ and $\beta\in \mathcal T(\Lambda)_0$. Then, $R^{\Lambda}(\beta)$ is
\begin{enumerate}
    \item of finite representation type if $\beta\in \mathscr T(\Lambda)_0$,
    \item of tame representation type if $\beta\in \mathscr T(\Lambda)_1$.
\end{enumerate}
Otherwise, $R^{\Lambda}(\beta)$ is of wild representation type.
\end{Prop}
\begin{proof}
By Lemma \ref{lem::step-2-T_0-exception}, $R^\Lambda(\beta)$ is wild for any $\beta\in \mathcal T(\Lambda)_0\setminus \mathcal T'(\Lambda)_0$.
In the following, we assume that $\beta=\sum_{m\in[i_j,i_{j+1}]}\alpha_m\in \mathcal T'(\Lambda)_0$.
Write $\Lambda=\bar\Lambda+\tilde\Lambda$ with $\bar\Lambda=m_{i_j}\Lambda_{i_j}+m_{i_{j+1}}\Lambda_{i_{j+1}}$.
Then, the surjective algebra homomorphism used in the proof of Lemma \ref{lem::reduction-level} induces an isomorphism
\begin{equation}\label{equ::step-2-alg-iso}
R^{\Lambda}(\beta)\simeq R^{\bar\Lambda}(\beta)
\end{equation}
because the graded dimensions coincide.
If $\beta\notin \mathscr T(\Lambda)_0\cup \mathscr T(\Lambda)_1$, then $m_{i_j}\geq 2$ and $m_{i_{j+1}}\geq 2$.
By Lemma \ref{lem::step-2-T_0-2i2j}, $R^{2\Lambda_{i_j}+2\Lambda_{i_{j+1}}}(\beta)$ is wild and so is $R^\Lambda(\beta)$ by Lemma \ref{lem::reduction-level}.

Suppose $\beta\in \mathscr T(\Lambda)_0\cup \mathscr T(\Lambda)_1$.
By Proposition \ref{prop::iso-sigma} and \eqref{equ::step-2-alg-iso}, we may assume that $\Lambda=(k-1)\Lambda_0+\Lambda_s$ for some $1\leq s\le \ell-1$, and either $\beta=\sum_{m=0}^s\alpha_m$ or $\beta=\alpha_0+\sum_{m=s}^\ell\alpha_m$. In this case, $k=2$ (resp. $k\ge 3$) if $\beta\in\mathscr T(\Lambda)_0$ (resp. $\beta\in\mathscr T(\Lambda)_1$).
We only show the case $\beta=\sum_{m=0}^s\alpha_m$, and the case $\alpha_0+\sum_{m=s}^\ell\alpha_m$ is left to the reader.
Let $\nu_i=(0~1~2~\ldots~i-1~s~s-1~\ldots~i+1~i)$ and denote $e_i=e(\nu_i)$, for $0\leq i\leq s$.
We understand $\nu_0=(s~s-1~\ldots~1~0)$.
For $e'=\sum_{i=0}^se_i$, we define $\mathcal{A}=e'R^\Lambda(\beta)e'$.
It is not difficult to check that
$$
\begin{aligned}
\dim_q e_i \mathcal{A} e_j&=0, \ |i-j|>1, \quad \dim_q e_0 \mathcal{A} e_0 =1+\ssum_{1\leq j\leq k-1}q^{2j},\\
\dim_q e_i \mathcal{A} e_i &=1+2\ssum_{1\leq j\leq k-2}q^{2j}+q^{2k-2}, \quad 1\leq i\leq s,\\
\dim_{q} e_i \mathcal{A} e_{i+1}&=\dim_q e_{i+1} \mathcal{A} e_i=\ssum_{1\leq j\leq k-1}q^{2j-1}, \quad 0\leq i\le s-1.
\end{aligned}
$$
These graded dimensions imply that $\{\mathcal{A}e_i\mid 0\le i\le s\}$ is a complete set of isomorphism classes of indecomposable projective $\mathcal{A}$-modules. We show that $\mathcal{A}$ is a Brauer graph algebra. To see this, we begin by giving an explicit basis of $\mathcal{A}$.
\begin{itemize}
    \item We show that $e_0\mathcal{A}e_0$ has basis $\{x_{s+1}^ae_0\mid 0\leq a\leq k-1\}$. In fact, $x_1e_0=\psi_ie_0=0$ and $\psi_i^2e_0=(x_i+x_{i+1})e_0$ for $1\leq i\leq s-1$ imply $x_je_0=0$ for any $2\le j \le s$. Then, $\psi_s^2e_0=x_{s+1}e_0$ and $x_{s+1}^{k-1}\psi^2_se_0=0$ imply $x_{s+1}^ke_0=0$. Lastly, $e_0\psi_we_0\neq 0$ yields $\psi_w=1$ as usual. Here, the same strategy as in \eqref{equ::step-2-basis-computation} is used to find $x_{s+1}^{k-1}\psi^2_se_0=0$ and $x_{s+1}\psi_s^2e_s=0$, which we will use below. We omit the details.

    \item We give a basis for other $e_i\mathcal{A}e_j$ when $i>0$ or $j>0$. Note that $x_1^{k-1}e_i=0$ for any $1\leq i\leq s$. Since $\psi_i^2 e_i=e_i$, we find $x_{i+1}e_i=\psi_{i}\cdots \psi_1 x_1e(s_1s_2\cdots s_i \nu_i)\psi_1\cdots \psi_i=0$ with $s_1, s_2, \ldots, s_i\in \mathfrak S_{s+1}$ and $1\leq i\leq s-1$. For any $1\leq i\leq s, 1\le j \le s-1$, $\psi_j e_i=0$ holds if $j\neq i$, and then $(x_j+x_{j+1})e_i=0$, for $j\neq i$, implies
    $$
    x_he_i=\left\{\begin{array}{cc}
    (-1)^{h-1}x_1e_i  & \text{ if } 2\leq h\leq i, \\
    0    & \text{ if } i+1\leq h\leq s.
    \end{array}\right.
    $$
    Using $x_{s+1}\psi_s^2 e_s=0$ and $x_{s+1}\psi_s^2e_i=(-1)^i x_1\psi_s^2e_i$ for $1\leq i\le s-1$, this together with $\psi_s^2e_i=(x_s+x_{s+1})e_i$ yield $x_{s+1}^2e_i=(-1)^ix_1x_{s+1}e_i$ for $1\leq i\leq s$.
    Similar to the proof of Lemma \ref{lem::step-2-delta-2}, we conclude that $e_i \mathcal{A}e_i$, for $1\leq i\leq s$, has basis $\{x_1^ax_{s+1}^be_i\mid 0\leq a\leq k-2, 0\leq b\leq 1 \}$. We know that $e_0\mathcal{A}e_1$ has basis $\{x_{s+1}^a\psi_s\cdots\psi_1e_1\mid 0\leq a\leq k-2\}$ by using $e_0\mathcal{A}e_0$ being spanned by $\{x_{s+1}^ae_0\mid 0\leq a\leq k-1\}$ and $\dim_q e_0\mathcal{A}e_1$. Moreover, $e_i\mathcal{A}e_{i+1}$, for $1\leq i\leq s-1$, has basis $\{x_1^a\psi_s\cdots\psi_{i+1}e_{i+1}\mid 0\leq a\leq k-2 \}$, and $e_j \mathcal{A}e_{j-1}$, for $1\leq j\leq s$, has basis $\{x_1^a\psi_{j}\cdots\psi_s e_{j-1}\mid 0\leq a\leq k-2 \}$. Note that $x_{s+1}\psi_s\cdots \psi_{i+1}e_{i+1}=\psi_s\cdots \psi_{s+1}x_{i+1}e_{i+1}=(-1)^ix_1\psi_s\cdots \psi_{i+1}e_{i+1}$ for $1\leq i\leq s-1$.
\end{itemize}

Let $\alpha=x_1e_s$, $\mu_i=\psi_s\cdots \psi_i e_i$ and $\eta_i=\psi_i\cdots \psi_s e_{i-1}$ for $1\leq i\leq s$.
Then, $\alpha^{k-1}=0$.
Since $\psi_{s-1}e_{i-1}=e_{i-1}\psi_{s-1}=0$, for $1\leq i\leq s-1$, we have
$$
\eta_{i+1}\eta_i=\psi_{i+1}\cdots\psi_{s-1}\psi_i\cdots \psi_{s-1}\psi_{s}\psi_{s-1}e_{i-1}=0,
$$
$$
\mu_i\mu_{i+1}=e_{i-1}\psi_{s-1}\psi_{s}\psi_{s-1}\cdots \psi_i\psi_{s-1}\cdots\psi_{i+1}=0.
$$
By direct calculation, we obtain $\eta_{i-1}\mu_{i-1}=((-1)^ix_1+x_{s+1})e_{i-1}$ for $2\leq i\leq s$ and $\mu_i\eta_i=x_{s+1} e_{i-1}$ for $1\leq i\leq s$, which gives
$$
(\mu_i\eta_i)^{k-1}=(-1)^k(\eta_{i-1}\mu_{i-1})^{k-1}
$$
for  $2\leq i\leq s$, and $(\mu_1\eta_1)^{k-1}\mu_1=0$.
Moreover, $\eta_s\mu_s=((-1)^{s-1}x_1+x_{s+1})e_s$ implies $\eta_s\mu_s\eta_s=(-1)^{s-1}\alpha\eta_s$ and $\mu_s\eta_s\mu_s=(-1)^{s-1}\mu_s\alpha$.
After replacing $\alpha$ by $\alpha'=\eta_s\mu_s+(-1)^s\alpha$, we obtain
$$
\alpha'\eta_s=\mu_s\alpha'=0, \quad (\alpha')^{k-1}=(-1)^k(\eta_s\mu_s)^{k-1}.
$$
We conclude that $\mathcal{A}$ is isomorphic to the Brauer graph algebra whose Brauer graph is
$$
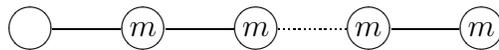
\begin{xy}
(0,0) *[o]+[Fo]{\hphantom{m}}="A",
(15,0) *[o]+[Fo]{m}="B",
(30,0)*[o]+[Fo]{m}="C",
(45,0)*[o]+[Fo]{m}="D",
(60,0)*[o]+[Fo]{m}="E",
\ar@{-} "A";"B"
\ar@{-} "B";"C"
\ar@{.} "C";"D"
\ar@{-} "D";"E"
\end{xy}
$$
where $m=k-1$ and the number of vertices is $s+2$. In particular, if $k=2$, then $\mathcal{A}$ is actually a Brauer tree algebra.

Similar to the proof of Lemma \ref{lem::step-2-delta-2}, the set of Kleshchev multipartitions which belong to the block $R^\Lambda(\beta)$ is $\{(\emptyset,\ldots,\emptyset,(i),(1^{s+1-i}))\mid 0\leq i\leq s\}$.
This implies that $R^\Lambda(\beta)$ shares the same number of pairwise non-isomorphic simple modules with $\mathcal{A}$.
Thus, $\mathcal{A}$ is the basic algebra of $R^\Lambda(\beta)$. Then, $R^\Lambda(\beta)$ is of finite (resp. tame) representation type if $\beta\in\mathscr T(\Lambda)_0$ (resp. $\beta\in\mathscr T(\Lambda)_1$).
\end{proof}

\subsection{Representation type on $\mathcal T(\Lambda)_2$}
In this subsection, we assume $\ell\geq 3$ and $|I(\Lambda)_1|\geq 1$.
Recall that $T(\Lambda)_2=\{\Lambda'_{i-1,i+1}\mid \Lambda'=\Lambda_{i,i},i\in I(\Lambda)_1\}$ is defined only when $\ell\geq 3$, and
$$
\mathcal T(\Lambda)_2=\{2\alpha_i+\alpha_{i-1}+\alpha_{i+1} \mid i\in I(\Lambda)_1\}.
$$

\begin{Lemma}\label{lem::step-2-T_2-3i}
Let $\Lambda=3\Lambda_i$ for some $i\in I$. 
Then, $R^{\Lambda}(\beta)$ is of wild representation type for any $\beta\in \mathcal T(\Lambda)_2$.
\end{Lemma}
\begin{proof}
By Proposition \ref{prop::iso-sigma}, $R^{\Lambda}(\beta)\simeq R^{3\Lambda_0}( 2\alpha_0+\alpha_1+\alpha_\ell)$.
Set $\mathcal{A}=e_0 R^{3\Lambda_0}(2\alpha_0+\alpha_1+\alpha_\ell) e_0$ with $e_0=e(010\ell)$.
It is not difficult to find that $\dim_q \mathcal{A}=1+3q^2+4q^4+3q^6+q^8$.
Then, $\mathcal{A}$ is a wild local algebra since $\dim \text{rad}\ \mathcal{A}/\text{rad}^2 \mathcal{A} \ge 3$, i.e., its quiver has at least $3$ loops.
Hence, $R^{\Lambda}(\beta)$ is also wild.
\end{proof}

\begin{Lemma}\label{lem::step-2-T_2-2ij}
Let $\Lambda=2\Lambda_i+\Lambda_j$ with $i\neq j\in I$ and $\beta\in \mathcal T(\Lambda)_2$. Then, $R^{\Lambda}(\beta)$ is of tame (resp. wild) representation type if $\ch\k \neq 2$ and $j\not\equiv_e i\pm 1$ (resp. otherwise).
\end{Lemma}
\begin{proof}
We have $R^{\Lambda}(\beta)\simeq \mathcal{A}:= R^{2\Lambda_0+\Lambda_s}( 2\alpha_0+\alpha_1+\alpha_\ell)$, where $s\equiv_e j-i$.
If $j\equiv_e i+1$, i.e., $s=1$, then set $e_0=e(0 \ell 1 0)$ and we find $\dim_q e_0 \mathcal{A}e_0=1+3q^2+3q^4+q^6$, which implies that $e_0 \mathcal{A}e_0$ is a wild local algebra whose quiver has at least $3$ loops. Similarly, the case for $j\equiv_e i-1$ can be checked.

Suppose $j\not\equiv_e i\pm 1$, i.e., $ s\neq 1,\ell$.
By comparing the dimensions, we see that the surjective algebra homomorphism from $\mathcal{A}$ to $R^{2\Lambda_0}(2\alpha_0+\alpha_1+\alpha_\ell)$ is actually an isomorphism.
Then, the assertion follows from \cite[Proposition 11.4]{Ar-rep-type} that $R^{2\Lambda_0}(2\alpha_0+\alpha_1+\alpha_\ell)$ is wild if $\ch\k=2$ and 
is tame if $\ch\k\neq2$, and in the latter case, it is Morita equivalent to the Brauer graph algebra whose Brauer graph is
$$
\begin{xy}
(0,0) *[o]+[Fo]{2}="A", (15,0) *[o]+[Fo]{2}="B", (30,0) *[o]+[Fo]{2}="C",
\ar@{-} "A";"B"
\ar@{-} "B";"C"
\end{xy}\ .
$$
\end{proof}

Recall that $I(\Lambda)_0=\{i_j\mid 1\leq j\leq h, i_1<i_2<\ldots<i_h\}$ and $\Lambda=\sum_{1\leq j\leq h}m_{i_j}\Lambda_{i_j}$.
Recall the subset $\mathscr T(\Lambda)_2$ of $\mathcal T(\Lambda)_2$ in Subsection \ref{sec::main-result}.

\begin{Prop}\label{prop::step-2-result-T_2}
For any $\Lambda\in \pcl$ with $k\geq 3$ and $\beta\in \mathcal T(\Lambda)_2$,
$R^{\Lambda}(\beta)$ is of tame (resp. wild) representation type if $\beta\in \mathscr T(\Lambda)_2$ (resp. otherwise).
\end{Prop}
\begin{proof}
Suppose $\beta=2\alpha_{i_j}+\alpha_{i_j-1}+\alpha_{i_j+1}\in \mathcal  T(\Lambda)_2$ for some $i_j\in I(\Lambda)_1$.
If $m_{i_j}\geq 3$, then $R^{3\Lambda_{i_j}}(\beta)$ is wild by Lemma \ref{lem::step-2-T_2-3i} and so is $R^{\Lambda}(\beta)$ by Lemma \ref{lem::reduction-level}.
If $m_{i_j}=2$, $i_{j-1}\not\equiv_e i_j-1$ and $i_{j+1}\not\equiv_e i_j+1$, then the surjective algebra homomorphism induces an isomorphism $R^{\Lambda}(\beta)\simeq R^{2\Lambda_{i_j}}(\beta)$ because the graded dimensions coincide.
Then, it is shown in \cite[Proposition 11.4]{Ar-rep-type} that $R^{2\Lambda_{i_j}}(\beta)\simeq R^{2\Lambda_0}(2\alpha_0+\alpha_1+\alpha_\ell)$ is tame if $\ch \k\neq 2$ and wild if $\ch \k=2$.
If $m_{i_j}=2$, $i_{j-1}\equiv_e i_j-1$ or $i_{j+1}\equiv_e i_j+1$, then $R^{2\Lambda_{i_j}+\Lambda_{i_{j-1}} }(\beta)$ or $ R^{2\Lambda_{i_j}+\Lambda_{i_{j+1}} }(\beta)$ is wild by Lemma \ref{lem::step-2-T_2-2ij} and so is $R^{\Lambda}(\beta)$ by Lemma \ref{lem::reduction-level}.
\end{proof}

\subsection{Representation type on $\mathcal T(\Lambda)_3$}
In this subsection, we assume $\ell\geq 2$ and $|I(\Lambda)_2|\geq 1$.
Recall that $T(\Lambda)_3=\{\Lambda'_{i, i+1}, \Lambda'_{i-1, i}\mid \Lambda'=\Lambda_{i,i}, i\in I(\Lambda)_2\}$ and
$$
\mathcal T(\Lambda)_3=\{2\alpha_i+\alpha_{i-1}, 2\alpha_i+\alpha_{i+1} \mid i\in I(\Lambda)_2\}.
$$

\begin{Lemma}\label{lem::step-2-T_3-4i}
Let $\Lambda=4\Lambda_i$ for some $i\in I$. Then, $R^{\Lambda}(\beta)$ is of wild representation type for $\beta\in \mathcal T(\Lambda)_3$.
\end{Lemma}
\begin{proof}
It suffices to assume $i=0$ by Proposition \ref{prop::iso-sigma}.

Suppose $\beta=2\alpha_0+\alpha_1$.
We consider $\mathcal{A}=e_0R^{\Lambda}(\beta)e_0$ with $e_0=e(0 1 0)$. Then,
$$
\dim_q \mathcal{A}= 1+2q^2+3q^4+3q^6+2q^8+q^{10}.
$$
We show that $\mathcal{A}$ has basis $\{x_1^ax_3^be_0\mid 0\leq a\leq 3, 0\leq b\leq 2 \}$.
First, $e_0\psi_we_0\neq 0$ implies $\psi_w=1$ or $\psi_w=\psi_1\psi_2\psi_1$, but the latter one does not occur by $\psi_1e_0=0$.
Second, we have $x_1^4e_0=0$ and $x_2e_0=-x_1e_0$.
Third, $x_1^4e(001)=0$ implies
$$
\begin{aligned}
0&=\psi_1x_1e(001)x_1^3 =(x_2\psi_1-1)e(001)x_1^3=x_2(x_2\psi_1-1)e(001)x_1^2-x_1^3e(001)\\
&=\ldots=-(x_1^3+x_1^2x_2+x_1x_2^2+x_2^3)e(001)+x_2^4\psi_1e(001),
\end{aligned}
$$
and then,
$$
\begin{aligned}
x_3^3e_0
&=x_3^3\psi_2\psi_1\psi_2e_0=\psi_2x_2^3e(001)\psi_1\psi_2\\
&=-\psi_2(x_1^3+x_1^2x_2+x_1x_2^2)\psi_1\psi_2e_0=-(x_1^3+x_1^2x_3+x_1x_3^2)e_0.
\end{aligned}
$$
Lastly, we obtain the desired basis of $\mathcal{A}$ by using $\dim_q \mathcal{A}$. Observe that $\mathcal{A}$ has a wild quotient algebra $\k[X,Y]/(X^2,Y^3,Y^2X)$ by sending $x_3e_0$ and $x_1e_0$ to $X$ and $Y$, respectively.
Hence, $R^{\Lambda}(\beta)$ is wild.

Similarly, one can check that $R^{4\Lambda_0}( 2\alpha_{0}+\alpha_{\ell})$ has a wild idempotent truncation.
\end{proof}

\begin{Lemma}\label{lem::step-2-T_3-3ij}
Suppose $\Lambda=3\Lambda_i+\Lambda_j$ with $i\neq j\in I$. Then,
\begin{itemize}
\item $R^{\Lambda}(2\alpha_i+\alpha_{i+1})$ is wild if $j\equiv_e i+ 1$.
\item $R^{\Lambda}(2\alpha_i+\alpha_{i-1})$ is wild if $j\equiv_e i-1$.
\end{itemize}
\end{Lemma}
\begin{proof}
We only show the case for $i=0$ and $j=1$ and one may check the case for $i=0$ and $j=\ell$ in a similar way.
Let $e_0=e(0 1 0)$ and $\mathcal{A}=e_0R^\Lambda(2\alpha_0+\alpha_1)e_0$.
Then,
$$
\dim_q \mathcal{A}=2+5q^2+7q^4+5q^6+2q^8.
$$
The degree 0 homogeneous component of $\mathcal{A}$ has basis $\{e_0, \psi_1\psi_2\psi_1e_0\}$.
Since $(\psi_1\psi_2\psi_1)^2e_0=-\psi_1\psi_2\psi_1e_0$, $e_0$ can be decomposed into a sum $e_0=e_1+e_2$ of two orthogonal primitive idempotents, namely, $e_1=-\psi_1\psi_2\psi_1e_0$ and $e_2=e_0+\psi_1\psi_2\psi_1e_0$.
In fact, if $\mathcal Ae_1\simeq \mathcal Ae_2$, then the dimension of $\mathcal{A}$ must be even and this contradicts with $\dim \mathcal{A}=21$.
Thus, $\mathcal{A}$ is a basic algebra.
Let $\mathcal{J}$ be the two-sided ideal generated by all positive degree elements of $\mathcal{J}$.
Then, $\mathcal{J}$ is nilpotent and $\mathcal{A}/\mathcal{J}\simeq \k\times \k$.
We deduce that $\mathcal{J}=\text{rad}\ \mathcal{A}$ and $\dim \mathcal{J}/\mathcal{J}^2\geq 5$.

If $\mathcal{A}$ is connected, then its quiver has at least $5$ arrows (or loops) which is not a quiver of tame two-point algebras (see \cite[IV.2.2]{Er-tame-block}).
If $\mathcal{A}$ is not connected, then its quiver has at least $3$ loops on one of the vertices, and $\mathcal{A}$ is also wild.
\end{proof}

Recall the subset $\mathscr T(\Lambda)_3$ of $\mathcal T(\Lambda)_3$ in Subsection \ref{sec::main-result}.

\begin{Prop}\label{prop::step-2-result-T_3}
For any $\Lambda\in \pcl$ with $k\geq 3$ and $\beta\in \mathcal T(\Lambda)_3$, $R^{\Lambda}(\beta)$ is of tame representation type if $\beta\in \mathscr T(\Lambda)_3$ and is of wild representation type otherwise.
\end{Prop}
\begin{proof}
Suppose $\Lambda=\sum_{1\leq j\leq h}m_{i_j}\Lambda_{i_j}$ and $\beta=2\alpha_{i_j}+\alpha_{i_j+1}$ for some $i_j\in I(\Lambda)_2$.
If $m_{i_j}\geq 4$, then $R^{\Lambda}(\beta)$ is wild since $R^{4\Lambda_{i_j}}(\beta)$ is wild by Lemma \ref{lem::step-2-T_3-4i}.
If $m_{i_j}=3$ and $i_{j+1}\equiv_e i_j+1$, then $R^{\Lambda}(\beta)$ is also wild since $R^{3\Lambda_{i_j}+\Lambda_{i_{j}+1}}(\beta)$ is wild by Lemma \ref{lem::step-2-T_3-3ij}.

Let $m_{i_j}=3$ and $i_{j+1}\not\equiv_e i_j+1$.
Then, $R^{\Lambda}(\beta)\simeq R^{3\Lambda_{i_j}}(\beta) \simeq R^{3\Lambda_0}( 2\alpha_0+\alpha_1)$.
We define $\mathcal{A}=e_0R^{3\Lambda_0}( 2\alpha_0+\alpha_1)e_0$ with $e_0=e(0 1 0)$. Then, $\dim_q \mathcal{A}=1+2q^2+2q^4+q^6$.
Moreover, $\mathcal{A}$ has basis $\{x_1^ax_3^be_0\mid 0\leq a\leq 2, 0\leq b\leq 1 \}$.
In fact, we can choose $\psi_w=1$ if $e_0\psi_we_0\neq 0$ by $\psi_1e_0=0$.
Then, $\mathcal{A}$ is generated by $x_1e_0$ and $x_3e_0$ since $(x_1+x_2)e_0=\psi_1^2e_0=0$.
Similar to the proof of Lemma \ref{lem::step-2-T_3-4i}, we may deduce $x_3^2e_0=-(x_1^2+x_1x_3)e_0$ from $x_1^3e(001)=0$.
We conclude that $\mathcal{A}$ is isomorphic to the local algebra $\k[X,Y]/(X^3,X^2+XY+Y^2,Y^3)$ by sending $x_1e_0, x_3e_0$ to $X$ and $Y$, respectively.
Lastly, we know by the categorification theorem that the number of simple modules of $R^{3\Lambda_0}( 2\alpha_0+\alpha_1)$ is $1$, and hence, $\mathcal{A}$ is the basic algebra of $R^{3\Lambda_0}( 2\alpha_0+\alpha_1)$.

If $\ch\k=3$, then $(X-Y)^2=-3XY=0$.
We choose $X$ and $Z:=X-Y$ as the new generators to obtain $\mathcal{A}\simeq \k[X,Z]/(X^3, Z^2)$.
It is easy to see that $\mathcal{A}$ has $\k[X,Z]/(X^3,Z^2,X^2Z)$ as a wild quotient algebra, and hence, $R^{3\Lambda_0}( 2\alpha_0+\alpha_1)$ is wild.

If $\ch\k\neq 3$, then $z^2+z+1=0$ has two different roots $\omega, \omega^{-1}$.
We choose $u=X-\omega Y$ and $v=X-\omega^{-1}Y$ as the new generators. Then, $\mathcal{A}\simeq \k[u,v]/(u^3-v^3,uv)$, which is a tame local algebra in \cite[Theorem III.1(a)]{EN-rep-type-Hecke}.
Thus, $R^{3\Lambda_0}( 2\alpha_0+\alpha_1)$ is tame.

We omit the details for $\beta=2\alpha_{i_j}+\alpha_{i_j-1}$ since the proof is similar to the above.
\end{proof}

\subsection{Representation type of $\mathcal T(\Lambda)_4$}
In this subsection, we assume $|I(\Lambda)_3|\geq 1$.
Recall that $T(\Lambda)_4=\{\Lambda'_{i, i}\mid \Lambda'=\Lambda_{i,i}, i\in I(\Lambda)_3\}$ and $\mathcal T(\Lambda)_4=\{2\alpha_i\mid i\in I(\Lambda)_3\}$.

\begin{Lemma}\label{lem::step-2-T_4-5i}
Let $\Lambda=5\Lambda_i$ for some $i\in I$.
Then, $R^\Lambda(2\alpha_i)$ is of wild representation type.
\end{Lemma}
\begin{proof}
We only show the case $i=0$ as usual.
Note that $R^{5\Lambda_0}(2\alpha_0)$ is the nilHecke algebra generated by $x_1,x_2,\psi$ subject to the relations
$$
x_1^5=\psi^2=0, \quad x_1x_2=x_2x_1, \quad x_1\psi-\psi x_2=\psi x_1-x_2\psi=-1.
$$
Let $e=-\psi x_1$. Then, $\mathcal{A}=eR^{5\Lambda_0}(2\alpha_0)e$ is the basic algebra of $R^{5\Lambda_0}(2\alpha_0)$.
Since $\psi e=0$, $\mathcal{A}$ is spanned by $ex_1^ax_2^be$, for $a,b\ge0$.
However, if we set $\tau=\psi(x_1-x_2)+1$, then $\tau^2=1$ and $\tau x_1\tau=x_2$ hold, which imply $x_2^5=0$.
Define $X=ex_1e$. Then,
$$
X=\psi x_1^2\psi x_1=(x_2\psi-1)x_1\psi x_1=x_2(x_2\psi-1)\psi x_1+x_1e=(x_1+x_2)e.
$$
Since $ex_2e=0$, we define $Y=ex_2^2e$. Then,
$$
Y=\psi x_1x_2^2\psi x_1=x_1x_2(x_1\psi+1)\psi x_1=-x_1x_2e.
$$
It is clear that $XY=X(-x_1x_2e)=(-x_1x_2)Xe=YX$.
Using
\begin{center}
$ex_1^ae=\psi(x_1^{a+1})\psi x_1=(x_2\psi-1)x_1^{a+1}\psi x_1=\ldots=-(x_1^{a+1}+x_1^ax_2+\ldots +x_2^{a+1})e$,
\end{center}
\begin{center}
$ex_2^be=x_1x_2\psi x_2^{b-1}\psi x_1=x_1x_2(x_1\psi+1)x_2^{b-2}\psi x_1=\ldots=(x_1x_2^{b-1}+x_1^2x_2^{b-2}+\ldots +x_1^{b-1}x_2)e$,
\end{center}
we see that $\mathcal{A}$ is generated by $X$ and $Y$.

Now, $ex_1^4e=\psi x_1^5\psi x_1=0$ implies $X^4+3X^2Y+Y^2=0$, and $ex_2^5e=0$ implies $X^3Y+2XY^2=0$, i.e., $2X^5+5X^3Y=0$. We compute
$$
\dim_q R^{5\Lambda_0}(2\alpha_0)=(q+q^{-1})^2(1+q^2+2q^4+2q^6+2q^8+q^{10}+q^{12}).
$$
Hence, $\dim_q \mathcal{A}=1+q^2+2q^4+2q^6+2q^8+q^{10}+q^{12}$ and we know that
$$
\left \{ e, X, X^2, Y, X^3, XY, X^4, X^2Y, X^3Y (\text{or } X^5), X^4Y \right \}
$$
is a basis of $\mathcal{A}$.
It follows that $\k[X,Y]/(X^3, X^2Y, Y^2)$ is a quotient algebra of $\mathcal{A}$,
which is of wild representation type by \cite[(1.1)(c)]{Ringel-local-alg}.
\end{proof}

\begin{Lemma}\label{lem::step-2-T_4-4i}
Suppose $\Lambda=4\Lambda_i$ for some $i\in I$.
Then, $R^\Lambda(2\alpha_i)$ is of tame (resp. wild) representation type if $\ch \k\neq 2$ (resp. otherwise).
\end{Lemma}
\begin{proof}
It suffices to consider $i=0$.
Similar to the previous lemma, we set $e=-\psi x_1$, $\mathcal{A}=eR^{4\Lambda_0}(2\alpha_0)e$, $X=ex_1e$ and $Y=ex_2^2e$.
We obtain $X^3+2XY=0$ and $X^2Y+Y^2=0$.
Then, $X^3Y=-2XY^2=2X^3Y$ implies $X^3Y=0$. The graded dimensions are
$$
\dim_q R^{4\Lambda_0}(2\alpha_0)=(q+q^{-1})^2(1+q^2+2q^4+q^6+q^8)
$$
and $\dim_q \mathcal{A}=1+q^2+2q^4+q^6+q^8$.
Therefore,
$$
\left \{ e, X, X^2, Y, X^3 (\text{or } XY), X^4\right \}
$$
is a basis of $\mathcal{A}$.
We conclude that $\mathcal{A}\simeq \k[X,Y]/(X^3+2XY, X^2Y+Y^2)$. If $\ch\k=2$, then $2XY=0$ and $\mathcal{A}$ has $\mathcal{B}=\k[X,Y]/(X^3, X^2Y, Y^2)$ as a quotient algebra. Since $\mathcal{B}$ is wild by \cite[(1.1) (c)]{Ringel-local-alg}, $R^{4\Lambda_0}(2\alpha_0)$ is of wild representation type.

Suppose $\ch\k\neq 2$.
If we choose $X$ and $Z=2Y+X^2$ as the new generators, we obtain
$$
\mathcal{A}\simeq \k[X,Z]/(XZ, X^4-Z^2).
$$
Hence, $\mathcal{A}$ is tame by \cite[Theorem III.1 (a)]{EN-rep-type-Hecke}.
\end{proof}

Recall the subset $\mathscr T(\Lambda)_4$ of $\mathcal T(\Lambda)_4$ in Subsection \ref{sec::main-result}.

\begin{Prop}\label{prop::step-2-result-T_4}
For any $\Lambda\in \pcl$ with $k\geq 3$ and $\beta\in \mathcal T(\Lambda)_4$, $R^{\Lambda}(\beta)$ is of tame representation type if $\beta\in \mathscr T(\Lambda)_4$ and of wild representation type otherwise.
\end{Prop}
\begin{proof}
Write $\Lambda=\sum_{j\in I}m_j\Lambda_j$ and $\beta=2\alpha_{i}$ for some $i\in I(\Lambda)_3$.
Then, $R^\Lambda(\beta)\simeq R^{m_i\Lambda_i}(\beta)$ since the graded dimensions coincide.
If $m_{i}\geq 5$, then $R^\Lambda(\beta)$ is wild since $R^{5\Lambda_{i}}(\beta)$ is wild by Lemma \ref{lem::step-2-T_4-5i}.
If $m_{i}=4$, then by Lemma \ref{lem::step-2-T_4-4i}, $R^{\Lambda}(\beta)$ is tame if $\ch\k \neq 2$ and wild otherwise.
\end{proof}

\subsection{Representation type of $\mathcal T(\Lambda)_5$}
In this subsection, we assume that $\ell\geq 2$ and $|I(\Lambda)_1|\geq 2$.
Recall that $T(\Lambda)_5=\{\Lambda'_{j, j}\mid i\neq j\in I(\Lambda)_1,\Lambda'=\Lambda_{i,i}\}$ and $\mathcal T(\Lambda)_5=\{\alpha_i+\alpha_j \mid i\neq j\in I(\Lambda)_1\}$.
Note again that $T(\Lambda)_5$ is defined only when $\ell\geq 2$.

\begin{Lemma}\label{lem::step-2-T_5-3i2j}
Suppose $\Lambda=3\Lambda_i+2\Lambda_j$ with $i\neq j\in I$. Then,
$R^\Lambda(\alpha_i+\alpha_j)$ is wild.
\end{Lemma}
\begin{proof}
We may assume $\Lambda=3\Lambda_0+2\Lambda_h$ with $h\equiv_e j-i$ such that $h\neq 0$.
If $h=1$ or $\ell$, then $\alpha_0+\alpha_h\in \mathcal T(3\Lambda_0+2\Lambda_h)_0$ and $R^{\Lambda}(\alpha_0+\alpha_h)$ is wild by Proposition \ref{prop::step-2-result-T_0}.

If $h\neq 1, \ell$, then set $\mathcal{A}=e_0R^{\Lambda}(\alpha_0+\alpha_h)e_0$ with $e_0=e(0h)$.
We compute $\dim_q \mathcal{A}=1+2q^2+2q^4+q^6$.
Moreover, $e_0\psi_1e_0=x_1^3e_0=0$, and $\psi_1^2 e_0=e_0$ implies $x_2^2e_0=x_2^2\psi_1^2e_0=\psi_1x_1^2e(h0)\psi_1=0$.
Thus, $\mathcal{A}$ has a basis $\{x_1^ax_2^be_0\mid 0\leq a\leq 2, 0\leq b\leq 1 \}$ and has $\k[X,Y]/(X^3,X^2Y,Y^2)$ as a wild quotient algebra.
Hence, $R^\Lambda(\alpha_0+\alpha_h)$ is wild.
\end{proof}

Recall the subset $\mathscr T(\Lambda)_5$ of $\mathcal T(\Lambda)_5$ in Subsection \ref{sec::main-result}.

\begin{Prop}\label{prop::step-2-result-T_5}
Suppose $\Lambda\in \pcl$ with $k\geq 3$ and $\beta \in \mathcal T(\Lambda)_5$.
Then, $R^{\Lambda}(\beta)$ is of tame (resp. wild) representation type if $\beta\in \mathscr T(\Lambda)_5$ (resp. otherwise).
\end{Prop}
\begin{proof}
Write $\Lambda=\sum_{s\in I}m_s\Lambda_s$ and $\beta=\alpha_i+\alpha_j$
for some $i\neq j\in I(\Lambda)_1$.
Then, $R^{\Lambda}(\beta)\simeq R^{\Lambda'}(\beta)$ with $\Lambda'=m_i\Lambda_i+m_j\Lambda_j$ since the graded dimensions coincide.
If $m_{i}\ge 3$ or $m_{j}\ge 3$, then $R^{\Lambda}(\beta)$ is wild by Lemma \ref{lem::step-2-T_5-3i2j}.

Suppose $m_{i}=m_{j}=2$. If $\beta\notin \mathscr T(\Lambda)_5$, i.e., $i\equiv_e j\pm1 $, then $\beta=\beta_{\Lambda_{i,j}}$ or $\beta_{\Lambda_{j,i}}\in \mathcal T(\Lambda)_0$ and hence, $R^{\Lambda}(\beta)$ is wild by Proposition \ref{prop::step-2-result-T_0}.
If $\beta\in \mathscr T(\Lambda)_5$, i.e., $i\not\equiv_e j\pm1 $, then $R^{\bar\Lambda}(\beta)\simeq R^{2\Lambda_0+2\Lambda_s}(\alpha_0+\alpha_1)$ by
Proposition \ref{prop::iso-sigma}, where $s\equiv_e j-i\not \equiv_e 0, \pm 1$.
Set $\mathcal{A}=e(0s) R^{2\Lambda_0+2\Lambda_s}(\alpha_0+\alpha_1) e(0s)$.
Similar computation as in the previous lemma shows that $\mathcal{A}$ has basis $\{x_1^ax_2^be(0s)\mid  0\leq a\leq 1,0\leq b\leq 1\}$ with relations $x_1^2e(0s)=x_2^2e(0s)=0$, which is isomorphic to the tame algebra $\k[X,Y]/( X^2, Y^2)$.
Moreover, $\mathcal{A}$ is the basic algebra of $R^{2\Lambda_0+2\Lambda_s}(\alpha_0+\alpha_1)$ since they share the same number of simple modules by the categorification theorem. Therefore, $R^{\Lambda}(\beta)$ is tame.
\end{proof}

\section{Proof of Step 3}
We give a proof of Step 3 in this section.
Recall that $T(\Lambda)=\{\Lambda\}\cup\bigcup_{0\leq s\leq 5}T(\Lambda)_s$ in Definition \ref{def::T_Lambda} and $NW(\Lambda)$ is the set consisting of $\Lambda'\in \pcl (\Lambda)$ such that $R^\Lambda(\beta_{\Lambda'})$ is not of wild representation type (see Definition \ref{def::non-wild-component}).
It is worth mentioning that $\mathcal {NW}(\Lambda) \subset \mathcal T(\Lambda)$ if and only if $NW(\Lambda)\subset T(\Lambda)$.
We shall show the latter in this section.

\subsection{Strategy for proving Step 3}
When we apply Corollary \ref{cor::cover-non-wild} with $S=T(\Lambda)$, we obtain $NW(\Lambda)\subset T(\Lambda)$ once we show that $R^{\Lambda}(\beta_{\Lambda''})$ is wild for any $\Lambda''\in S(\Lambda')_\rightarrow\setminus T(\Lambda)$ and $\Lambda'\in T(\Lambda)$. 
If $\Lambda'=\Lambda$, then $S(\Lambda')_\rightarrow\setminus T(\Lambda)=\emptyset$ 
since $S(\Lambda')_\rightarrow=T(\Lambda)_0\cup T(\Lambda)_1$ as was mentioned in Remark \ref{rem::-P-lambda}.
Therefore, it suffices to consider $\Lambda'\in \bigcup_{0\leq s\leq 5}T(\Lambda)_s$.
In fact, we will show the following stronger claim.

\begin{claim}\label{claim::step-3} 
$R^{\Lambda}(\beta_{\Lambda''})$ is wild for any $\Lambda''\in S(\Lambda')_\rightarrow\setminus T(\Lambda)$ if $s=1$,
and for any $\Lambda''\in S(\Lambda')_\rightarrow$ if $s\neq 1$.
\end{claim}

We will use the following notation in the proof below.
$$
S(T(\Lambda)_s)_\rightarrow=\{\Lambda''\in S(\Lambda')_\rightarrow \mid \Lambda'\in T(\Lambda)_s\}.
$$
In the subsections below, we distinguish the following cases,
$$
(\star): \left\{\begin{aligned}
&T(\Lambda)_0\cup T(\Lambda)_1 \cup T(\Lambda)_2 &\text{for }  k=3,4,\\
&T(\Lambda)_3                                    &\text{for } k=3,4,5,\\
&T(\Lambda)_4\cup T(\Lambda)_5                   &\text{for } k=4,5,6.
\end{aligned}\right.
$$
We remind the reader that $T(\Lambda)_4=T(\Lambda)_5=\emptyset$ if $k=3$.
Proposition \ref{prop::iso-sigma} is also frequently used below without further notice.
The next proposition is the reason why we distinguish the cases in $(\star)$.

\begin{Prop}
If Claim \ref{claim::step-3} holds for all $\Lambda'$ in $(\star)$, then it holds for arbitrary $k\geq 3$ and $0\leq s\leq 5$.
\end{Prop}
\begin{proof}
Suppose $\Lambda'\in T(\Lambda)_s$ and $\Lambda''\in S(\Lambda')_{\rightarrow}$ such that $\Lambda'\notin (\star)$.
By applying Lemma \ref{lem::reduction-step-3}  repeatedly as many times as possible for reducing the level $k$ to $3\le k'<k$, we may reach 
$k'=4$ if $s=0,1,2$, $k'=5$ if $s=3$ and $k'=6$ if $s=4,5$. 
Then, we write $\Lambda'=\bar \Lambda' +\tilde \Lambda$, $\Lambda''=\bar \Lambda'' +\tilde \Lambda$ such that  $\bar\Lambda'\in (\star)$, $\bar \Lambda''\in S(\bar \Lambda')_{\rightarrow}$ and $\bar\Lambda, \bar \Lambda', \bar\Lambda''\in P^+_{cl,k'}$. 

According to our assumption, $R^{\bar \Lambda}(\beta_{\bar \Lambda,\bar \Lambda''})$ is wild if $s\neq 1$, and so is $R^\Lambda(\beta_{\Lambda,\Lambda''})$ by $\beta_{\Lambda,\Lambda''}=\beta_{\bar\Lambda,\bar\Lambda''}$ and Lemma \ref{lem::reduction-level}.
If $s=1$ and $\Lambda''\notin T(\Lambda)$, then $\bar\Lambda''\notin T(\bar\Lambda)$ by Lemma \ref{lem::embdding-subgraph}, and $R^\Lambda(\beta_{\Lambda,\Lambda''})$ is also wild.
\end{proof}

\subsection{Level 3}
We prove Claim \ref{claim::step-3} for $k=3$ in this subsection.
We divide the proof into three cases: $\Lambda=3\Lambda_i$, $2\Lambda_i+\Lambda_j$ for $i\neq j$, and $\Lambda_i+\Lambda_j+\Lambda_h$ for $i<j<h$.
\begin{Lemma}\label{lem::step-3-level3-3i}
Claim \ref{claim::step-3} holds for  all $\Lambda'\in \cup_{0\leq s\leq 3}T(\Lambda)_s$ if $\Lambda=3\Lambda_i$ for some $i\in I$.
\end{Lemma}
\begin{proof}
It suffices to assume $\Lambda=3\Lambda_0$ by \eqref{equ::iso-sigma-graph}.
Then, $T(\Lambda)_0=T(\Lambda)_4=T(\Lambda)_5=\emptyset$, $T(\Lambda)_1=\{\Lambda_0+\Lambda_1+\Lambda_\ell\}$, $T(\Lambda)_2=\{\Lambda_0+\Lambda_2+\Lambda_{\ell-1}\}$ if $\ell\ge 3$,   $T(\Lambda)_3=\{2\Lambda_1+\Lambda_{\ell-1}, \Lambda_2+2\Lambda_\ell\}$ if $\ell\geq 2$. We find all arrows with source $\Lambda_0+\Lambda_1+\Lambda_\ell$.

Since $S(\Lambda_0+\Lambda_1+\Lambda_\ell)_\rightarrow= T(\Lambda)_2\cup T(\Lambda)_3$, we have $S(T(\Lambda)_1)_\rightarrow\setminus T(\Lambda)=\emptyset$ and there is nothing to prove.
If $\Lambda'=\Lambda_0+\Lambda_2+\Lambda_{\ell-1}$, then $R^\Lambda(\beta_{\Lambda'})$ is wild  by Lemma \ref{lem::step-2-T_2-3i}, and so is $R^\Lambda(\beta_{\Lambda''})$ for any $\Lambda''\in S(\Lambda')_\rightarrow$ by Lemma \ref{lem::reduction-arrow}.
Thus, it remains to prove the claim when $\Lambda'\in  T(\Lambda)_3$.

We only show the case $\Lambda'=\Lambda_2+2\Lambda_\ell$ and the other case can be checked similarly.
By noting that $X_{\Lambda'}=(2,1,0,\ldots,0)$ and $\min(X_{\Lambda'}+\Delta_{2,\ell})\neq 0$ and that $\min(X_{\Lambda'}+\Delta_{\ell,2})\neq 0$ if $\ell=3$, we find $S(\Lambda')_\rightarrow=\{\Lambda_0+\Lambda_2+\Lambda_{\ell-1}, \Lambda_3+\Lambda_{\ell-1}+\Lambda_\ell \}$,
where the former is in $T(\Lambda)_2$ and the latter appears only if $\ell\ge 4$, and we have an arrow 
$$
\xymatrix@C=1cm{
\boxed{\Lambda_0+\Lambda_2+\Lambda_{\ell-1}}_2 \ar[r]^-{(0,2)}&
\Lambda_3+\Lambda_{\ell-1}+\Lambda_\ell}
$$
in the latter case, where the box with subscript 2 means that $\Lambda_0+\Lambda_2+\Lambda_{\ell-1}$ belongs to $T(\Lambda)_2$.
Then, $R^\Lambda(\beta_{\Lambda_0+\Lambda_2+\Lambda_{\ell-1}})$ is wild by Lemma \ref{lem::step-2-T_2-3i} and so is $R^\Lambda(\beta_{\Lambda_3+\Lambda_{\ell-1}+\Lambda_\ell})$ by Lemma \ref{lem::reduction-arrow}.
\end{proof}

\begin{Lemma}\label{lem::step-3-level3-2ij}
Claim \ref{claim::step-3} holds for any $\Lambda'\in \cup_{0\leq s\leq 3}T(\Lambda)_s$ if $\Lambda=2\Lambda_i+\Lambda_j$ with $i\neq j\in I$.
\end{Lemma}
\begin{proof}
We assume $\Lambda=2\Lambda_0+\Lambda_h$ with $h \equiv_e j-i$ by \eqref{equ::iso-sigma-graph}.
In this case, $T(\Lambda)_3=T(\Lambda)_4=T(\Lambda)_5=\emptyset$ and 
we have
$$
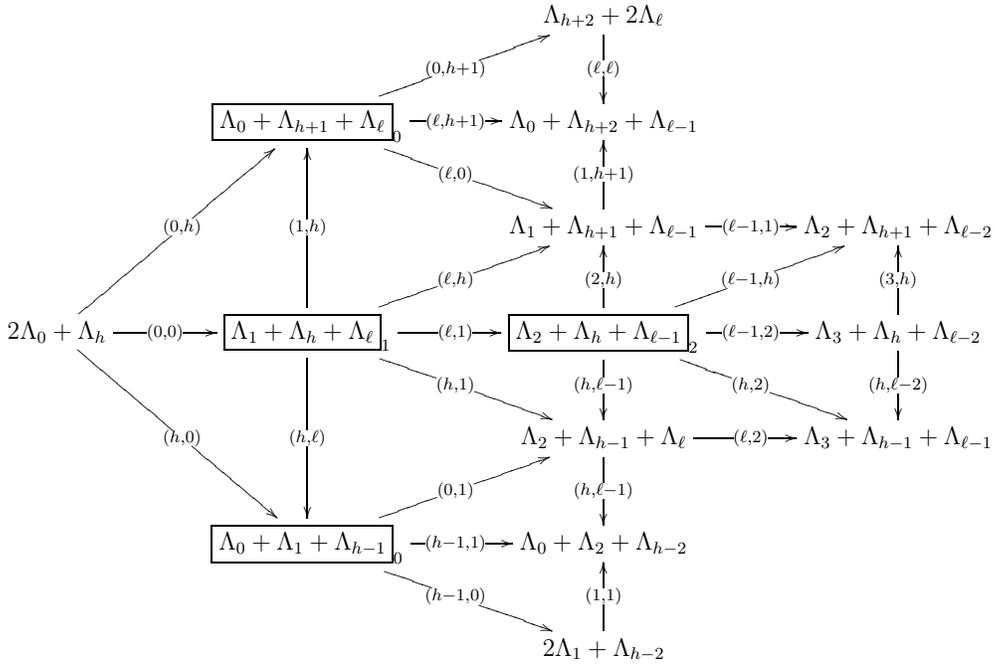
\scalebox{0.8}{
\xymatrix@C=1.5cm@R=1cm{
&&\Lambda_{h+2}+2\Lambda_\ell
\ar[d]|{(\ell,\ell)}
&\\
&\boxed{\Lambda_0+\Lambda_{h+1}+\Lambda_\ell}_0
\ar[rd]|{(\ell,0)}\ar[r]|{(\ell,h+1)}\ar[ru]|{(0,h+1)}
&\Lambda_0+\Lambda_{h+2}+\Lambda_{\ell-1}&\\
&&\Lambda_1+\Lambda_{h+1}+\Lambda_{\ell-1}\ar[r]|{(\ell-1, 1)}\ar[u]|{(1, h+1)}
&\Lambda_2+\Lambda_{h+1}+\Lambda_{\ell-2}\\
2\Lambda_0+\Lambda_h
\ar[r]|-{(0,0)}\ar[uur]|{(0,h)}\ar[ddr]|{(h,0)}
&\boxed{\Lambda_1+\Lambda_h+\Lambda_\ell}_1
\ar[uu]|{(1,h)}\ar[dd]|{(h,\ell)}\ar[r]|{(\ell, 1)}\ar[ru]|{(\ell,h)}\ar[rd]|{(h,1)}
&\boxed{\Lambda_2+\Lambda_h+\Lambda_{\ell-1}}_2
\ar[r]|{(\ell-1, 2)}\ar[ru]|{(\ell-1, h)}\ar[rd]|{(h,2)}\ar[u]|{(2,h)}\ar[d]|{(h,\ell-1)}
&\Lambda_3+\Lambda_h+\Lambda_{\ell-2}
\ar[u]|{(3,h)}\ar[d]|{(h,\ell-2)}
\\
&&\Lambda_2+\Lambda_{h-1}+\Lambda_\ell \ar[r]|{(\ell,2)}\ar[d]|{(h,\ell-1)}
&\Lambda_3+\Lambda_{h-1}+\Lambda_{\ell-1}\\
&\boxed{\Lambda_0+\Lambda_1+\Lambda_{h-1}}_0
\ar[r]|{(h-1,1)}\ar[ru]|{(0,1)}\ar[rd]|{(h-1,0)}
&\Lambda_0+\Lambda_2+\Lambda_{h-2}&\\
&&2\Lambda_1+\Lambda_{h-2}\ar[u]|{(1,1)}
}}
$$
where the subscripts $0,1,2$ indicate that the element in the box belongs to   $T(\Lambda)_0$, $T(\Lambda)_1$, $T(\Lambda)_2$, respectively. 
We omit the conditions on $h$ (see Corollary \ref{cor::find-arrow}) that guarantee the existence of some arrows (and hence vertices) in the above quiver. For example, the arrow 
$\xymatrix@C=1.2cm{
\Lambda_1+\Lambda_h+\Lambda_\ell\ar[r]|-{(h,1)}&
\Lambda_2+\Lambda_{h-1}+\Lambda_\ell}$
exists only if $h\geq 3$.

It is easy to find $S(T(\Lambda)_s)_\rightarrow$ for $s=0,1,2$ from the above quiver. In particular,
$$
S(T(\Lambda)_1)_\rightarrow\setminus T(\Lambda)=\{\Lambda_2+\Lambda_{h-1}+\Lambda_\ell, \Lambda_1+\Lambda_{h+1}+\Lambda_{\ell-1}\}.
$$
We omit the lists of elements for $S(T(\Lambda)_0)_\rightarrow$ and $S(T(\Lambda)_2)_\rightarrow$.
To prove the claim, it suffices to show the wildness of $R^\Lambda(\beta_{\Lambda''})$ by Lemma \ref{lem::reduction-arrow}
for 
$$
\Lambda''\in\{\Lambda_{h+2}+2\Lambda_\ell, 
\Lambda_1+\Lambda_{h+1}+\Lambda_{\ell-1}, 
\Lambda_2+\Lambda_{h-1}+\Lambda_\ell, 
2\Lambda_1+\Lambda_{h-2}, 
\Lambda_3+\Lambda_h+\Lambda_{\ell-2}\}.
$$ 
\begin{itemize}
    \item Set $R=R^\Lambda(\beta_{\Lambda_{h+2}+2\Lambda_\ell})$. By Lemma \ref{lem::recurrence}, $\beta_{\Lambda_{h+2}+2\Lambda_\ell}=2(\alpha_0+\alpha_1+\ldots+\alpha_h)+\alpha_{h+1}$. We choose $e_0=e(h~0~1~2~\ldots~h~h+1~h-1~\ldots~1~0)$. Then, $\dim_q e_0Re_0=1+3q^2+3q^4+q^6$ so that $R$ is wild. Similarly, $R^\Lambda(\beta_{2\Lambda_1+\Lambda_{h-2}})$ is wild by $\beta_{2\Lambda_1+\Lambda_{h-2}}=\alpha_{h-1}+2(\alpha_h+\alpha_{h+1}+\ldots+\alpha_\ell+\alpha_0)$ and the same reasoning. 

    \item Set $R=R^{\Lambda}(\beta_{\Lambda_1+\Lambda_{h+1}+\Lambda_{\ell-1}})$. By Lemma \ref{lem::recurrence}, $\beta_{\Lambda_1+\Lambda_{h+1}+\Lambda_{\ell-1}}=2\alpha_0+\alpha_1+\ldots+\alpha_h+\alpha_\ell$. We choose $e_0=e(0~1~2~\ldots~h-1~0~h~ \ell)$, then $\dim_q e_0Re_0=1+3q^2+3q^4+q^6$, so that $R$ is wild. Similarly, $R^{\Lambda}(\beta_{\Lambda_2+\Lambda_{h-1}+\Lambda_\ell})$ has a wild idempotent truncation. 

    \item We find $\beta_{\Lambda,\Lambda_3+\Lambda_h+\Lambda_{\ell-2}}=\beta_{2\Lambda_0,\Lambda_3+\Lambda_{\ell-2}}=3\alpha_0+2\alpha_1+\alpha_2+\alpha_{\ell-1}+2\alpha_\ell$ by Lemma \ref{lem::embedding-Lambda} and Lemma \ref{lem::recurrence}. Then, one can apply \cite[Theorem B]{Ar-rep-type} to find that $R^{2\Lambda_0}(\beta_{\Lambda_3+\Lambda_{\ell-2}})$ is wild, and therefore, $R^\Lambda(\beta_{\Lambda_3+\Lambda_h+\Lambda_{\ell-2}})$ is wild by Lemma \ref{lem::reduction-level}.
\end{itemize}
We have proved the claim.
\end{proof}

\begin{Lemma}\label{lem::step-3-level3-ijp}
Let $\Lambda=\Lambda_i+\Lambda_j+\Lambda_h$ with $h<i<j\in I$. Then, Claim \ref{claim::step-3} is true for all $\Lambda'\in \cup_{0\leq s\leq 3}T(\Lambda)_s$.
\end{Lemma}
\begin{proof}
By \eqref{equ::iso-sigma-graph}, we assume $\Lambda=\Lambda_0+\Lambda_i+\Lambda_j$ with $0<i<j\in I$.
Then, $T(\Lambda)_s=\emptyset$ for $1\leq s\leq 5$ and $T(\Lambda)_0=\{\Lambda_{0,i}, \Lambda_{0,j}, \Lambda_{i,0}, \Lambda_{i,j}, \Lambda_{j,0}, \Lambda_{j,i}\}$, where $\Lambda_{0,j}$ (resp. $\Lambda_{j,i}$, $\Lambda_{i,0}$) appears only if $j<\ell$ (resp. $i+1<j$, $i>1$).
Recall $T'(\Lambda)_0=\{\Lambda_{0,i}, \Lambda_{i,j}, \Lambda_{j,0}\}$.
It is proved in Lemma \ref{lem::step-2-T_0-exception} that $R^\Lambda(\beta_{\Lambda'})$ is wild for $\Lambda'\in T(\Lambda)_0\setminus T'(\Lambda)_0$, and so is $R^\Lambda(\beta_{\Lambda''})$ for $\Lambda''\in S(\Lambda')_\rightarrow$ by Lemma \ref{lem::reduction-arrow}.
Hence, it remains to prove the claim when $\Lambda'\in T'(\Lambda)_0$.
In the following, we only show the case $\Lambda'=\Lambda_{i,j}$ and leave the other two cases to the reader.

All elements of $S(\Lambda')_\rightarrow$ are displayed as
$$
\scalebox{0.8}{\xymatrix@C=1.7cm@R=1.2cm{
\boxed{\Lambda_1+\Lambda_{i-1}+\Lambda_j}_0
\ar[d]|{(i-1, j)}
&\boxed{\Lambda_0+\Lambda_{i-1}+\Lambda_{j+1}}_0
\ar[l]|-{(j+1,0)}
\ar[ld]|{(i-1,0)}
\ar[d]|{(i-1,j+1)}
\ar[rd]|{(0,j+1)}
\ar[r]|-{(0,i-1)}
&
\boxed{\Lambda_{i}+\Lambda_{j+1}+\Lambda_\ell}_0
\ar[d]|{(i, j+1)}
\\
\Lambda_1+\Lambda_{i-2}+\Lambda_{j+1}
&\Lambda_0+\Lambda_{i-2}+\Lambda_{j+2}
&\Lambda_{i-1}+\Lambda_{j+2}+\Lambda_\ell
}}
$$
where $\Lambda_0+\Lambda_{i-2}+\Lambda_{j+2}$ (resp. $\Lambda_{i-1}+\Lambda_{j+2}+\Lambda_\ell$, resp. $\Lambda_1+\Lambda_{i-2}+\Lambda_{j+1}$) appears only if $j-i+2\leq \ell$ (resp. $j\leq \ell-2$, resp. $i\geq 3$). 
We observe that $\Lambda_{i,0}=\Lambda_1+\Lambda_{i-1}+\Lambda_j$ and $\Lambda_{0,j}=\Lambda_{i}+\Lambda_{j+1}+\Lambda_\ell$, namely, they belong to $T(\Lambda)_0\setminus T'(\Lambda)_0$.
Thus, $R^\Lambda(\beta_{\Lambda''})$ is wild for $\Lambda_0+\Lambda_{i-2}+\Lambda_{j+2}\neq \Lambda''\in S(\Lambda')_\rightarrow$ by Lemma \ref{lem::step-2-T_0-exception} and Lemma \ref{lem::reduction-arrow}.
Besides,
$$
\beta_{\Lambda,\Lambda_0+\Lambda_{i-2}+\Lambda_{j+2}}=\beta_{\Lambda_i+\Lambda_j,\Lambda_{i-2}+\Lambda_{j+2}}= \alpha_{i-1}+2(\alpha_i+\alpha_{i+1}+\ldots+\alpha_{j})+\alpha_{j+1}
$$
by Lemma \ref{lem::embedding-Lambda}. Then, $R^\Lambda(\beta_{\Lambda_0+\Lambda_{i-2}+\Lambda_{j+2}})$ is wild  because $R^{\Lambda_i+\Lambda_j}(\beta_{\Lambda_{i-2}+\Lambda_{j+2}})$ is wild \cite[Theorem A]{Ar-rep-type}.
\end{proof}

\subsection{Level 4}
In this subsection, we prove Claim \ref{claim::step-3} for $k=4$.
We divide the proof into five cases.
\begin{Lemma}\label{lem::step-3-level4-4i}
Claim \ref{claim::step-3} holds for all $\Lambda'\in \cup_{0\leq s\leq 5}T(\Lambda)_s$ if $\Lambda=4\Lambda_i$ for some $i\in I$.
\end{Lemma}
\begin{proof}
It suffices to consider $i=0$.
Then, $T(\Lambda)_0=T(\Lambda)_5=\emptyset$, $T(\Lambda)_1=\{2\Lambda_0+\Lambda_1+\Lambda_{\ell}\}$, $T(\Lambda)_2=\{2\Lambda_0+\Lambda_2+\Lambda_{\ell-1}\}$ if $\ell\geq 3$, $T(\Lambda)_3=\{ \Lambda_0+\Lambda_2+2\Lambda_{\ell}, \Lambda_0+2\Lambda_1+\Lambda_{\ell-1}\}$ if $\ell\geq 2$ and $T(\Lambda)_4=\{ 2\Lambda_1+2\Lambda_{\ell}\}$.
All elements of $S(2\Lambda_0+\Lambda_1+\Lambda_{\ell})_\rightarrow$ are displayed as
$$
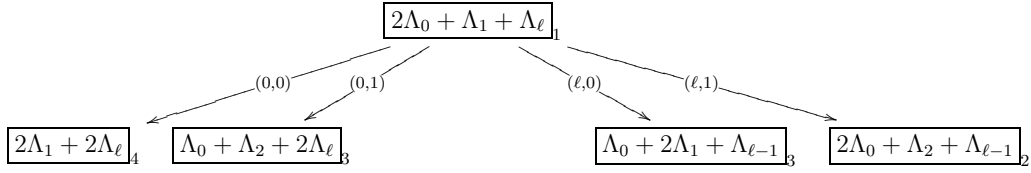
\scalebox{0.8}{\xymatrix@C=0.3cm@R=1.2cm{
&&\boxed{2\Lambda_0+\Lambda_1+\Lambda_{\ell}}_1
\ar[lld]|{(0,0)}
\ar[ld]|{(0,1)}
\ar[rd]|{(\ell, 0)}
\ar[rrd]|{(\ell, 1)}
&&\\
\boxed{2\Lambda_1+2\Lambda_{\ell}}_4
&\boxed{\Lambda_0+\Lambda_2+2\Lambda_{\ell}}_3
&
&\boxed{\Lambda_0+2\Lambda_1+\Lambda_{\ell-1}}_3
&\boxed{2\Lambda_0+\Lambda_2+\Lambda_{\ell-1}}_2
}}
$$
where one finds $S(T(\Lambda)_1)_\rightarrow=\cup_{2\leq j\leq 4} T(\Lambda)_j$, i.e., $S(T(\Lambda)_1)_\rightarrow\setminus T(\Lambda)=\emptyset$. Hence there is nothing to prove for $s=1$.
We also observe that $R^\Lambda(\beta_{\Lambda'})$ is wild for any $\Lambda'\in T(\Lambda)_2\cup T(\Lambda)_3$ by Proposition \ref{prop::step-2-result-T_2} and Lemma \ref{lem::step-2-T_3-4i}, which implies  the claim  for $s=2,3$ by Lemma \ref{lem::reduction-arrow}.
It then remains to consider the wildness of $R^\Lambda (\beta_{\Lambda''})$ for $\Lambda''\in S(2\Lambda_1+2\Lambda_\ell)_\rightarrow$.

The following computation of arrows gives us all elements of $S(T(\Lambda)_4)_\rightarrow$.
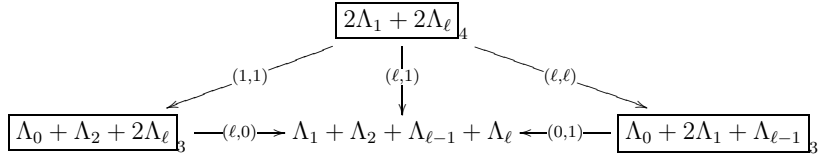
$$
\scalebox{0.8}{
\xymatrix@C=1.5cm@R=1cm{
&\boxed{2\Lambda_1+2\Lambda_{\ell}}_4
\ar[dl]|{(1,1)} \ar[dr]|{(\ell,\ell)}\ar[d]|{(\ell,1)}&\\
\boxed{\Lambda_0+\Lambda_2+2\Lambda_\ell}_3 \ar[r]|-{(\ell,0)}
&\Lambda_1+\Lambda_2+\Lambda_{\ell-1}+\Lambda_\ell
&\boxed{\Lambda_0+2\Lambda_1+\Lambda_{\ell-1}}_3 \ar[l]|-{(0,1)}
}}
$$
Here, $\Lambda_1+\Lambda_2+\Lambda_{\ell-1}+\Lambda_\ell$ appears only if $\ell\geq 3$. Recall that we already know that $R^\Lambda(\beta_{\Lambda'})$ is wild if $\Lambda'\in T(\Lambda)_3$. Hence, $R^\Lambda (\beta_{\Lambda_1+\Lambda_2+\Lambda_{\ell-1}+\Lambda_\ell})$ is  wild by  Lemma \ref{lem::reduction-arrow} again.
\end{proof}

\begin{Lemma}\label{lem::step-3-level4-3ij}
Claim \ref{claim::step-3} holds for all $\Lambda'\in \cup_{0\leq s\leq 5}T(\Lambda)_s$ if
$\Lambda=3\Lambda_i+\Lambda_j$ with $i\neq j\in I$.
\end{Lemma}
\begin{proof}
We assume $\Lambda=3\Lambda_0+\Lambda_h$ with $h\neq 0$ as usual.
We have $T(\Lambda)_4=T(\Lambda)_5=\emptyset$.
Since $R^\Lambda(\beta_{\Lambda'})$ is wild for any $\Lambda'\in T(\Lambda)_2$ by Proposition \ref{prop::step-2-result-T_2}, it remains to consider $\Lambda'\in T(\Lambda)_s$ for $s=0,1,3$.

Following Corollary \ref{cor::embedding-path}, we embed the quiver $\vec C(2\Lambda_0+\Lambda_h)$ into $\vec C(3\Lambda_0+\Lambda_h)$ by adding $\Lambda_0$ to each vertex.  
The wildness is preserved under the embedding by Lemma \ref{lem::reduction-level}. 
Observe that $T(\Lambda)_0=\{2\Lambda_0+\Lambda_{h+1}+\Lambda_\ell, 2\Lambda_0+\Lambda_1+\Lambda_{h-1}\}$ and $T(\Lambda)_1=\{\Lambda_0+\Lambda_1+\Lambda_h+\Lambda_\ell\}$ belong to this embedded $\vec C(2\Lambda_0+\Lambda_h)$. 
Then, utilizing the drawing of the quiver $\vec C(2\Lambda_0+\Lambda_h)$ in Lemma \ref{lem::step-3-level3-2ij}, we can show that all elements of $S(T(\Lambda)_0)_\rightarrow\cup S(T(\Lambda)_1)_\rightarrow$ appear in the following quiver.
Indeed, Corollary \ref{cor::embedding-path} asserts that no extra arrows appear among the vertices of the embedded quiver, which are in dotted boxes, and $X_{\Lambda'}=\Delta_{0,0}$ for $\Lambda'=\Lambda_0+\Lambda_1+\Lambda_h+\Lambda_\ell$ implies that we have at most 4 extra vertices in $S(\Lambda')_\rightarrow$, $X_{\Lambda'}=\Delta_{0,0}+\Delta_{1,h}$ or $\Delta_{0,0}+\Delta_{h,\ell}$ for $\Lambda'=2\Lambda_0+\Lambda_{h+1}+\Lambda_\ell$ or $2\Lambda_0+\Lambda_1+\Lambda_{h-1}$ implies that an extra vertex exists in $S(\Lambda')_\rightarrow$ only when it is reached by replacing $2\Lambda_0$ with $ \Lambda_1+\Lambda_\ell$.
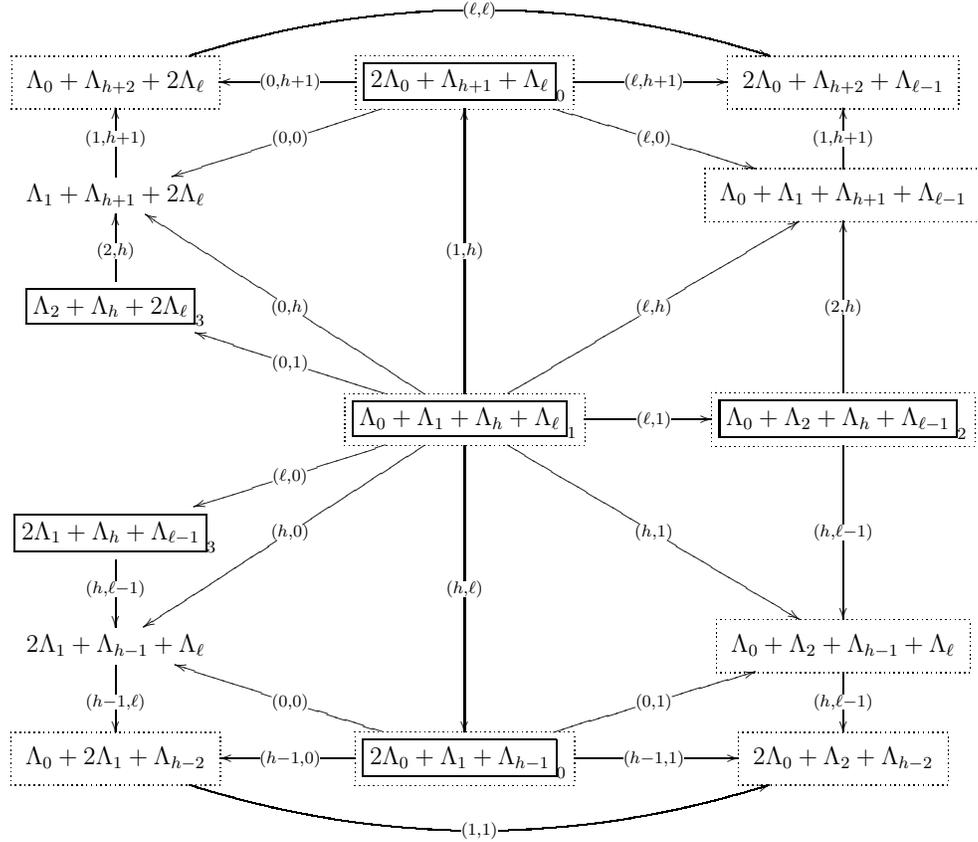
$$
\scalebox{0.8}{
\xymatrix@C=2cm@R=1cm{
*++[F.]{\Lambda_0+\Lambda_{h+2}+2\Lambda_\ell}
\ar@/^1.2cm/[rr]|{(\ell,\ell)}
&*+[F.]{\boxed{2\Lambda_0+\Lambda_{h+1}+\Lambda_\ell}_0}
\ar[rd]|{(\ell,0)}\ar[r]|{(\ell,h+1)}\ar[l]|{(0,h+1)}\ar[dl]|{(0,0)}
&*++[F.]{2\Lambda_0+\Lambda_{h+2}+\Lambda_{\ell-1}}
\\
\Lambda_1+\Lambda_{h+1}+2\Lambda_\ell
\ar[u]|{(1,h+1)}
&&*++[F.]{\Lambda_0+\Lambda_1+\Lambda_{h+1}+\Lambda_{\ell-1}}
\ar[u]|{(1, h+1)}
\\
\boxed{\Lambda_2+\Lambda_h+2\Lambda_\ell}_3
\ar[u]|{(2,h)}
&&\\
&*+[F.]{\boxed{\Lambda_0+\Lambda_1+\Lambda_h+\Lambda_\ell}_1}
\ar[uuu]|{(1,h)}\ar[ddd]|{(h,\ell)}\ar[r]|{(\ell, 1)}
\ar[ruu]|{(\ell,h)}\ar[rdd]|{(h,1)}
\ar[ul]|{(0,1)}\ar[dl]|{(\ell,0)}\ar[uul]|{(0,h)}\ar[ddl]|{(h,0)}
&*+[F.]{\boxed{\Lambda_0+\Lambda_2+\Lambda_h+\Lambda_{\ell-1}}_2}
\ar[uu]|{(2,h)}\ar[dd]|{(h,\ell-1)}
\\
\boxed{2\Lambda_1+\Lambda_h+\Lambda_{\ell-1}}_3
\ar[d]|{(h,\ell-1)}
&&\\
2\Lambda_1+\Lambda_{h-1}+\Lambda_{\ell}
\ar[d]|{(h-1, \ell)}
&&*++[F.]{\Lambda_0+\Lambda_2+\Lambda_{h-1}+\Lambda_\ell}
\ar[d]|{(h,\ell-1)}
\\
*++[F.]{\Lambda_0+2\Lambda_1+\Lambda_{h-2}} \ar@/_1.2cm/[rr]|{(1,1)}
&*+[F.]{\boxed{2\Lambda_0+\Lambda_1+\Lambda_{h-1}}_0}
\ar[r]|{(h-1,1)}\ar[ru]|{(0,1)}\ar[l]|{(h-1,0)}\ar[ul]|{(0,0)}
&*++[F.]{2\Lambda_0+\Lambda_2+\Lambda_{h-2}}
}}
$$
The vertex $\Lambda_1+\Lambda_{h+1}+2\Lambda_\ell$ (resp. $2\Lambda_1+\Lambda_{h-1}+\Lambda_\ell$) appears only if $h<\ell$ (resp. $h>1$).
Then, we only need to check the wildness of $R^{\Lambda}(\beta_{\Lambda_1+\Lambda_{h+1}+2\Lambda_\ell})$ and $R^{\Lambda}(\beta_{2\Lambda_1+\Lambda_{h-1}+\Lambda_\ell})$.
If $h=1$, then $\Lambda_1+\Lambda_2+2\Lambda_\ell\in T(\Lambda)_3$ and $R^{\Lambda}(\beta_{\Lambda_1+\Lambda_2+2\Lambda_\ell})=R^\Lambda(2\alpha_0+\alpha_1)$ is wild by Lemma \ref{lem::step-2-T_3-3ij}.
If $h>1$, then $\beta_{\Lambda_1+\Lambda_{h+1}+2\Lambda_\ell}=2\alpha_0+\alpha_1+\ldots +\alpha_h$ and
$$
\dim_q e_0 R^{\Lambda}(\beta_{\Lambda_1+\Lambda_{h+1}+2\Lambda_\ell})e_0=1 +3q^2+4q^4+3q^6+q^8
$$
for $e_0=e(0~1~2~\ldots~h-1~h~0)$ implies that $R^{\Lambda}(\beta_{\Lambda_1+\Lambda_{h+1}+2\Lambda_\ell})$ is wild.
Similarly, we can show that $R^{\Lambda}(\beta_{2\Lambda_1+\Lambda_{h-1}+\Lambda_\ell})$ is wild.

Now we turn to $S(T(\Lambda)_3)_\rightarrow$, where $T(\Lambda)_3=\{\Lambda_2+\Lambda_h+2\Lambda_\ell,2\Lambda_1+\Lambda_h+\Lambda_{\ell-1}\}$.
By direct calculation, all elements of $S(T(\Lambda)_3)_\rightarrow$ are given as follows.
$$
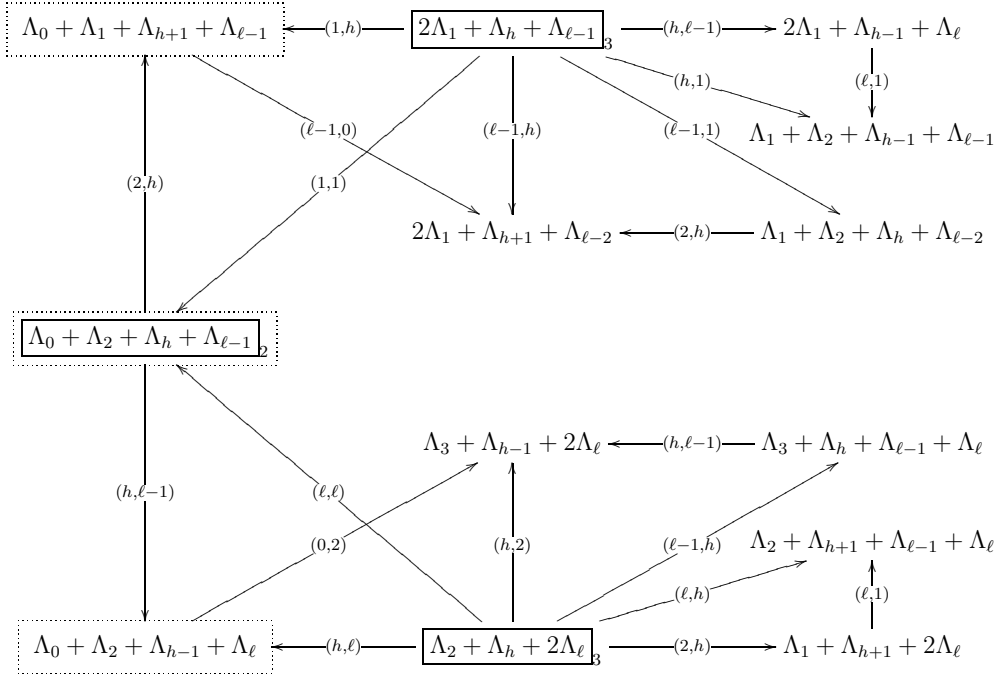
\scalebox{0.8}{
\xymatrix@C=2cm@R=1cm{
*++[F.]{\Lambda_0+\Lambda_1+\Lambda_{h+1}+\Lambda_{\ell-1}}
\ar[ddr]|{(\ell-1, 0)}
&\boxed{2\Lambda_1+\Lambda_h+\Lambda_{\ell-1}}_3
\ar[l]|-{(1,h)}\ar[dddl]|{(1,1)}\ar[dr]|{(h,1)}\ar[r]|{(h,\ell-1)}
\ar[ddr]|{(\ell-1,1)}\ar[dd]|{(\ell-1,h)}
&2\Lambda_1+\Lambda_{h-1}+\Lambda_\ell
\ar[d]|{(\ell,1)}
\\
&&\Lambda_1+\Lambda_2+\Lambda_{h-1}+\Lambda_{\ell-1}
\\
&2\Lambda_1+\Lambda_{h+1}+\Lambda_{\ell-2}
&\Lambda_1+\Lambda_2+\Lambda_h+\Lambda_{\ell-2}
\ar[l]|{(2,h)}
\\
*+[F.]{\boxed{\Lambda_0+\Lambda_2+\Lambda_h+\Lambda_{\ell-1}}_2}
\ar[uuu]|{(2,h)}\ar[ddd]|{(h,\ell-1)}&
&\\
&\Lambda_3+\Lambda_{h-1}+2\Lambda_\ell
&\Lambda_3+\Lambda_h+\Lambda_{\ell-1}+\Lambda_\ell
\ar[l]|{(h,\ell-1)}
\\
&&\Lambda_2+\Lambda_{h+1}+\Lambda_{\ell-1}+\Lambda_\ell
\\
*++[F.]{\Lambda_0+\Lambda_2+\Lambda_{h-1}+\Lambda_\ell}
\ar[uur]|{(0,2)}
&\boxed{\Lambda_2+\Lambda_h+2\Lambda_\ell}_3
\ar[l]|-{(h,\ell)}\ar[uu]|{(h,2)}\ar[luuu]|{(\ell,\ell)}
\ar[r]|{(2,h)}\ar[ur]|{(\ell,h)}\ar[uur]|{(\ell-1,h)}
&\Lambda_1+\Lambda_{h+1}+2\Lambda_\ell
\ar[u]|{(\ell,1)}
}}
$$
where the restrictions on $h$ that guarantee the existence of some vertices are omitted.
Note that we have proved the wildness of $R^\Lambda(\beta_{\Lambda'})$ for $\Lambda'=2\Lambda_1+\Lambda_{h-1}+\Lambda_\ell$ and $\Lambda_1+\Lambda_{h+1}+2\Lambda_\ell$ just above. 
If $\Lambda'=\Lambda_0+\Lambda_2+\Lambda_h+\Lambda_{\ell-1}$, then $\beta_{\Lambda'}=2\alpha_0+\alpha_1+\alpha_\ell\in \mathcal T(\Lambda)_2$ and $R^\Lambda(\beta_{\Lambda'})$ is wild by Proposition \ref{prop::step-2-result-T_2}.
Therefore, we only need to show that $R^\Lambda(\beta_{\Lambda_1+\Lambda_2+\Lambda_h+\Lambda_{\ell-2}})$ and $R^\Lambda(\beta_{\Lambda_3+\Lambda_h+\Lambda_{\ell-1}+\Lambda_\ell})$ are wild.
Since $\Lambda_3+\Lambda_{\ell-1}+\Lambda_\ell\in S(T(3\Lambda_0)_3)_\rightarrow$,
$R^{3\Lambda_0}(\beta_{\Lambda_3+\Lambda_{\ell-1}+\Lambda_{\ell}})$ is wild  by  Lemma \ref{lem::step-3-level3-3i}.  We conclude that $R^\Lambda(\beta_{\Lambda_3+\Lambda_h+\Lambda_{\ell-1}+\Lambda_\ell})$ is wild by Lemma \ref{lem::reduction-level} and Lemma \ref{lem::embedding-Lambda}.
Similarly, $R^\Lambda(\beta_{\Lambda_1+\Lambda_2+\Lambda_h+\Lambda_{\ell-2}})$ is wild.
\end{proof}

\begin{Lemma}\label{lem::step-3-level4-2i2j}
Claim \ref{claim::step-3} holds for all $\Lambda'\in \cup_{0\leq s\leq 5}T(\Lambda)_s$ if
$ \Lambda=2\Lambda_i+2\Lambda_j$ with $i\neq j\in I$.
\end{Lemma}
\begin{proof}
We may assume $\Lambda=2\Lambda_0+2\Lambda_h$ as usual.
In this case, $T(\Lambda)_s=\emptyset$ for $s=3,4$, and $R^\Lambda(\beta_{\Lambda'})$ is wild for any $\Lambda'\in T(\Lambda)_0$ by Lemma \ref{lem::step-2-T_0-2i2j}.
Hence, it suffices to consider $T(\Lambda)_s$ for $s=1,2,5$.

(1) $T(\Lambda)_1=\{\Lambda_1+2\Lambda_h+\Lambda_\ell,2\Lambda_0+\Lambda_{h-1}+\Lambda_{h+1}\}$. 
We only show the first case since the second case is reduced to the first case by  Proposition \ref{prop::iso-sigma}.
We embed $\vec C(2\Lambda_0+\Lambda_h)$ into $\vec C(2\Lambda_0+2\Lambda_h)$ by adding $\Lambda_h$ to each vertex, in which the wildness is preserved by Lemma \ref{lem::reduction-level}. 
Then we have the unique extra arrow
$$
\xymatrix@C=1.2cm@R=1cm{
\boxed{\Lambda_1+2\Lambda_h+\Lambda_\ell}_1
\ar[r]^-{(h,h)}
&\boxed{\Lambda_1+\Lambda_{h-1}+\Lambda_{h+1}+\Lambda_\ell}_5
},
$$
if $\ell\geq2$, but the right hand side belongs to $T(\Lambda)$.
Hence, no extra vertex appears in $S(T(\Lambda)_1)_\rightarrow\setminus T(\Lambda)$, and there is nothing more to prove.

(2) $T(\Lambda)_2=\{\Lambda_2+2\Lambda_h+\Lambda_{\ell-1}, 2\Lambda_0+\Lambda_{h-2}+\Lambda_{h+2} \}$ if $\ell\geq 3$. 
We show that $R^\Lambda(\beta_{\Lambda''})$ for $\Lambda''\in S(\Lambda_2+2\Lambda_h+\Lambda_{\ell-1})_\rightarrow$ is wild as follows.
The proof for $\Lambda''\in S(2\Lambda_0+\Lambda_{h-2}+\Lambda_{h+2})$ is similar.
We embed $\vec C(2\Lambda_0+\Lambda_h)$ to $\vec C(2\Lambda_0+2\Lambda_h)$ as in (1). 
Then, we only need to consider the unique extra arrow
$$
\xymatrix@C=1.2cm@R=1cm{
\boxed{\Lambda_2+2\Lambda_h+\Lambda_{\ell-1}}_2
\ar[r]^-{(h,h)}
&\Lambda_2+\Lambda_{h-1}+\Lambda_{h+1}+\Lambda_{\ell-1}
}.
$$
Let $\Lambda''=\Lambda_2+\Lambda_{h-1}+\Lambda_{h+1}+\Lambda_{\ell-1}$.
Then $\beta:=\beta_{\Lambda''}=2\alpha_0+\alpha_1+\alpha_h+\alpha_\ell$.
If $3\le h \le \ell-2$, then $\dim_q e_0R^\Lambda(\beta)e_0=1+3q^2+3q^4+q^6$ for $e_0=e(01\ell0h) $, which implies that $R^{\Lambda}(\beta)$ is wild. 
Suppose $h=1,2,\ell-1$ or $\ell$. 
In each case, there is an arrow given below. 
$$
\begin{aligned}
&\xymatrix@C=1.2cm@R=1cm{
\boxed{\Lambda_{0,h}=\Lambda_0+\Lambda_1+\Lambda_2+\Lambda_\ell}_0
\ar[r]^-{(\ell,1)}
&\Lambda''=\Lambda_0+2\Lambda_2+\Lambda_{\ell-1}
} \quad \text{if } h=1,
\\
&\xymatrix@C=1.2cm@R=1cm{
\boxed{\Lambda_{0,h}=\Lambda_0+\Lambda_2+\Lambda_3+\Lambda_\ell}_0
\ar[r]^-{(\ell,0)}
&\Lambda''=\Lambda_1+\Lambda_2+\Lambda_3+\Lambda_{\ell-1}
} \quad \text{if } h=2,
\\
&\xymatrix@C=1.2cm@R=1cm{
\boxed{\Lambda_{h,0}=\Lambda_0+\Lambda_1+\Lambda_{\ell-2}+\Lambda_{\ell-1}}_0
\ar[r]^-{(0,1)}
&\Lambda''=\Lambda_2+\Lambda_{\ell-2}+\Lambda_{\ell-1}+\Lambda_\ell
} \quad \text{if } h=\ell-1,\\
&\xymatrix@C=1.2cm@R=1cm{
\boxed{\Lambda_{h,0}=\Lambda_0+\Lambda_1+\Lambda_{\ell-1}+\Lambda_\ell}_0
\ar[r]^-{(\ell,1)}
&\Lambda''=\Lambda_0+\Lambda_2+2\Lambda_{\ell-1}
} \quad \text{if } h=\ell.
\end{aligned}
$$
Since the left hand side is wild by Lemma \ref{lem::step-2-T_0-2i2j}, $R^\Lambda (\beta)$ for $h=1,2,\ell-1$ and $\ell$ is wild by Lemma \ref{lem::reduction-arrow}.

(3) $T(\Lambda)_5=\{\Lambda_1+\Lambda_{h-1}+\Lambda_{h+1}+\Lambda_\ell\}$ if $\ell\geq 2$.
If $h=1$ or $\ell$, then $\Lambda'=\Lambda_1+\Lambda_{h-1}+\Lambda_{h+1}+\Lambda_\ell\in T(\Lambda)_0$ and $R^{\Lambda}(\beta_{\Lambda'})$ is wild by Lemma \ref{lem::step-2-T_0-2i2j}.
Suppose $h\neq 1,\ell$.  
We compute the vertices of $S(\Lambda')_\rightarrow$ as follows.
$$
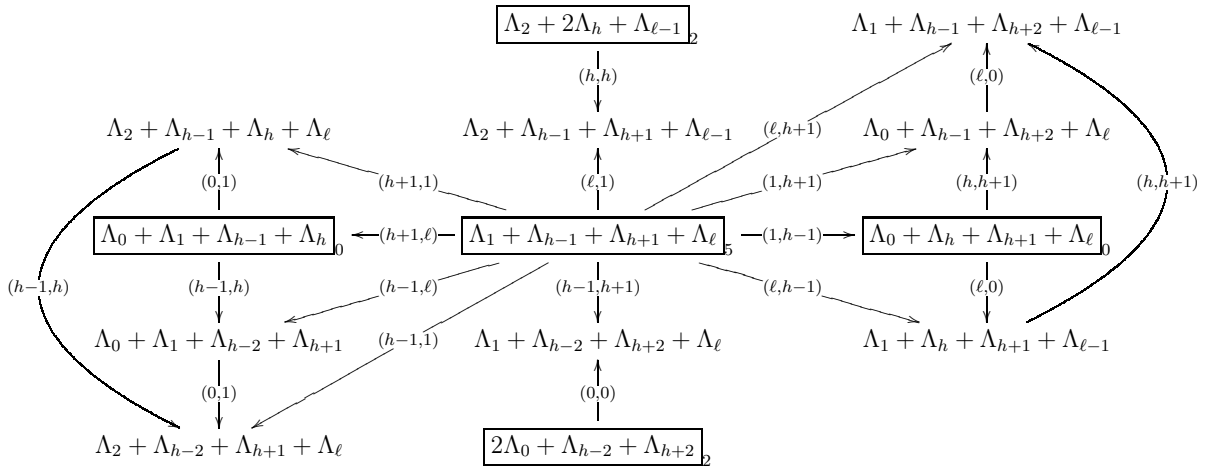
\scalebox{0.8}{
\xymatrix@C=1.7cm@R=1cm{
&\boxed{\Lambda_2+2\Lambda_h+\Lambda_{\ell-1}}_2
\ar[d]|{(h,h)}
&\Lambda_1+\Lambda_{h-1}+\Lambda_{h+2}+\Lambda_{\ell-1}
\\
\Lambda_2+\Lambda_{h-1}+\Lambda_h+\Lambda_\ell
\ar@/_3cm/[ddd]|{(h-1,h)}
&\Lambda_2+\Lambda_{h-1}+\Lambda_{h+1}+\Lambda_{\ell-1}
&\Lambda_0+\Lambda_{h-1}+\Lambda_{h+2}+\Lambda_\ell
\ar[u]|{(\ell,0)}
\\
\boxed{\Lambda_0+\Lambda_1+\Lambda_{h-1}+\Lambda_h}_0
\ar[u]|{(0,1)}\ar[d]|{(h-1,h)}
&\boxed{\Lambda_1+\Lambda_{h-1}+\Lambda_{h+1}+\Lambda_\ell}_5
\ar[u]|{(\ell,1)}\ar[d]|{(h-1,h+1)}\ar[r]|{(1,h-1)}\ar[l]|{(h+1,\ell)}
\ar[ur]|{(1,h+1)}\ar[uur]|{(\ell,h+1)}\ar[dr]|{(\ell,h-1)}
\ar[ul]|{(h+1,1)}\ar[dl]|{(h-1,\ell)}\ar[ddl]|{(h-1,1)}
&\boxed{\Lambda_0+\Lambda_h+\Lambda_{h+1}+\Lambda_\ell}_0
\ar[u]|{(h,h+1)}\ar[d]|{(\ell,0)}
\\
\Lambda_0+\Lambda_1+\Lambda_{h-2}+\Lambda_{h+1}
\ar[d]|{(0,1)}
&\Lambda_1+\Lambda_{h-2}+\Lambda_{h+2}+\Lambda_{\ell}
&\Lambda_1+\Lambda_h+\Lambda_{h+1}+\Lambda_{\ell-1}
\ar@/_3cm/[uuu]|{(h,h+1)}
\\
\Lambda_2+\Lambda_{h-2}+\Lambda_{h+1}+\Lambda_\ell
&\boxed{2\Lambda_0+\Lambda_{h-2}+\Lambda_{h+2}}_2
\ar[u]|{(0,0)}
&
}}
$$
Using the wildness criteria Proposition \ref{prop::step-2-result-T_0} for $T(\Lambda)_0$ and Proposition \ref{prop::step-2-result-T_2} for $T(\Lambda)_2$, 
we conclude that $R^\Lambda(\beta_{\Lambda''})$ is wild for any $\Lambda''\in S(T(\Lambda)_5)_\rightarrow$.
\end{proof}

\begin{Lemma}\label{lem::step-3-level4-2ijp}
Claim \ref{claim::step-3} holds for all $\Lambda'\in\cup_{0\leq s\leq 5}T(\Lambda)_s$ if $\Lambda=2\Lambda_p+\Lambda_i+\Lambda_j$ with distinct $p,i,j\in I$.
\end{Lemma}
\begin{proof}
We assume $\Lambda=2\Lambda_0+\Lambda_i+\Lambda_j$ with $0<i<j\leq \ell$.
Then, $T(\Lambda)_s=\emptyset$ for $3\leq s\leq 5$,
$$
T'(\Lambda)_0=\{\Lambda_{0,i}, \Lambda_{i,j}, \Lambda_{j,0}\}
\subseteq
T(\Lambda)_0=\{\Lambda_{0,i}, \Lambda_{0,j}, \Lambda_{i,0}, \Lambda_{i,j}, \Lambda_{j,0}, \Lambda_{j,i}\},
$$
where $\Lambda_{0,j}$ (resp. $\Lambda_{j,i}$, $\Lambda_{i,0}$) appears only if $j<\ell$ (resp. $i+1<j$, $i>1$).
It is proved in Lemma \ref{lem::step-2-T_0-exception} that $R^\Lambda(\beta_{\Lambda'})$ is wild for $\Lambda'\in T(\Lambda)_0\setminus T'(\Lambda)_0$.
Thus, it suffices to consider $S(\Lambda')_\rightarrow$ for $\Lambda'\in T'(\Lambda)_0, T(\Lambda)_1, T(\Lambda)_2$.

(1) Set $\Lambda'=\Lambda_{0,i}=\Lambda_0+\Lambda_{i+1}+\Lambda_j+\Lambda_\ell \in T'(\Lambda)_0$.
To prove the wildness of $R^\Lambda(\beta_{\Lambda''})$ for $\Lambda''\in S(\Lambda')_\rightarrow$, we embed $\vec C(2\Lambda_0+\Lambda_i)$
and $\vec C(\Lambda_0+\Lambda_i+\Lambda_j)$
into $\vec C(\Lambda)$, and then use Lemma \ref{lem::step-3-level3-2ij} and Lemma \ref{lem::step-3-level3-ijp} for wildness criteria.
In fact, the embedding of $\vec C(2\Lambda_0+\Lambda_i)$ and $\vec C(\Lambda_0+\Lambda_i+\Lambda_j)$ gives all but  two elements of $S(\Lambda')_\rightarrow$.
The extra two arrows are given by $(0,j)$ and $(j,0)$. We have
$$
\scalebox{0.8}{\xymatrix@C=1.4cm@R=1.2cm{
\boxed{\Lambda_{j,i}=2\Lambda_{0}+\Lambda_{i+1}+\Lambda_{j-1}}_0
\ar[d]|{(0,0)}
&\boxed{\Lambda'}_0
\ar[l]|-{(j,\ell)}
\ar[ld]|{(j,0)}
\ar[rd]|{(0,j)}
\ar[r]|-{(i,j)}
&
\boxed{\Lambda_{0,j}=\Lambda_{0}+\Lambda_{i}+\Lambda_{j+1}+\Lambda_{\ell}}_0
\ar[d]|{(0,i)}
\\
\Lambda_1+\Lambda_{i+1}+\Lambda_{j-1}+\Lambda_\ell
&
&\Lambda_{i+1}+\Lambda_{j+1}+2\Lambda_{\ell}
}}
$$
where $\Lambda'_{0,j}$ (resp. $\Lambda'_{j,0}$) appears only if $j<\ell$ (resp. $i+1<j$). 
Then, Proposition \ref{prop::step-2-result-T_0} implies that both $R^{\Lambda}(\beta_{\Lambda'_{0,j}})$ and $R^{\Lambda}(\beta_{\Lambda'_{j,0}})$ are wild by Lemma \ref{lem::reduction-arrow} and Lemma \ref{lem::step-2-T_0-exception}. We omit the details for the cases $\Lambda'=\Lambda_{i,j}, \Lambda_{j,0}\in T'(\Lambda)_0$.

(2) $T(\Lambda)_1=\{\Lambda':=\Lambda_1+\Lambda_i+\Lambda_j+\Lambda_\ell\}$.
After embedding $\vec C(2\Lambda_0+\Lambda_i)$ and $\vec C(2\Lambda_0+\Lambda_j)$ into $\vec C(2\Lambda_0+\Lambda_i+\Lambda_j)$, we only need to consider two extra arrows: $(i,j)$ and $(j,i)$, that is,
$$
\Lambda'_{i,j}=\Lambda_1+\Lambda_{i-1}+\Lambda_{j+1}+\Lambda_\ell,
\quad
\Lambda'_{j,i}=\Lambda_1+\Lambda_{i+1}+\Lambda_{j-1}+\Lambda_\ell,
$$
where $\Lambda'_{i,j}$ (resp. $\Lambda'_{j,i}$) appears only if $i>1$ or $j<\ell$ (resp. $i+1<j$). Since we have proved that $R^\Lambda(\beta_{\Lambda_{j,i}})$ is wild in (1) and $\Lambda'_{j,i}\in S(\Lambda_{j,i})_\rightarrow$, $R^{\Lambda}(\beta_{\Lambda'_{j,i}})$ is wild.

We show that $R^{\Lambda}(\beta_{\Lambda'_{i,j}})$ is wild.
If $i=1, j<\ell$ or $i>1, j=\ell$, then $\Lambda'_{i,j}=\Lambda_{0,j}$ or $ \Lambda'_{i,j}=\Lambda_{i,0}$, so that Lemma \ref{lem::step-2-T_0-exception} implies that $R^{\Lambda}(\beta_{\Lambda'_{i,j}})$ is wild because $\Lambda_{0,j}, \Lambda_{i,0} \in T(\Lambda)_0\setminus T'(\Lambda)_0$.
Suppose $i>1$ and  $j<\ell$. Since $\beta_{\Lambda'_{i,j}}=\alpha_0+\alpha_i+\alpha_{i+1}+\ldots +\alpha_j$, we choose $e_1=e(i~i+1~\ldots~j~0)$ and $e_2=e(i~i+1~\ldots~j-2~j~j-1~0)$. Set $\mathcal{A}=(e_1+e_2)R^\Lambda(\beta_{\Lambda'_{i,j}})(e_1+e_2)$. Then,
$$
\dim_q e_1\mathcal{A}e_1=\dim_q e_2Ae_2=1+2q^2+q^4,
\quad
\dim_q e_1\mathcal{A}e_2=\dim_q e_2Ae_1=q+q^3.
$$
On the other hand, we have $x_se_t=0, 1\leq  s\leq j-i$ and $x_{j-i+1}^2e_t =x^2_{j-i+2}e_t=0$ for $t=1,2$.
This implies that $e_t\mathcal{A}e_t$ has a basis $\{ x^a_{j-i+1}x^b_{j-i+2}e_t\mid 0\le a,b \le 1\}$, $e_1\mathcal{A}e_2$ has a basis $\{\psi_{j-i}e_2, x_{j-i+2}\psi_{j-i}e_2\}$ and $e_2\mathcal{A}e_1$ has a basis $\{\psi_{j-i}e_1, x_{j-i+2}\psi_{j-i}e_1\}$.
By setting $\alpha=x_{j-i+2}e_1$, $\beta=x_{j-i+2}e_2$, $\mu=\psi_{j-i}e_2$ and $\nu=\psi_{j-i}e_1$, $\mathcal{A}$ is isomorphic to the bound quiver algebra defined by
$$
\xymatrix@C=0.8cm{1 \ar@<0.5ex>[r]^{\mu}\ar@(dl,ul)^{\alpha}&2 \ar@<0.5ex>[l]^{\nu}\ar@(ur,dr)^{\beta}}
\quad \text{and}\quad 
\left<\alpha^2, \beta^2, \mu\nu\mu, \nu\mu\nu, \alpha\mu-\mu\beta, \beta\nu-\nu\alpha\right>.
$$
Then, $\mathcal{A}/\left<\nu\alpha\right>$ is the minimal wild algebra labeled by (32) in \cite{H-wild-two-point}.

(3) $T(\Lambda)_2=\{\Lambda':=\Lambda_2+\Lambda_i+\Lambda_j+\Lambda_{\ell-1}\}$ if $\ell\geq 3$.
Similar to the above, we only need to consider two extra elements: $$
\Lambda'_{i,j}=\Lambda_2+\Lambda_{i-1}+\Lambda_{j+1}+\Lambda_{\ell-1},
\quad \Lambda'_{j,i}=\Lambda_2+\Lambda_{i+1}+\Lambda_{j-1}+\Lambda_{\ell-1},
$$
where $\Lambda'_{i,j}$ (resp. $\Lambda'_{j,i}$) appears only if $i>2$ or $j<\ell-1$ (resp. $i+1< j$).
Observe that we have the following paths.
$$
\scalebox{0.8}{
\xymatrix@C=1.5cm@R=0.1cm{
&\Lambda_1+\Lambda_{i-1}+\Lambda_{j+1}+\Lambda_{\ell}
\ar[r]|-{(\ell,1)}
&\Lambda'_{i,j}\\
\boxed{\Lambda_1+\Lambda_i+\Lambda_j+\Lambda_\ell}_1
\ar[ur]|{(i,j)}\ar[dr]|{(j,i)}
&&\\
&\Lambda_{1}+\Lambda_{i+1}+\Lambda_{j-1}+\Lambda_{\ell}
\ar[r]|-{(\ell,1)}
&\Lambda'_{j,i}
}}
$$
Since we have proved in (2) that $R^\Lambda(\beta_{\Lambda''})$ is wild  for $\Lambda''\in\{ \Lambda_1+\Lambda_{i-1}+\Lambda_{j+1}+\Lambda_\ell, \Lambda_1+\Lambda_{i+1}+\Lambda_{j-1}+\Lambda_\ell\}$, 
we conclude that $R^\Lambda(\beta_{\Lambda''})$ for $\Lambda''\in\{ \Lambda'_{i,j},\Lambda'_{j,i}\}$ is wild by Lemma \ref{lem::reduction-arrow}.
\end{proof}

\begin{Lemma}\label{lem::step-3-level4-ijpn}
Claim \ref{claim::step-3} holds for all $\Lambda'\in \cup_{0\leq s\leq 5}T(\Lambda)_s$ if
$\Lambda= \Lambda_p+\Lambda_i+\Lambda_j+\Lambda_h$ with $0\leq p<i<j<h\leq \ell\in I$.
\end{Lemma}
\begin{proof}
We assume $\Lambda=\Lambda_0+\Lambda_i+\Lambda_j+\Lambda_h$ with $0<i<j<h\leq \ell$.
Then,
$$
T(\Lambda)_0=\{\Lambda_{0,i}, \Lambda_{0,j}, \Lambda_{0,h}, \Lambda_{i,0}, \Lambda_{i,j}, \Lambda_{i,h}, \Lambda_{j,0}, \Lambda_{j,i}, \Lambda_{j,h}, \Lambda_{h,0}, \Lambda_{h,i}, \Lambda_{h,j}\},
$$
and $T(\Lambda)_s=\emptyset$ for $s\neq 0$. 
By Lemma \ref{lem::step-2-T_0-exception}, it is enough to show that  Claim \ref{claim::step-3} holds  for $\Lambda'\in T'(\Lambda)_0=\{\Lambda_{0,i}, \Lambda_{i,j}, \Lambda_{j,h}, \Lambda_{h,0}\}$.
We show the case $\Lambda':=\Lambda_{0, i}=\Lambda_{i+1}+\Lambda_j+\Lambda_h+\Lambda_\ell$ and we may deduce the claim for the other 3 cases by Proposition \ref{prop::iso-sigma}.

We embed  $\vec C(\Lambda_0+\Lambda_i+\Lambda_h)$ and $\vec C(\Lambda_{0}+\Lambda_i+\Lambda_j)$ into $\vec C(\Lambda_0+\Lambda_i+\Lambda_j+\Lambda_h)$.
Since the label of $\Lambda\rightarrow\Lambda'$ is $(0,i)$,
$\Lambda'$ belongs to both of the embedded subquivers. 
Moreover, 
if $\xymatrix@C=1cm{
\Lambda'\ar[r]|-{(a,b)}&
\Lambda''}$
and $\{a,b\}\cap \{j\}=\emptyset$ or $\{a,b\}\cap\{ h\}=\emptyset$, then $\Lambda''$ belongs to the embedded subquivers, and Lemma \ref{lem::step-3-level3-ijp} implies that $R^\Lambda(\beta_{\Lambda''})$ is wild. Therefore, 
we only need to consider two extra arrows to find the new elements of $S(\Lambda')_\rightarrow$, i.e.,
$$
\Lambda'_{j,h}=\Lambda_{i+1}+\Lambda_{j-1}+\Lambda_{h+1}+\Lambda_\ell,
\quad
\Lambda'_{h,j}=\Lambda_{i+1}+\Lambda_{j+1}+\Lambda_{h-1}+\Lambda_\ell,
$$
where $\Lambda'_{j,h}$ (resp. $ \Lambda'_{h,j} $) appears only if $i+1<j$ or $h<\ell$ (resp. $h>j+1$).
Then, $R^\Lambda(\beta_{\Lambda'_{h,j}})$ is wild by the existence of the arrow
$$
\xymatrix@C=1.2cm@R=1cm{
\boxed{\Lambda_{h,j}=\Lambda_0+\Lambda_i+\Lambda_{j+1}+\Lambda_{h-1}}_0
\ar[r]^-{(0,i)}
&\Lambda'_{h,j}
}.
$$

We show that $R^\Lambda(\beta_{\Lambda'_{j,h}})$ is also wild. In fact, we have
$$
\Lambda'_{j,h}=\Lambda_i+\Lambda_j+\Lambda_{h+1}+\Lambda_\ell=\Lambda_{0,h} \in T(\Lambda)_0\setminus T'(\Lambda)_0 \quad \text{if } i+1=j,
$$
$$
\Lambda'_{j,h}=\Lambda_0+\Lambda_{i+1}+\Lambda_{j-1}+\Lambda_h =\Lambda_{j,i}\in T(\Lambda)_0\setminus T'(\Lambda)_0 \quad \text{if } h=\ell,
$$
so that $R^\Lambda(\beta_{\Lambda'_{j,h}})$ is wild if $i+1=j$ or $h=\ell$. If $i+1<j$ and $h<\ell$, then
$$
\beta_{\Lambda'_{j,h}}=\alpha_0+\alpha_1+\ldots+\alpha_i+\alpha_j+\alpha_{j+1}+\ldots+\alpha_{h-1}+\alpha_h.
$$
We choose $e_1=e(0~1~\ldots~i~j~j+1~\ldots~h)$, $e_2=e(0~1~\ldots~i-2~i~i-1~j~j+1~\ldots~h)$, and define $\mathcal{A}=(e_1+e_2)R^{\Lambda}(\beta_{\Lambda'_{j,h}})(e_1+e_2)$.
Then, by direct calculation we have
$$
\dim_q e_1\mathcal{A}e_1=\dim_q e_2\mathcal{A}e_2=1+2q^2+q^4,
\quad
\dim_q e_1\mathcal{A}e_2=\dim_q e_2\mathcal{A}e_1=q+q^3.
$$
Moreover, $x_se_t=0$ and $x_{h-j+i+2}^2e_t=0$ for $t=1,2$, $1\leq s\leq i$ and $i+2\leq s\leq h-j+i+1$, which imply that $e_t\mathcal{A}e_t$ has a basis $\{x^a_{i+1}x^b_{h-j+i+2}e_t\mid 0\le a, b \le 1\}$, $e_1\mathcal{A}e_2$ has a basis $\{\psi_{i}e_2, x_{h-j+i+2}\psi_ie_2\}$ and $e_2\mathcal{A}e_1$ has a basis $\{\psi_{i}e_1, x_{h-j+i+2}\psi_ie_1\}$.
Set $\alpha=x_{h-j+i+2}e_1$, $\beta=x_{h-j+i+2}e_2$, $\mu=\psi_ie_2$ and $\nu=\psi_ie_1$. Then, $\mathcal{A}$ is wild by the proof which is similar to that of Lemma \ref{lem::step-3-level4-2ijp}.
\end{proof}

\subsection{Level 5 and 6}
In this subsection, we show the cases for $k=5,6$.

\begin{Lemma}\label{lem::step-3-T_3-level5}
Claim \ref{claim::step-3} holds for any $\Lambda'\in T(\Lambda)_3$ if $\Lambda$ is of level $k=5$.
\end{Lemma}
\begin{proof}
Write $\Lambda=\sum_{i\in I(\Lambda)_0}m_i\Lambda_i$.
If $T(\Lambda)_3\neq \emptyset$, then $\ell\ge 2$ and $I(\Lambda)_2\ge 1$, i.e., there exists at least one $m_i\ge 3$.
By Remark \ref{rem::T_Lambda-beta}, we have $\mathcal T(\Lambda)_3=\{2\alpha_i+\alpha_{i-1}, 2\alpha_i+\alpha_{i+1}\}$.
If $m_i\geq 4$, then both $R^\Lambda(2\alpha_i+\alpha_{i-1})$ and $R^\Lambda(2\alpha_i+\alpha_{i+1})$ are wild by Proposition \ref{prop::step-2-result-T_3}. If $m_i=3$, we only need to consider $3\Lambda_0+2\Lambda_h$ with $0<h\leq \ell$ and $3\Lambda_0+\Lambda_i+\Lambda_j$ with $0<i<j\leq \ell$ (since $\sum_{i\neq 0}m_i=2$). 

(1) Suppose $\Lambda= 3\Lambda_0+ 2\Lambda_h$.
In this case, $T(\Lambda)_3=\{\Lambda_2+2\Lambda_h+2\Lambda_\ell, 2\Lambda_1+2\Lambda_h+\Lambda_{\ell-1}\}$.
We only show the case $\Lambda'=\Lambda_2+2\Lambda_h+ 2\Lambda_\ell$.
We embed $\vec C(3\Lambda_0+\Lambda_h)$ into $\vec C(3\Lambda_0+2\Lambda_h)$, and look at the second quiver in the proof of Lemma \ref{lem::step-3-level4-3ij}. Then, $\bar{\Lambda}':=\Lambda_2+\Lambda_h+2\Lambda_\ell$ appears in the middle of the bottom line and $R^{3\Lambda_0+\Lambda_h}(\beta_{\bar{\Lambda}'})$ is wild for all $\bar{\Lambda}''\in S(\bar {\Lambda}')_\rightarrow$. Thus, we only need to consider one extra arrow 
$$
\xymatrix@C=1.2cm@R=1cm{
\boxed{\Lambda'}_3
\ar[r]^-{(h,h)}
&\Lambda'_{h,h}=\Lambda_2+\Lambda_{h-1}+\Lambda_{h+1}+2\Lambda_\ell
}.
$$
If $3\leq h\leq \ell-1$, then one may easily check
$$
\dim_q e(010h) R^{\Lambda}(\beta_{\Lambda'_{h,h}})e(010h)=1+3q^2+4q^4+3q^6+q^8,
$$
and therefore, $R^{\Lambda}(\beta_{\Lambda'_{h,h}})$ is wild.
For the remaining cases $h=1,2,\ell$, we have
$$
\begin{aligned}
&\xymatrix@C=1.2cm@R=1cm{
\boxed{\Lambda_{0,h}=2\Lambda_0+\Lambda_1+\Lambda_2+\Lambda_\ell}_0
\ar[r]^-{(0,1)}
&\Lambda_0+2\Lambda_2+2\Lambda_{\ell}
} \quad \text{if } h=1,
\\
&\xymatrix@C=1.2cm@R=1cm{
\boxed{\Lambda_{0,h}=2\Lambda_0+\Lambda_2+\Lambda_3+\Lambda_\ell}_0
\ar[r]^-{(0,0)}
&\Lambda_1+\Lambda_2+\Lambda_3+2\Lambda_{\ell}
} \quad \text{if } h=2,
\\
&\xymatrix@C=1.2cm@R=1cm{
\boxed{\Lambda_{h,0}=2\Lambda_0+\Lambda_1+\Lambda_{\ell-1}+\Lambda_\ell}_0
\ar[r]^-{(0,1)}
&\Lambda_0+\Lambda_2+\Lambda_{\ell-1}+2\Lambda_\ell
} \quad \text{if } h=\ell,
\end{aligned}
$$
and $R^\Lambda(\beta_{\Lambda'_{h,h}})$ is also wild by Lemma \ref{lem::reduction-level}, Lemma \ref{lem::reduction-arrow} and Lemma \ref{lem::step-2-T_0-2i2j}.

(2) Suppose $\Lambda=3\Lambda_0+\Lambda_i+\Lambda_j$.
Then, $T(\Lambda)_3=\{2\Lambda_1+\Lambda_i+\Lambda_j+\Lambda_{\ell-1}, \Lambda_2+\Lambda_i+\Lambda_j +2\Lambda_\ell\}$.
We only show the case $\Lambda'=\Lambda_2+\Lambda_i+\Lambda_j +2\Lambda_\ell$.
Similar to the above, we embed  $\vec C(3\Lambda_0+\Lambda_j)$ and $\vec C(3\Lambda_0+\Lambda_i)$ into $\vec C(3\Lambda_0+\Lambda_i+\Lambda_j)$.
Since 
$\xymatrix@C=1.2cm{
\Lambda\ar[r]|-{(0,0)}&
\Lambda_{0,0}\ar[r]|-{(0,1)}&\Lambda'}$, $\Lambda'$ belongs to both of the embedded subquivers. If 
$\xymatrix@C=1cm{
\Lambda'\ar[r]|-{(a,b)}&
\Lambda''}$
and $\{a,b\}\cap\{i\}=\emptyset$ or $\{a,b\}\cap\{j\}=\emptyset$, then $\Lambda''$ is in the embedded  $\vec C(3\Lambda_0+\Lambda_j)$ or in the embedded $\vec C(3\Lambda_0+\Lambda_i)$. Therefore,    
it remains to consider
$$
\Lambda'_{i,j}=\Lambda_2+\Lambda_{i-1}+\Lambda_{j+1}+2\Lambda_\ell \quad  \text{if } i>2 \text{ or } j<\ell,
$$
and
$$
\Lambda'_{j,i}=\Lambda_2+\Lambda_{i+1}+\Lambda_{j-1}+2\Lambda_\ell \quad  \text{if } i+1<j.
$$
Observe that we have two arrows
$$
\xymatrix@C=1.2cm@R=1cm{
\Lambda_0+\Lambda_1+\Lambda_{i-1}+\Lambda_{j+1}+\Lambda_\ell
\ar[r]^-{(0,1)}
&\Lambda'_{i, j}
}, \quad
\xymatrix@C=1.2cm@R=1cm{
\Lambda_0+\Lambda_1+\Lambda_{i+1}+\Lambda_{j-1}+\Lambda_\ell
\ar[r]^-{(0,1)}
&\Lambda'_{j, i}
}.
$$
Since we showed that 
  $R^{2\Lambda_0+\Lambda_i+\Lambda_j}(\beta_{\Lambda_1+\Lambda_{i-1}+\Lambda_{j+1}+\Lambda_\ell})$ is wild in (2) of the proof of Lemma \ref{lem::step-3-level4-2ijp}, $R^{\Lambda}(\beta_{\Lambda'_{i, j}})$ is wild by Lemma \ref{lem::reduction-level} and  Lemma \ref{lem::reduction-arrow}. Similarly, $R^{\Lambda}(\beta_{\Lambda'_{j,i}})$ is also wild.
\end{proof}

\begin{Lemma}\label{lem::step-3-T_4-level56}
Claim \ref{claim::step-3} holds for any $\Lambda'\in T(\Lambda)_4$ if $\Lambda$ is of level $5$ or $6$.
\end{Lemma}
\begin{proof}
Write $\Lambda=\sum _{i\in I(\Lambda)_0}m_i\Lambda_i$.
If $T(\Lambda)_4\neq \emptyset$, then there exists at least one $m_i\ge 4$.
By Remark \ref{rem::T_Lambda-beta}, we have $\mathcal T(\Lambda)_4=\{2\alpha_i\}$.
If $m_i\ge 5$, then $R^\Lambda(2\alpha_i)$ is wild by Proposition \ref{prop::step-2-result-T_4}.
Hence, by Proposition \ref{prop::iso-sigma}, we may consider the following three cases. 

(1) Suppose $\Lambda=4\Lambda_0+\Lambda_h$ with $0<h\le \ell$.
Then, $T(\Lambda)_4=\{2\Lambda_1+\Lambda_h+2\Lambda_\ell\}$. We have
$$
\scalebox{0.8}{\xymatrix@C=2cm@R=1.2cm{
\Lambda_0+2\Lambda_1+\Lambda_{h-1}+\Lambda_\ell
\ar[d]|{(0,1)}
&\boxed{2\Lambda_1+\Lambda_h+2\Lambda_\ell}_4
\ar[r]|-{(1,h)}
\ar[ld]|{(h,1)}
\ar[d]|{ }
\ar[rd]|{(\ell, h)}
\ar[l]|-{(h, \ell)}
&\Lambda_0+\Lambda_1+\Lambda_{h+1}+2\Lambda_\ell
\ar[d]|{(\ell,0)}
\\
\Lambda_1+\Lambda_2+\Lambda_{h-1}+2\Lambda_\ell
& \ldots
&2\Lambda_1+\Lambda_{h+1}+\Lambda_{\ell-1}+\Lambda_\ell
}}
$$
in which the  middle dots in the second row is the second quiver in the proof of Lemma \ref{lem::step-3-level4-4i}
embedded to $\vec C(4\Lambda_0+\Lambda_h)$ via
 $\vec C(4\Lambda_0)\rightarrow\vec C(4\Lambda_0+\Lambda_h)$, and $\Lambda'_{1,h}$ (resp. $\Lambda'_{h,1}$, $\Lambda'_{\ell,h}$, $\Lambda'_{h,\ell}$) appears only if $h<\ell$ (resp. $h>2$, $h<\ell-1$, $h>1$ ). It is proved in Lemma \ref{lem::step-3-level4-3ij} that
$R^{3\Lambda_0+\Lambda_h}(\beta_{\Lambda_1+\Lambda_{h+1}+2\Lambda_\ell})$ and $R^{3\Lambda_0+\Lambda_h}(\beta_{2\Lambda_1+\Lambda_{h-1}+\Lambda_\ell})$ are wild, and hence, $R^\Lambda(\beta_{\Lambda''})$ is wild for any $\Lambda''\in S(2\Lambda_1+\Lambda_h+2\Lambda_\ell)_\rightarrow$ by Corollary \ref{cor::embedding-path}, Lemma \ref{lem::reduction-level} and Lemma \ref{lem::reduction-arrow}.

(2) Suppose $\Lambda=4\Lambda_0+2\Lambda_h$ with $0<h\le \ell$.
We have $T(\Lambda)_4=\{\Lambda':= 2\Lambda_1+2\Lambda_h+ 2\Lambda_\ell\}$.
After embedding the same $\vec C(4\Lambda_0+\Lambda_h)$ into $\vec C(4\Lambda_0+2\Lambda_h)$, we only need to check the wildness of
$R^\Lambda(\beta_{\Lambda'_{h,h}})$. Here, $\Lambda'_{h,h}=2\Lambda_1+\Lambda_{h-1}+\Lambda_{h+1}+2\Lambda_\ell$ and $\beta_{\Lambda'_{h,h}}=2\alpha_0+\alpha_h$.
If $h=1$ or $h=\ell$, then $\Lambda'_{h,h}\in T(\Lambda)_3$ and $R^{\Lambda}(\beta_{\Lambda'_{h,h}}) $ is wild by Proposition \ref{prop::step-2-result-T_3}.

We assume $2\le h\le \ell-1$.
If $\ch \k=2$, then $R^{\Lambda}(\beta_{\Lambda'})$ is wild by Lemma \ref{lem::step-2-T_4-4i} and hence, $R^{\Lambda}(\beta_{\Lambda'_{h,h}}) $ is wild by Lemma \ref{lem::reduction-arrow}.
In the following, we assume $\ch \k\neq 2$.
Using Definition \ref{def::cyclotomic-quiver}, we compute that $$
x_3^2e(00h)=x^2_3\psi_2^2e(00h)=\psi_2\psi_1 x_1^2e(h00)\psi_1\psi_2=0,
\quad
e(00h)\psi_2e(00h)=0.
$$
Then, $x_3^2e(00h)=0$ gives an algebra homomorphism
$$
\phi: R^{4\Lambda_0}(2\alpha_0)\otimes R^{2\Lambda_h}(\alpha_h)\longrightarrow e(00h)R^{\Lambda}(2\alpha_0+\alpha_h)e(00h)
$$
by sending $e(00)\otimes e(h)$, $x_se(00)\otimes e(h)$, $\psi_1 e(00)\otimes e(h)$,
$e(00)\otimes x_1e(h)$ to $e(00h)$, $x_se(00h)$, $\psi_1 e(00h)$, $x_3e(00h)$ for $s=1,2$, respectively.
Moreover, $\phi$ is surjective since $e(00h)R^\Lambda(2\alpha_0+\alpha_h)e(00h)$ is generated by $e(00h), x_1e(00h), x_2e(00h), x_3e(00h), \psi_1e(00h)$.
By comparing the graded dimensions computed by Theorem \ref{theo::graded-dim}, we find that
$\phi$ is actually an isomorphism.
It is shown in Lemma \ref{lem::step-2-T_4-4i} that the basic algebra of $R^{4\Lambda_0}(2\alpha_0)$ is isomorphic to $\k[X,Y]/(X^4-Y^2,XY)$.
Then, the basic algebra of $R^{4\Lambda_0}(2\alpha_0)\otimes R^{2\Lambda_h}(\alpha_h)$ is isomorphic to $\k[X,Y]/(X^4-Y^2,XY)\otimes \k[Z]/(Z^2)$, whose quiver has one vertex and three loops. Therefore, $R^{\Lambda}(2\alpha_0+\alpha_h)$ is wild.

(3) Suppose $\Lambda=4\Lambda_0+\Lambda_i+\Lambda_j$ with $0<i<j\le\ell$.
In this case, $T(\Lambda)_4=\{\Lambda':=2\Lambda_1+\Lambda_i+\Lambda_j+2\Lambda_\ell\}$.
Since 
$\xymatrix@C=1cm{
\Lambda\ar[r]|-{(0,0)}&
\Lambda_{0,0}\ar[r]|-{(0,0)}&\Lambda'}$, $\Lambda'$ is in the intersection of the image of $\vec C(4\Lambda_0+\Lambda_i)\rightarrow \vec C(\Lambda)$ and $\vec C(4\Lambda_0+\Lambda_j)\rightarrow\vec C(\Lambda)$. Hence, it suffices to consider
$$
\scalebox{0.8}{\xymatrix@C=1.5cm@R=1.2cm{
2\Lambda_0+\Lambda_1+\Lambda_{i-1}+\Lambda_{j+1}+\Lambda_\ell
\ar[d]|{(0,0)}
&\boxed{2\Lambda_1+\Lambda_i+\Lambda_j+2\Lambda_\ell}_4
\ar[ld]|{(i,j)}
\ar[d]|{ }
\ar[rd]|{(j,i)}
&2\Lambda_0+\Lambda_1+\Lambda_{i+1}+\Lambda_{j-1}+\Lambda_\ell
\ar[d]|{(0,0)}
\\
2\Lambda_1+\Lambda_{i-1}+\Lambda_{j+1}+2\Lambda_\ell
& \ldots
&2\Lambda_1+\Lambda_{i+1}+\Lambda_{j-1}+2\Lambda_\ell
}}
$$
where the dots stand for the union of 2 embedded subquivers, and some restrictions on $i,j$ are omitted.
Note that if 
$\xymatrix@C=1cm{
\Lambda'\ar[r]|-{(a,b)}&
\Lambda''}$
such that $\{a,b\}\cap\{i\}=\emptyset$ or $\{a,b\}\cap\{j\}=\emptyset$, then $\Lambda''$ belongs to one of the images and $R^\Lambda(\beta_{\Lambda''})$ is wild by part (1).
Since $R^{2\Lambda_0+\Lambda_i+\Lambda_j}(\beta_{\Lambda_1+\Lambda_{i-1}+\Lambda_{j+1}+\Lambda_\ell})$  and $R^{2\Lambda_0+\Lambda_i+\Lambda_j}(\beta_{\Lambda_1+\Lambda_{i+1}+\Lambda_{j-1}+\Lambda_\ell})$ were shown to be wild in (2) of the proof of Lemma \ref{lem::step-3-level4-2ijp}, $R^{\Lambda}(\beta_{\Lambda'_{i, j}})$ 
and $R^\Lambda(\beta_{\Lambda'_{j,i}})$ are wild by Lemma \ref{lem::reduction-arrow}.
\end{proof}

\begin{Lemma}\label{lem::step-3-T_5-level56}
Claim \ref{claim::step-3} holds for any $\Lambda'\in T(\Lambda)_5$ if $\Lambda$ is of level $k=5,6$.
\end{Lemma}
\begin{proof}
Write $\Lambda=\sum_{i\in I(\Lambda)_0}m_i\Lambda_i$.
In this case, $T(\Lambda)_5\neq \emptyset$ yields $\ell\ge 2$ and $I(\Lambda)_1\ge 2$, say, $m_i,m_j\ge 2$ for $i\neq j$ such that  $\beta_{\Lambda'}=\alpha_i+\alpha_j$.
If $m_i\ge 3$ or $m_j\ge 3$, then $R^{\Lambda}(\alpha_i+\alpha_j)$ is wild by Proposition \ref{prop::step-2-result-T_5}. It suffices to consider the following three cases.

(1) Suppose $\Lambda=2\Lambda_0+2\Lambda_s+\Lambda_p$ with $0<s<p\leq \ell$. In this case,
$$
T(\Lambda)_5=\{\Lambda':=\Lambda_1+\Lambda_{s-1}+\Lambda_{s+1}+\Lambda_p+\Lambda_\ell\}.
$$
If $s=1$, then $\Lambda'=\Lambda_{0,1}\in T(\Lambda)_0$ and $\beta_{\Lambda'}=\alpha_0+\alpha_1$. Moreover, $R^\Lambda(\alpha_0+\alpha_1)$ is wild by Proposition \ref{prop::step-2-result-T_0} (since $m_0=m_1=2$).
Suppose $s\ge 2$. We embed $\vec C(2\Lambda_0+2\Lambda_s)$ into $\vec C(2\Lambda_0+2\Lambda_s+\Lambda_p)$. 
Since $\Lambda'$ is reached from $\Lambda$ by arrows labeled by $(0,0)$ and $(s,s)$, $\Lambda'$ belongs to the embedded $\vec C(2\Lambda_0+2\Lambda_s)$. 
If 
$\xymatrix@C=1cm{
\Lambda'\ar[r]|-{(a,b)}&
\Lambda''}$
such that $\{a,b\}\cap\{ p\}=\emptyset$, then $\Lambda''$ belongs to the embedded $\vec C(2\Lambda_0+2\Lambda_s)$. Thus, it remains to consider the following quiver,
$$
\scalebox{0.8}{\xymatrix@C=1.2cm@R=1.5cm{
\Lambda_1+\Lambda_{s-2}+\Lambda_{s+1}+\Lambda_{p+1}+\Lambda_\ell
&\Lambda_0+\Lambda_{s-1}+\Lambda_{s+1}+\Lambda_{p+1}+\Lambda_\ell
\ar[r]|-{(\ell,0)}
&\Lambda_1+\Lambda_{s-1}+\Lambda_{s+1}+\Lambda_{p+1}+\Lambda_{\ell-1}
\\
\Lambda_1+\Lambda_{s-1}+\Lambda_s+\Lambda_{p+1}+\Lambda_\ell
\ar[u]|-{(s-1,s)}
&\boxed{\Lambda_1+\Lambda_{s-1}+\Lambda_{s+1}+\Lambda_p+\Lambda_\ell}_5
\ar[r]|-{(p,\ell)}\ar[l]|-{(s+1,p)}
\ar[ld]|-{(p,s+1)}
\ar[d]|-{(p,s-1)}
\ar[rd]|-{(p,1)}
\ar[ur]|-{(\ell,p)}
\ar[u]|-{(1,p)}
\ar[ul]|-{(s-1,p)}
&\Lambda_0+\Lambda_1+\Lambda_{s-1}+\Lambda_{s+1}+\Lambda_{p-1}
\ar[d]|{(0,1)}
\\
\Lambda_1+\Lambda_{s-1}+\Lambda_{s+2}+\Lambda_{p-1}+\Lambda_\ell
&\Lambda_1+\Lambda_s+\Lambda_{s+1}+\Lambda_{p-1}+\Lambda_\ell
\ar[l]|-{(s,s+1)}
&\Lambda_2+\Lambda_{s-1}+\Lambda_{s+1}+\Lambda_{p-1}+\Lambda_\ell
}}
$$
where the restrictions on $s,p$ are omitted.
We also notice that
$$
\scalebox{0.8}{
\xymatrix@C=1.2cm@R=0.1cm{
&\boxed{2\Lambda_0+\Lambda_{s+1}+\Lambda_{p-1}}_0\ar[r]|-{(0,0)}&\Lambda_1+\Lambda_{s+1}+\Lambda_{p-1}+\Lambda_{\ell}
\\
2\Lambda_0+\Lambda_s+\Lambda_p
\ar[ur]|-{(p,s)}\ar[dr]|-{(s,p)}&&\\
&\boxed{2\Lambda_0+\Lambda_{s-1}+\Lambda_{p+1}}_0\ar[r]|-{(0,0)}
&\Lambda_{1}+\Lambda_{s-1}+\Lambda_{p+1}+\Lambda_{\ell}
}}
$$
$$
\scalebox{0.8}{
\xymatrix@C=1.2cm@R=0.1cm{
& \boxed{\Lambda_1+2\Lambda_s+\Lambda_{p-1}}_0\ar[r]|-{(s,s)}&\Lambda_1+\Lambda_{s-1}+\Lambda_{s+1}+\Lambda_{p-1}
\\
\Lambda_0+2\Lambda_s+\Lambda_p\ar[ur]|-{(p,0)}\ar[dr]|-{(0,p)}&&\\ &\boxed{2\Lambda_{s}+\Lambda_{p+1}+\Lambda_\ell}_0\ar[r]|-{(s,s)}&\Lambda_{s-1}+\Lambda_{s+1}+\Lambda_{p+1}+\Lambda_{\ell}
}}
$$
and all four cases in the rightmost column are wild by Lemma \ref{lem::step-3-level4-2ijp}. 
In fact, set $\bar{\Lambda}=2\Lambda_0+\Lambda_s+\Lambda_p$ or $\Lambda_0+2\Lambda_s+\Lambda_p$.
Then, by Lemma \ref{lem::reduction-level}, we deduce that
$R^\Lambda(\beta_{\Lambda''})$ is wild for 
$$
\begin{aligned}
\Lambda''&=\Lambda_1+\Lambda_s+\Lambda_{s+1}+\Lambda_{p-1}+\Lambda_{\ell}=(\Lambda_1+\Lambda_{s+1}+\Lambda_{p-1}+\Lambda_{\ell})+\Lambda_s ,\\
\Lambda''&=\Lambda_{1}+\Lambda_{s-1}+\Lambda_s+\Lambda_{p+1}+\Lambda_{\ell}=(\Lambda_{1}+\Lambda_{s-1}+\Lambda_{p+1}+\Lambda_{\ell})+\Lambda_s,\\
\Lambda''&=\Lambda_0+\Lambda_1+\Lambda_{s-1}+\Lambda_{s+1}+\Lambda_{p-1}=(\Lambda_1+\Lambda_{s-1}+\Lambda_{s+1}+\Lambda_{p-1})+\Lambda_0,\\
\Lambda''&=\Lambda_0+\Lambda_{s-1}+\Lambda_{s+1}+\Lambda_{p+1}+\Lambda_{\ell}=(\Lambda_{s-1}+\Lambda_{s+1}+\Lambda_{p+1}+\Lambda_{\ell})+\Lambda_0.
\end{aligned}
$$
Therefore, we conclude that $R^\Lambda(\beta_{\Lambda''})$ is wild for any $\Lambda''\in S(\Lambda')_\rightarrow$ by Lemma \ref{lem::reduction-arrow}.

(2) Suppose $\Lambda=2\Lambda_0+2\Lambda_s+2\Lambda_p$ with $0<s<p\leq \ell$. In this case,
$$
T(\Lambda)_5=\left\{
\begin{matrix}
\Lambda^{(1)}:=\Lambda_1+\Lambda_{s-1}+\Lambda_{s+1}+2\Lambda_p+\Lambda_\ell,
\\
\Lambda^{(2)}:=\Lambda_1+2\Lambda_s+\Lambda_{p-1}+\Lambda_{p+1}+\Lambda_\ell,
\\
\Lambda^{(3)}:=2\Lambda_0+\Lambda_{s-1}+\Lambda_{s+1}+\Lambda_{p-1}+\Lambda_{p+1}
\end{matrix}
\right \}.
$$
We embed $\vec C(2\Lambda_0+2\Lambda_s+\Lambda_p)$, $\vec C(2\Lambda_0+\Lambda_s+2\Lambda_p)$ and $\vec C(\Lambda_0+2\Lambda_s+2\Lambda_p)$ into $\vec C(\Lambda)$. We also observe the following subquiver of $\vec C(\Lambda)$,
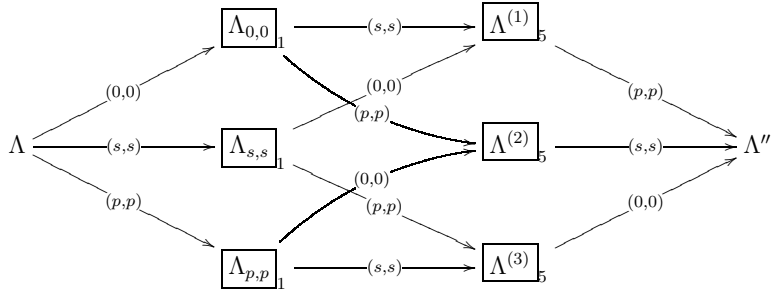
$$
\scalebox{0.8}{
\xymatrix@C=3cm@R=1cm{
& \boxed{\Lambda_{0,0}}_1\ar[r]|-{(s,s)} \ar@/_0.5cm/[dr]|{(p,p)}
& \boxed{\Lambda^{(1)}}_5\ar[dr]|-{(p,p)} &
\\
\Lambda \ar[ur]|-{(0,0)}\ar[r]|-{(s,s)}\ar[dr]|-{(p,p)}
&\boxed{\Lambda_{s,s}}_1\ar[dr]|-{(p,p)}\ar[ur]|-{(0,0)}
& \boxed{\Lambda^{(2)}}_5\ar[r]|-{(s,s)}
& \Lambda''
\\ &\boxed{\Lambda_{p,p}}_1\ar[r]|-{(s,s)}\ar@/^0.5cm/[ur]|{(0,0)}
& \boxed{\Lambda^{(3)}}_5\ar[ur]|-{(0,0)}
&
}}
$$
where $\Lambda''=\Lambda^{(1)}_{p,p}=\Lambda^{(2)}_{s,s}=\Lambda^{(3)}_{0,0}$. 
It is obvious that $\Lambda^{(1)}$, $\Lambda^{(2)}$ and $\Lambda^{(3)}$ belong to the embedded $\vec C(2\Lambda_0+2\Lambda_s+\Lambda_p)$, $\vec C(2\Lambda_0+\Lambda_s+2\Lambda_p)$ and $\vec C(\Lambda_0+2\Lambda_s+2\Lambda_p)$, respectively.
Then, for $\Lambda'=\Lambda^{(1)}$ or $\Lambda^{(2)}$ or $\Lambda^{(3)}$, $\xymatrix@C=1cm{
\Lambda'\ar[r]|-{(a,b)}&
\Lambda''}$
satisfying $(a,b)\neq (p,p)$ or $(s,s)$ or $(0,0)$ implies that $\Lambda''$ belongs to the aforementioned 3 embedded subquivers, and $R^\Lambda(\beta_{\Lambda''})$ is wild by part (1). Hence, we only need to consider
$$
\Lambda''=\Lambda_1+\Lambda_{s-1}+\Lambda_{s+1}+\Lambda_{p-1}+\Lambda_{p+1}+\Lambda_\ell.
$$
Note that $\beta_{\Lambda^{(1)}}=\alpha_0+\alpha_s$. 
If $s=1$ or $p=\ell$ or $s=p-1$, then $R^\Lambda (\beta_{\Lambda''})$ is wild by Lemma \ref{lem::reduction-arrow} and Proposition \ref{prop::step-2-result-T_5}.

Suppose $s\neq1, s\neq p-1, p\neq\ell$. Then, $\beta_{\Lambda''}=\alpha_0+\alpha_s+\alpha_p$.
We have
$$
\dim_qe(0sp)R^{\Lambda}(\beta_{\Lambda''})e(0sp)=1+3q^2+3q^4+q^6,
$$
which implies the wildness of $R^\Lambda(\beta_{\Lambda''})$.

(3) Suppose $\Lambda=2\Lambda_0+2\Lambda_s+\Lambda_i+\Lambda_j$ with $0<s\neq i<j\leq \ell$.
Then,
$$
T(\Lambda)_5=\{\Lambda'=\Lambda_1+\Lambda_{s-1}+\Lambda_{s+1}+\Lambda_i+\Lambda_j+\Lambda_\ell\}.
$$
We embed $\vec C(2\Lambda_0+2\Lambda_s+\Lambda_i)$ and $\vec C(2\Lambda_0+2\Lambda_s+\Lambda_j)$ into $\vec C(2\Lambda_0+2\Lambda_s+\Lambda_i+\Lambda_j)$.  
Similar to the above, $\Lambda'$ belongs to the intersection of the two embedded subquivers. 
If $\xymatrix@C=1cm{\Lambda'\ar[r]|-{(a,b)}&\Lambda''}$ such that $\{a,b\}\cap\{i\}=\emptyset$ or $\{a,b\}\cap\{j\}=\emptyset$, then $\Lambda''$ belongs to one of the images and $R^\Lambda(\beta_{\Lambda''})$ is wild by part (1).
Hence, it suffices to consider
$$
\scalebox{0.8}{\xymatrix@C=-0.7cm@R=1.2cm{
\Lambda_1+2\Lambda_s+\Lambda_{i-1}+\Lambda_{j+1}+\Lambda_\ell
\ar[d]|{(s,s)}
&\boxed{\Lambda_1+\Lambda_{s-1}+\Lambda_{s+1}+\Lambda_i+\Lambda_j+\Lambda_\ell}_5
\ar[ld]|{(i,j)}
\ar[d]|{ }
\ar[rd]|{(j,i)}
&\Lambda_1+2\Lambda_s+\Lambda_{i+1}+\Lambda_{j-1}+\Lambda_\ell
\ar[d]|{(s,s)}
\\
\Lambda_1+\Lambda_{s-1}+\Lambda_{s+1}+\Lambda_{i-1}+\Lambda_{j+1}+\Lambda_\ell
& \ldots
&\Lambda_1+\Lambda_{s-1}+\Lambda_{s+1}+\Lambda_{i+1}+\Lambda_{j-1}+\Lambda_\ell
}}
$$
where the dots stand for the union of the two embedded subquivers, and the restrictions on $s,i,j$ are omitted.
Since both $R^{2\Lambda_0+\Lambda_{i}+\Lambda_j}(\beta_{\Lambda_1+\Lambda_{i-1}+\Lambda_{j+1}+\Lambda_\ell})$ and $R^{2\Lambda_0+\Lambda_{i}+\Lambda_j}(\beta_{\Lambda_1+\Lambda_{i+1}+\Lambda_{j-1}+\Lambda_\ell})$ were shown to be wild in (2) of the proof of   Lemma \ref{lem::step-3-level4-2ijp}, we obtain the desired wild algebras by Lemmas \ref{lem::reduction-level}, \ref{lem::reduction-arrow}. 
\end{proof}

\section{Decomposition matrices of cyclotomic Hecke algebras}

In Section \ref{sec::step-2}, we have obtained some representation-finite and tame $R^{\Lambda}(\gamma)$'s for the representatives of $W$-orbits of $P(\Lambda)$, i.e., $\gamma\in O(\Lambda)$, which give a complete description of representation-finite and tame $R^{\Lambda}(\beta)$'s with $\beta\in Q_+$ for $k\geq 3$ up to derived equivalence. 

In this section, we are aiming to give a complete description of representation-finite and tame $R^{\Lambda}(\beta)$'s with $\beta\in Q_+$ up to Morita equivalence. 
Besides, we notice that the decomposition matrix is well-defined for a cellular algebra, and the cyclotomic Hecke algebra $H^\Lambda_n$, i.e., $R^{\Lambda}(\beta)$ with $t=-2$ if $\ell=1$ and $t=(-1)^{\ell+1}$ if $\ell\geq2$, is known to be cellular (\cite{GL-cellular-alg}). 
It is also worth mentioning that $H^\Lambda_n$ is a graded cellular algebra in the sense of Hu and Mathas, see (\cite[Main Theorem]{HM-graded-cellular}).
Therefore, as an application, we find the decomposition matrices for all representation-finite and certain tame blocks of cyclotomic Hecke algebras $H^\Lambda_n$.

\begin{rem}
Suppose $t\neq -2$ if $\ell=1$ or $t\neq(-1)^{\ell+1}$ if $\ell\geq2$. We may conjecture that the cyclotomic quiver Hecke algebra $R^{\Lambda}(\beta)$ in affine type $A$ is graded cellular. 
\end{rem}

\subsection{Representation-finite case}
By Theorem \ref{theo::main-result}, if $R^{\Lambda}(\beta)$ is representation-finite, then it is derived equivalent to one of the following two algebras:
\begin{enumerate}
\item the symmetric local algebra $\k[X]/(X^m)$, for some $m\ge 1$ (see Proposition \ref{prop::step-2-result-T_1}), 
\item the Brauer tree algebra $\mathcal{S}_0$ whose Brauer tree is a straight line without exceptional vertex (see Proposition \ref{prop::step-2-result-T_0}), that is,
$$
\begin{xy}
(0,0) *[o]+[Fo]{\hphantom{3}}="A",
(15,0) *[o]+[Fo]{\hphantom{3}}="B",
(30,0)*[o]+[Fo]{\hphantom{3}}="C",
(45,0)*[o]+[Fo]{\hphantom{3}}="D",
(60,0)*[o]+[Fo]{\hphantom{3}}="E",
\ar@{-} "A";"B"
\ar@{-} "B";"C"
\ar@{.} "C";"D"
\ar@{-} "D";"E"
\end{xy}.
$$
\end{enumerate}

We have the following result.

\begin{Prop}\label{prop::result-finite-block}
Let $\mathcal{A}$ be the basic algebra of $R^{\Lambda}(\beta)$ and $\Lambda\in \pcl, k\ge 3, \beta\in Q_+$.
If $\mathcal{A}$ is representation-finite, then $\mathcal{A}$ is isomorphic to either $\k[X]/(X^m)$ for some $m\ge 1$ or a Brauer tree algebra without exceptional vertex. 
Moreover, if $R^\Lambda(\beta)$ is cellular, the Brauer graph is a straight line.
\end{Prop}
\begin{proof}
It is mentioned in \cite[Corollary 2.13]{RZ-picard-group} that two local algebras are derived equivalent if and only if they are Morita equivalent, and hence, if and only if they are isomorphic. 
This can also be verified via silting theory (see \cite{AI-silting} for more details). 
If $\mathcal{A}$ is derived equivalent to $\mathcal{S}_0$, then the assertion follows from Theorem \ref{thm-brauer-graph-condition} and Corollary \ref{cor::tree-to-tree}.
\end{proof}

Note that the decomposition matrix of $\mathcal{S}_0$ is already mentioned in \cite[Theorem C]{Ar-tame-block} for representation-finite blocks of $H^\Lambda_n$ of level $2$.

\begin{Cor}
If $\mathcal{A}$ is a representation-finite block algebra of $H^\Lambda_n$ with $\Lambda\in \pcl, k\ge 3$, then it is either isomorphic to $\k[X]/(X^m)$ for some $m\ge 1$, or Morita equivalent to $\mathcal{S}_0$. 
In the latter case, the decomposition matrix of $\mathcal{A}$ is given by
\begin{center}
\scalebox{0.9}{$\begin{pmatrix}
1 & 0 & 0 & \ldots & 0 & 0 \\
1 & 1 & 0 & \ldots & 0 & 0 \\
0 & 1 & 1 & \ldots & 0 & 0 \\
\vdots & \vdots & \vdots &  & \vdots & \vdots \\
0 & 0 & 0 & \ldots & 1 & 1 \\
0 & 0 & 0 & \ldots & 0 & 1
\end{pmatrix}$}. 
\end{center}
\end{Cor}

\subsection{Tame case}
Let $\Lambda\in \pcl, \beta\in Q_+$ and $k\ge 3$. 
If $R^{\Lambda}(\beta)$ is tame, then it is derived equivalent to one of the following algebras:
\begin{enumerate}
    \item the local algebra $\k[X,Y]/(X^k-Y^k,XY)$, for $k\ge 3, t\neq \pm 2$ (see Lemma \ref{lem::step-2-delta-1}),
    
    \item the local algebra $\k[X,Y]/(X^3-Y^3,XY)$ with $\ch \k\neq3$ (see Proposition \ref{prop::step-2-result-T_3}),
    
    \item the local algebra $\k[X,Y]/(X^4-Y^2,XY)$ with $\ch \k\neq2$ (see Lemma \ref{lem::step-2-T_4-4i} and Proposition \ref{prop::step-2-result-T_4}), 
    
    \item the local algebra $\k[X,Y]/(X^2,Y^2)$ (see Proposition \ref{prop::step-2-result-T_5}),
    
    \item the Brauer graph algebra $\mathcal{S}_1$\footnote{This also appears when $k=2$ and $\ell\ge3$. In the statement of \cite[Theorem 1.2 (2) (b)]{Ar-tame-block},  the word $'$even$'$ must be dropped.} with $\ch \k \neq 2$ (see Lemma \ref{lem::step-2-T_2-2ij} and Proposition \ref{prop::step-2-result-T_2}) whose Brauer graph is
    $$
    \begin{xy}
    (0,0) *[o]+[Fo]{2}="A", (15,0) *[o]+[Fo]{2}="B", (30,0) *[o]+[Fo]{2}="C",
    \ar@{-} "A";"B"
    \ar@{-} "B";"C"
    \end{xy}\ ,
    $$
    
    \item the Brauer graph algebra $\mathcal{S}_2$ (see Lemma \ref{lem::step-2-delta-2}) whose Brauer graph is
    $$
    \begin{xy}
    (0,0) *[o]+[Fo]{k}="A",
    (15,0)*[o]+[Fo]{k}="B",
    (30,0)*[o]+[Fo]{k}="C",
    (45,0)*[o]+[Fo]{k}="D",
    (60,0)*[o]+[Fo]{k}="E",
    \ar@{-} "A";"B"
    \ar@{-} "B";"C"
    \ar@{.} "C";"D"
    \ar@{-} "D";"E"
    \end{xy}
    $$
    with $\ell+1$ vertices, and $t\neq (-1)^{\ell+1}$,
    
    \item the Brauer graph algebra $\mathcal{S}_3$ (see Proposition \ref{prop::step-2-result-T_0}) whose Brauer graph is
    $$
    \begin{xy}
    (0,0) *[o]+[Fo]{\hphantom{m}}="A",
    (15,0)*[o]+[Fo]{m}="B",
    (30,0)*[o]+[Fo]{m}="C",
    (45,0)*[o]+[Fo]{m}="D",
    (60,0)*[o]+[Fo]{m}="E",
    \ar@{-} "A";"B"
    \ar@{-} "B";"C"
    \ar@{.} "C";"D"
    \ar@{-} "D";"E"
    \end{xy}
    $$
    with $m=m_{i_j}$, in which the number of vertices is $(i_{j+1}-i_j+2)$ in the sense of $\Lambda=\sum_{i\in I(\Lambda)_0}m_i\Lambda_i$, $I(\Lambda)_0=\{i_j\mid i_1<i_2<\ldots<i_h\}$, $m_{i_j}\geq 2$, $m_{i_{j+1}}= 1$ and $\beta=\sum_{i\in[i_j,i_{j+1}]}\alpha_i$.
\end{enumerate}
Similar to the proof of Proposition \ref{prop::result-finite-block}, 
we see that derived equivalence implies Morita equivalence in (1)-(4).
The case (5) was proved by the first author in \cite[Proposition 5.3]{Ar-tame-block}, i.e., any finite-dimensional algebra which is derived equivalent to $\mathcal{S}_1$ is Morita equivalent to $\mathcal{S}_1$.

\begin{Prop}\label{prop:tame-block}
Let $\mathcal{A}$ be the basic algebra of $R^{\Lambda}(\beta)$ with $\Lambda\in \pcl, \beta\in Q_+$ and $k\ge 3$. 
\begin{enumerate}
    \item If $\mathcal{A}$ is derived equivalent to $\mathcal{S}_2$, then it is isomorphic to a Brauer graph algebra whose Brauer graph is a tree, and all multiplicities are $k$.
    
    \item If $\mathcal{A}$ is derived equivalent to $\mathcal{S}_3$, then it is isomorphic to a Brauer graph algebra whose Brauer graph is a tree, one multiplicity is $1$ and all the other multiplicities are $m$, where $m$ is as above.
\end{enumerate}
Moreover, if $R^\Lambda(\beta)$ is cellular, the Brauer graph is a straight line in both cases.
\end{Prop}
\begin{proof}
It is easy to see by Theorem \ref{thm-brauer-graph-condition} and Corollary \ref{cor::tree-to-tree}.
\end{proof}

Note that the cases (1) and (6) above will not happen if $R^{\Lambda}(n)\simeq H^\Lambda_n$.
If one wants to specify the Morita equivalence classes for tame blocks of cyclotomic Hecke algebra $H^\Lambda_n$, it remains to consider the algebras which are derived equivalent to $\mathcal{S}_3$. 

We introduce the following Brauer graphs. 
Set $s=i_{j+1}-i_j$ (in the sense of $\mathcal{S}_3$). 
We define $\Gamma_a^s(m)=(V,E,\mathfrak m,\mathfrak o)$ to be the Brauer graph such that $(V,E)$ is a straight line with $s+2$ vertices, the multiplicity of each vertex is $m$ except for the $a$-th vertex whose multiplicity is $1$. 
Namely, $\Gamma_a^s(m)$ is of form
$$
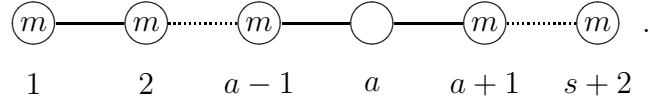
\begin{xy}
(0,0) *[o]+[Fo]{m}="A", (0,-8) *{1}="G", 
(15,0) *[o]+[Fo]{m}="B", (15,-8) *{2}="H", 
(30,0) *[o]+[Fo]{m}="F", (30,-8) *{a-1}="I", 
(45,0)*[o]+[Fo]{\hphantom{m}}="C", (45,-8) *{a}="J",
(60,0)*[o]+[Fo]{m}="D", (60,-8) *{a+1}="K",
(75,0)*[o]+[Fo]{m}="E", (75,-8) *{s+2}="L",
\ar@{-} "A";"B"
\ar@{.} "B";"F"
\ar@{-} "F";"C"
\ar@{-} "C";"D"
\ar@{.} "D";"E"
\end{xy}\ .
$$ 

Similar to the representation-finite case, we have the following corollary.
\begin{Cor}\label{prop:tame-cyc-hecke-local}
Let $\mathcal{A}$ be a tame block algebra of $H^\Lambda_n$ with $\Lambda\in \pcl, k\ge 3$. 
Then, $\mathcal{A}$ is Morita equivalent to a local algebra in (2), (3), and (4) listed above, or $\mathcal S_1$, or a Brauer graph algebra associated to $\Gamma^s_a(m)$ for $1\leq a\leq \frac{s+3}{2}$.
\end{Cor}

\begin{rem}\label{tame-cyc-hecke}
We have not yet checked whether every $\Gamma_{a}^s(m)$ for $a\geq 2$ occurs as the Brauer graph of some tame $R^\Lambda(\beta)$'s. 
\end{rem}

\subsection{Decomposition matrix of tame blocks of cyclotomic Hecke algebras}
Since we have determined the quiver presentation of tame $R^\Lambda(\beta)$'s, we are able to find the decomposition matrices for some (but infinitely many) tame block algebras $\mathcal{A}$ of $H^\Lambda_n$ via the equation $D_\mathcal{A}^tD_\mathcal{A}=C_\mathcal{A}$, where $C_\mathcal{A}$ is the Cartan matrix of $\mathcal{A}$.

Let $\mathcal{A}$ be a tame block of $H^\Lambda_n$ with $\Lambda\in \pcl, k\ge 3$. 
\begin{enumerate}
    \item If $\mathcal{A}$ is Morita equivalent to $\mathcal{S}_1$, then we have 
    \begin{center}
    $C_\mathcal{A}=$\scalebox{0.9}{$\begin{pmatrix}
    4 & 2 \\
    2 & 4 
    \end{pmatrix}$}, $\qquad$
    $D_\mathcal{A}=$\scalebox{0.9}{$\begin{pmatrix}
    1 & 0 \\
    1 & 0 \\
    1 & 1 \\
    1 & 1 \\
    0 & 1 \\
    0 & 1 
    \end{pmatrix}$},
    \end{center}
    where $D_\mathcal{A}$ is unique, up to permutations on rows. We mention that $D_\mathcal{A}$ has appeared in the case of level $2$, see \cite{Ar-tame-block}.

    \item If $\mathcal{A}$ is Morita equivalent to $\mathcal{S}_3$, then set $s=i_{j+1}-i_j$, we have 
    \begin{center}
    $C_\mathcal{A}=$\scalebox{0.9}{$\begin{pmatrix}
    m+1 & m & 0 & \ldots &0 &0 \\
    m & 2m & m & \ldots & 0&0 \\
    0 & m & 2m & \ldots &0& 0 \\
    \vdots &\vdots&\vdots & &\vdots&\vdots \\
    0 &0&0&\ldots &2m& m \\
    0 &0&0& \ldots &m& 2m
    \end{pmatrix}_{(s+1)\times (s+1)}$}.
    \end{center}
    If $m=1$, then $\mathcal{S}_3$ is isomorphic to $\mathcal{S}_0$ which is not tame. If $m=2$, then $D_\mathcal{A}=(d_{ij})_{(m(s+1)+1)\times (s+1)}$ is uniquely determined, up to permutations on rows, by 
    $$
    d_{ij}= \left\{\begin{array}{ll}
    1 & \hbox{ if } m(j-2)+1<i\leq mj+1 \hbox{ and } 2\leq j\leq s+1, \\
    1 & \hbox{ if } j=1 \hbox{ and } 0<i\leq m, \\
    0 & \hbox{ otherwise}.
    \end{array} \right.
    $$
    However, the decomposition matrix can not be determined by the equation $D_\mathcal{A}^tD_\mathcal{A}=C_\mathcal{A}$ if $m\ge 3$.
\end{enumerate}

\section*{Acknowledgements}
The first author is supported partially by JSPS KAKENHI (Grant No. 21K03163). 
The second author is supported partially by NSFC (Grant No. 12071346) and China Scholarship Council (Grant No. 202006265007). 
The third author is supported partially by National Key Research and Development Program of China (Grant No. 2020YFA0713000) and China Postdoctoral Science Foundation (Grant No. 315251 and No. 2023M731988).
The second author wishes to express his sincere gratitude to Professor Ariki, who kindly invited him to visit Osaka University and provided invaluable guidance and support during his stay which is funded by China Scholarship Council. He is also grateful to the Graduate School of Information Science and Technology for their hospitality and support during his one-year visit that began in May 2022 and during which this research was initiated.


\end{document}